\documentclass[aos,preprint]{imsart}

\RequirePackage[OT1]{fontenc}
\RequirePackage{amsthm,amsmath}
\RequirePackage[colorlinks,citecolor=blue,urlcolor=blue]{hyperref}
  
\usepackage{amsgen,amsfonts,amsmath,amsthm,amstext,amsbsy,amsopn,amssymb}
\usepackage[pdftex]{graphicx}
\usepackage{verbatim}
\usepackage{comment}
\usepackage{blkarray}
\usepackage[dvipsnames,table,dvipsnames*, svgnames*]{xcolor}
\newcommand{\red }{\color{red}}

\usepackage{epstopdf}
\usepackage{slashbox}
\usepackage{enumitem}

\usepackage{natbib} 
\usepackage{algorithm}
\usepackage{algpseudocode}

\definecolor{blue}{RGB}{000,000,200}
\definecolor{green}{RGB}{000,150,100}
\definecolor{purple}{RGB}{220,040,250}

\def\red{\color{red}}

\newtheorem{Theorem}{Theorem}

{
	\theoremstyle{definition}

	\newtheorem{Remark}{Remark}
}

\newtheorem{Lemma}{Lemma}
\newtheorem{Corollary}{Corollary}

\newtheorem{Proposition}{Proposition}

\newcommand{\be}{\begin{equation}}
\newcommand{\ee}{\end{equation}}

\newcommand{\z}{{\mathbf{z}}}

\newcommand{\D}{{\mathbf{D}}}

\newcommand{\E}{{\mathbb{E}}}
\newcommand{\G}{{\mathbf{G}}}

\renewcommand{\P}{{\mathbb{P}}}
\newcommand{\bbP}{{\mathbb{P}}}

\newcommand{\V}{{\rm V}}
\newcommand{\Var}{{\rm Var}}

\newcommand{\Cov}{{\rm Cov}}

\newcommand{\X}{{\mathbf{X}}}
\newcommand{\Y}{{\mathbf{Y}}}
\newcommand{\Z}{{\mathbf{Z}}}

\newcommand{\rank}{{\rm rank}}

\newcommand{\tr}{{\rm tr}}

\newcommand{\diag}{{\rm diag}}
\newcommand{\sgn}{{\rm sgn}}
\newcommand{\argmin}{\mathop{\rm arg\min}}
\newcommand{\argmax}{\mathop{\rm arg\max}}

\def\mybox#1{\vskip1mm \begin{center} \bf \red
		\hspace{.0\textwidth}\vbox{\hrule\hbox{\vrule\kern6pt
				\parbox{.95\textwidth}{\kern6pt#1\vskip6pt}\kern6pt\vrule}\hrule}
	\end{center} \vskip-5mm}

\arxiv{arXiv:1605.00353}

\startlocaldefs
\numberwithin{equation}{section}
\theoremstyle{plain}

\endlocaldefs

\begin{document}

\begin{frontmatter}
\title{Rate-Optimal Perturbation Bounds for Singular Subspaces with Applications to  High-Dimensional Statistics}
\runtitle{Perturbation Bounds for Singular Subspaces}
\thankstext{T1}{The research of Tony Cai was supported in part by NSF Grant DMS-1208982 and DMS-1403708, and NIH Grant R01 CA127334.}

\begin{aug}
\author{\fnms{T. Tony} \snm{Cai}\thanksref{T1}\ead[label=e1]{tcai@wharton.upenn.edu}
\ead[label=u1,url]{http://www-stat.wharton.upenn.edu/{\raise.17ex\hbox{$\scriptstyle\sim$}}tcai/}}
\and
\author{\fnms{Anru} \snm{Zhang}\thanksref{}\ead[label=e2]{anruzhang@stat.upenn.edu}\ead[label=u2,url]{http://www.stat.wisc.edu/{\raise.17ex\hbox{$\scriptstyle\sim$}}anruzhang/}}

\runauthor{T. T. Cai and A. Zhang}

\affiliation{University of Pennsylvania and University of Wisconsin-Madison}

\address{Department of Statistics\\
The Wharton School\\
University of Pennsylvania\\
Philadelphia, PA, 19104.\\
\printead{e1}\\
\printead{u1}\\
}

\address{Department of Statistics\\
University of Wisconsin-Madison\\
Madison, WI, 53706.\\
\printead{e2}\\
\printead{u2}
}
\end{aug}

\begin{abstract}
Perturbation bounds for singular spaces, in particular Wedin's $\sin \Theta$ theorem, are a fundamental tool in many  fields including high-dimensional statistics, machine learning, and applied mathematics. In this paper, we establish  separate perturbation bounds, measured in both spectral and Frobenius $\sin \Theta$ distances, for the left and right singular subspaces. Lower bounds, which show that the individual  perturbation bounds are rate-optimal, are also given.

The new perturbation bounds are applicable to a wide range of problems. In this paper, we consider in detail applications to low-rank matrix denoising and singular space estimation, high-dimensional clustering,  and canonical correlation analysis (CCA). In particular, separate matching upper and lower bounds are obtained for estimating the left and right singular spaces. To the best of our knowledge, this is the first result that gives different optimal rates for  the left and right singular spaces under the same perturbation. 
\end{abstract}

  

\end{frontmatter}

\section{Introduction}
\label{sec.intro}

Singular value decomposition (SVD) and spectral methods have been widely used in statistics, probability, machine learning, and applied mathematics as well as many applications. Examples include low-rank matrix denoising \citep{shabalin2013reconstruction,yang2014rate,donoho2014minimax}, matrix completion \citep{candes2009exact,candes2010power,keshavan2010matrix,gross2011recovering,chatterjee2014matrix}, principle component analysis \citep{anderson2003introduction,johnstone2009consistency,cai2013sparse,Cai2015spiked}, canonical correlation analysis \citep{hotelling1936relations,hardoon2004canonical,gao2014sparse,gao2015minimax}, community detection \citep{von2008consistency,rohe2011spectral,balakrishnan2011noise,lei2015consistency}. Specific applications include collaborative filtering (the Netflix problem) \citep{goldberg1992using}, multi-task learning \citep{argyriou2008convex}, system identification \citep{liu2009interior}, and sensor localization \citep{singer2010uniqueness, candes2010matrix}, among many others.
In addition, the SVD is often used to find a ``warm start" for more delicate iterative algorithms, see, e.g., \cite{cai2015optimal,sun2015guaranteed}. 

Perturbation bounds, which concern how the spectrum changes after a small perturbation to a matrix, often play a critical role in the analysis of the SVD and spectral methods. To be more specific, for an approximately low-rank matrix $X$ and a perturbation matrix $Z$, it is crucial in many applications to understand how much the left or right singular spaces of $X$ and $X+Z$ differ from each other.
This problem has been widely studied in the literature \citep{Davis,Wedin, weyl1912asymptotische, Stewart, stewart2006perturbation, Yu2014Davis-Kahan}. 
Among these results, the $\sin \Theta$ theorems, established by \cite{Davis} and \cite{Wedin}, have become fundamental tools and are commonly used in applications. While \cite{Davis} focused on eigenvectors of symmetric matrices, Wedin's $\sin \Theta$ theorem studies the more general singular vectors for asymmetric matrices and provides a uniform perturbation bound for both the left and right singular spaces in terms of the singular value gap and perturbation level. 

Several generalizations and extensions have been made in different settings after the seminal work of \cite{Wedin}. For example,  
\cite{vu2011singular}, \cite{shabalin2013reconstruction}, \cite{o2013random}, \cite{wang2015singular} considered the rotations of singular vectors after random perturbations; \cite{fan2016eigenvector} gave an $\ell_{\infty}$ eigenvector perturbation bound and used the result for robust covariance estimation. See also \cite{dopico2000note,stewart2006perturbation}.


Despite its wide applicability, Wedin's perturbation bound is not sufficiently precise for some analyses, as
the bound is uniform for both the left and right singular spaces. It clearly leads to sub-optimal result if the left and right singular spaces change in different orders of magnitude after the perturbation. 
In a range of applications, especially when the row and column dimensions of the matrix differ significantly, it is even possible that one side of the singular space can be accurately recovered, while the other side cannot. The numerical experiment given in Section \ref{comparison.sec} provides a good illustration for this point.  It can be seen from the experiment that the left and right singular perturbation bounds behave distinctly when the row and column dimensions are significantly different. Furthermore, for a range of applications, the primary interest only lies in one of the singular spaces. For example, in the analysis of bipartite network data, such as the Facebook user-public-page-subscription network, the interest is often focused on grouping the public pages (or grouping the users). This is the case for many clustering problems. See Section \ref{clustering.sec} for further discussions. 


In this paper, we establish separate perturbation bounds for the left and right singular subspaces. The bounds are measured in both the spectral and Frobenius $\sin \Theta$ distances, 
which are equivalent to several widely used losses in the literature. We also derive lower bounds that are within a constant factor of the corresponding upper bounds. These results together show that the obtained perturbation bounds are rate-optimal.

The newly established perturbation bounds are applicable to a wide range of problems in high-dimensional statistics. In this paper, we discuss in detail the applications of the perturbation bounds to the following high-dimensional statistical problems. 
\begin{enumerate}
	\item {\it Low-rank matrix denoising and singular space estimation:} Suppose one observes a low-rank matrix with random additive  noise and wishes to estimate the mean matrix or its left or right singular spaces. Such a problem arises in many applications. We apply the obtained perturbation bounds to study this problem. Separate matching upper and lower bounds are given for estimating the left and right singular spaces. These results together establish the optimal rates of convergence. Our analysis shows an interesting phenomenon that in some settings it is possible to accurately estimate the left singular space but not the right one and vice versa. To the best of our knowledge, this is the first result that gives different optimal rates for  the left and right singular spaces under the same perturbation. Another fact we observe is that in certain class of low-rank matrices, one can stably recover the original matrix if and only if one can accurately recover both its left and right singular spaces.
	
	\item {\it High-dimensional clustering:} Unsupervised learning is an important problem in statistics and machine learning with a wide range of applications.  We apply the perturbation bounds to the analysis of clustering for high-dimensional Gaussian mixtures. Particularly in a high-dimensional two-class clustering setting, we propose a simple PCA-based clustering method and use the obtained perturbation bounds to prove matching upper and lower bounds for the misclassification rates.
	
	\item {\it Canonical correlation analysis (CCA):} CCA is a commonly used tools in multivariate analysis to identify and measure the associations among two sets of random variables. The perturbation bounds are also applied to analyze CCA. Specifically, we develop sharper upper bounds for estimating the left and right canonical correlation directions. To the best of our knowledge, this is the first result that captures the phenomenon that in some settings it is possible to accurately estimate one side of canonical correlation directions but not the other side.

\end{enumerate}
In addition to these applications, the perturbation bounds can also be applied to the analysis of {\it community detection in bipartite networks,  multidimensional scaling,  cross-covariance matrix estimation, and singular space estimation for matrix completion}, and other problems to yield better results than what are known in the literature. These applications demonstrate the usefulness of the newly established perturbation bounds.

The rest of the paper is organized as follows. In Section \ref{sec.main}, after basic notation and definitions are introduced, the perturbation bounds  are presented separately for the left and right singular subspaces. Both the upper bounds and lower bounds are provided.  We then apply the newly established perturbation bounds to low-rank matrix denoising and singular space estimation, high-dimensional clustering, and canonical correlation analysis in Sections \ref{sec.denoising}--\ref{sec.CCA}. Section \ref{sec.simulation} presents some numerical results and other potential applications are briefly discussed in Section \ref{sec.discussion}. The  main theorems are proved in Section \ref{sec.proofs} and the proofs of some additional technical results are given in the supplementary materials.

\section{Rate-Optimal Perturbation Bounds for Singular Subspaces}
\label{sec.main}

We establish in this section rate-optimal perturbation bounds for singular subspaces. We begin with basic notation and definitions that will be used in the rest of the paper.

\subsection{Notation and Definitions}\label{sec.notations}

For $a, b \in \mathbb{R}$, let $a\wedge b = \min(a, b)$, $a \vee b = \max(a, b)$. Let $\mathbb{O}_{p, r} = \{V\in \mathbb{R}^{p \times r}: V^{\intercal}V = I_r\}$ be the set of all $p \times r$ orthonormal columns and write $\mathbb{O}_p$ for $\mathbb{O}_{p, p}$, the set of $p$-dimensional orthogonal matrices. For a matrix $A\in \mathbb{R}^{p_1\times p_2}$, write the SVD as $A = U\Sigma V^{\intercal}$, where $\Sigma={\rm diag}\{\sigma_1(A),  \sigma_2(A), \cdots\}$ with the singular values $\sigma_1(A)\geq  \sigma_2(A)\geq \cdots \geq 0$ in descending order. In particular, we use $\sigma_{\min}(A) = \sigma_{\min(p_1, p_2)}(A), \sigma_{\max}(A) = \sigma_1(A)$ as the smallest and largest non-trivial singular values of $A$. Several matrix norms will be used in the paper: $\|A\| = \sigma_1(A)$ is the spectral norm; $\|A\|_F = \sqrt{\sum_{i}\sigma_i^2(A)}$ is the Frobenius norm; and $\|A\|_\ast = \sum_i \sigma_i(A)$ is the nuclear norm. We denote $\bbP_{A}\in \mathbb{R}^{p_1\times p_1}$ as the projection operator onto the column space of $A$, which can be  written as $\bbP_A = A(A^{\intercal}A)^{\dagger}A^{\intercal}$. Here $(\cdot)^\dagger$ represents the Moore-Penrose pseudoinverse. Given the SVD $A = U\Sigma V^{\intercal}$ with $\Sigma$ non-singular, a simpler form for $\bbP_A$ is $\bbP_A = UU^{\intercal}$. We adopt the R convention to denote the submatrix: $A_{[a:b, c:d]}$ represents the $a$-to-$b$-th row, $c$-to-$d$-th column of matrix $A$; we also use $A_{[a:b, :]}$ and $A_{[:, c:d]}$ to represent $a$-to-$b$-th full rows of $A$ and $c$-to-$d$-th full columns of $A$, respectively. We use $C, C_0, c, c_0, \cdots$  to denote generic constants, whose actual values may vary from time to time. 

We use the {\it sin}$\mathit{\Theta}$ {\it distance} to measure the difference between two $p\times r$ orthogonal columns $V$ and $\hat{V}$. Suppose the singular values of $V^{\intercal}\hat {V}$ are $\sigma_1 \geq \sigma_2 \geq\cdots \geq \sigma_r\geq 0$. Then we call 
$$\Theta(V, \hat V) = \diag(\cos^{-1}(\sigma_1), \cos^{-1}(\sigma_2), \cdots, \cos^{-1}(\sigma_r))$$ 
as the principle angles. A quantitative measure of distance between the column spaces of $V$ and $\hat V$ is then $\|\sin\Theta(\hat V, V)\|$ or $\|\sin\Theta(\hat V, V)\|_F$. Some more convenient characterizations and properties of the $\sin \Theta$ distances will be given in Lemma \ref{lm:sin_Theta_distance} in Section \ref{sec.proof.perturbationBD}.

\subsection{Perturbation Upper Bounds and Lower Bounds}
\label{sec.perturbation}

We are now ready to present the perturbation bounds for the singular subspaces. Let $X\in \mathbb{R}^{p_1 \times p_2}$ be an approximately low-rank matrix and let $Z\in \mathbb{R}^{p_1 \times p_2}$ be a ``small" perturbation matrix. Our goal is to provide separate and rate-sharp bounds for the $\sin \Theta$ distances between the left singular subspaces of $X$ and $X+Z$ and between the right singular subspaces of $X$ and $X+Z$. 

Suppose $X$ is approximately rank-$r$ with the SVD $X = U \Sigma V^{\intercal}$, where a significant gap exists between $ \sigma_r(X)$ and $ \sigma_{r+1}(X)$. The leading $r$ left and right singular vectors of $X$ are of particular interest. We decompose $X$ as follows,
\begin{equation}\label{eq:X_decomopsition}
X = \begin{bmatrix}
U & U_{\perp}\\
\end{bmatrix}\cdot\begin{bmatrix}
\Sigma_1 & 0 \\
0 & \Sigma_2\\
\end{bmatrix}\cdot\begin{bmatrix}
V^{\intercal} \\
V_{\perp}^{\intercal} \\
\end{bmatrix},
\end{equation}
where $U\in \mathbb{O}_{p_1, r}, V\in \mathbb{O}_{p_2, r}$, $\Sigma_1 = \diag(\sigma_1(X),\cdots, \sigma_r(X))\in \mathbb{R}^{r\times r}, \Sigma_2 = \diag(\sigma_{r+1}(X), \cdots) \in \mathbb{R}^{(p_1-r)\times (p_2-r)}$, $[U ~ U_{\perp}] \in \mathbb{O}_{p_1}, [V ~ V_{\perp}]\in \mathbb{O}_{p_2}$ are orthogonal matrices. 

Let $Z$ be a perturbation matrix and let $\hat X = X+Z$. Partition the SVD of  $\hat X$ in the same way as in \eqref{eq:X_decomopsition},
\begin{equation}\label{eq:X_hatX}
\hat X = X  + Z = \begin{bmatrix}
\hat U & \hat U_{\perp}\\
\end{bmatrix}\cdot\begin{bmatrix}
\hat\Sigma_1 & 0 \\
0 & \hat\Sigma_2\\
\end{bmatrix}\cdot
\begin{bmatrix}
\hat V^{\intercal} \\
\hat V_{\perp}^{\intercal} \\
\end{bmatrix},
\end{equation}
while $\hat{U}, \hat{U}_{\perp}, \hat{\Sigma}_1, \hat{\Sigma}_2, \hat{V}$ and $\hat{V}_{\perp}$ have the same structures as ${U}, {U}_{\perp}, {\Sigma}_1, {\Sigma}_2, {V}$, and  ${V}_{\perp}$. Decompose the perturbation $Z$ into four blocks
\begin{equation}\label{eq:Z_decomposition}
Z= Z_{11} + Z_{12} + Z_{21} + Z_{22},
\end{equation}
where
\[
Z_{11} = \bbP_{U}Z\bbP_{V},\quad Z_{21} = \bbP_{U_\perp}Z\bbP_{V},\quad Z_{12} = \bbP_{U}Z\bbP_{V_\perp},\quad Z_{22} = \bbP_{U_\perp}Z\bbP_{V_\perp}.
\]

Define
\[
z_{ij} := \|Z_{ij}\| \quad\mbox{for}\quad i,\; j = 1, 2. 
\]

Theorem \ref{th:perturbation} below provides separate perturbation bounds for the left and right singular subspaces in terms of both spectral and Frobenius $\sin \Theta$ distances.

\begin{Theorem}[Perturbation bounds for singular subspaces]
	\label{th:perturbation}
	
	Let $X$, $\hat X$, and $Z$ be given as \eqref{eq:X_decomopsition}-\eqref{eq:Z_decomposition}. Denote 
	$$\alpha := \sigma_{\min}(U^{\intercal}\hat X V)  \quad\mbox{and}\quad \beta := \|U_{\perp}^{\intercal} \hat X V_{\perp}\|.$$ 
	If $\alpha^2 > \beta^2 + z_{12}^2\wedge z_{21}^2$, then
	\begin{equation}\label{ineq:th-main_spectral}
	\begin{split}
	\|\sin \Theta(V, \hat{V})\| & \leq \frac{\alpha z_{12} + \beta z_{21}}{\alpha^2 - \beta^2 - z_{21}^2\wedge z_{12}^2} \wedge 1,\\
	\|\sin \Theta(V, \hat{V})\|_F & \leq \frac{\alpha\|Z_{12}\|_F + \beta\|Z_{21}\|_F}{\alpha^2 - \beta^2 - z_{21}^2\wedge z_{12}^2} \wedge \sqrt{r}.
	\end{split}
	\end{equation}
	\begin{equation}
	\begin{split}
	\|\sin \Theta(U, \hat{U})\| & \leq \frac{\alpha z_{21} + \beta z_{12}}{\alpha^2 - \beta^2 - z_{21}^2\wedge z_{12}^2} \wedge 1,\\ 
	\|\sin \Theta(U, \hat{U})\|_F & \leq \frac{\alpha\|Z_{21}\|_F + \beta\|Z_{12}\|_F}{\alpha^2 - \beta^2 - z_{21}^2\wedge z_{12}^2} \wedge \sqrt{r}.
	\end{split}
	\end{equation}
\end{Theorem}
One can see the respective effects of the perturbation on the left and right singular spaces. In particularly, if $z_{12} \geq z_{21}$ (which is typically the case when $p_2 \gg p_1$), then Theorem \ref{th:denoising} gives a smaller bound for $\|\sin\Theta(U, \hat{U})\|$ than for $\|\sin\Theta(V, \hat{V})\|$.

\begin{Remark}
	The assumption $\alpha^2 > \beta^2 + z_{12}^2 \wedge z_{21}^2$ in Theorem \ref{th:perturbation} ensures that  the amplitude of $U^\intercal\hat{X} V = \Sigma_1 + U^\intercal Z V$ dominates those of $U_{\perp}^\intercal\hat{X}V_{\perp} = \Sigma_2 + U_{\perp}^\intercal ZV_{\perp}$, $U^\intercal ZV_{\perp}$, and $ U_{\perp}^\intercal ZV$, so that $\hat{U}$ and $\hat{V}$ can be close to $U$ and $V$, respectively. This assumption essentially means that there exists significant gap between the $r$-th and $(r+1)$-st singular values of $X$ and the perturbation term $Z$ is bounded. We will show in Theorem \ref{th:perturbation_lower} that $\hat{U}$ and $\hat{V}$ might be inconsistent when this condition fails to hold.
\end{Remark}

\begin{Remark}\label{rm:illustration_scenario}
	Consider the setting where $X\in \mathbb{R}^{p_1 \times p_2}$ is a fixed rank-$r$ matrix with $r\le p_1 \ll p_2$, and $Z\in \mathbb{R}^{p_1 \times p_2}$ is a random matrix with i.i.d. standard normal entries. In this case, $Z_{11}, Z_{12}, Z_{21},$ and $Z_{22}$ are all i.i.d. standard normal matrices of dimensions $r\times r$, $r\times (p_2-r)$, $(p_1 - r)\times r$, and $(p_1-r)\times (p_2-r)$, respectively. By random matrix theory (see, e.g. \cite{vershynin2012introduction,tao2012topics}), $\alpha \geq \sigma_r(X) - \|Z_{11}\| \geq \sigma_r(X) - C(\sqrt{p_1} + \sqrt{p_2}) $, $\beta \leq C(\sqrt{p_1}+\sqrt{p_2})$, $z_{12} \leq C\sqrt{p_2},$ and $ z_{21} \leq C\sqrt{p_1}$ for some constant $C>0$ with high probability. When $\sigma_r(X) \geq C_{\rm gap}p_2/\sqrt{p_1}$ for some large constant $C_{\rm gap}$, Theorem \ref{th:denoising} immediately implies
	\begin{equation*}
	\|\sin\Theta(V, \hat{V})\| \leq \frac{C\sqrt{p_2}}{\sigma_r(X)}, \quad \|\sin\Theta(U, \hat{U})\| \leq \frac{C\sqrt{p_1}}{\sigma_r(X)}.
	\end{equation*}	
	Further discussions on perturbation bounds for general sub-Gaussian perturbation matrix with matching lower bounds with be given in Section \ref{sec.denoising}.
\end{Remark}

Theorem  \ref{th:perturbation} gives upper bounds for the perturbation effects. We now establish lower bounds for the differences as measured by the $\sin \Theta$ distances. Theorem \ref{th:perturbation_lower} first states that $\hat{U}$ and $\hat{V}$ might be inconsistent when the condition $\alpha^2 > \beta^2 + z_{12}^2\wedge z_{21}^2$ fails to hold, and then provides the lower bounds that match those in \eqref{ineq:lm_V_W_spectral_norm} and \eqref{ineq:lm_V_W_Frobenius}, proving that the results given in Theorem \ref{th:perturbation} is essentially sharp. Theorem \ref{th:perturbation_lower} also provides the worst-case  matrix pair  $(X, Z)$  that nearly achieves the supremum in \eqref{ineq:perturbation_lower_spectral} and \eqref{ineq:perturbation_lower_Frobenius}. The matrix pair shows where the lower bound is ``close" to the upper bound, which is useful in understanding the fundamentals about singular subspace perturbations.  

Before stating the lower bounds, we define the following class of  $(X, Z)$ pairs of $p_1\times p_2$ matrices and perturbations,
	\begin{equation}
	\begin{split}
	&\mathcal{F}_{r, \alpha, \beta, z_{21}, z_{12}}
	= \big\{ (X, Z):  \text{$\hat{X}, U, V$ are given as \eqref{eq:X_decomopsition} and \eqref{eq:X_hatX}}, \\
	& \quad \sigma_{\min}(U^{\intercal}\hat X V) \geq \alpha,  \|U_{\perp}^{\intercal} \hat{X} V_{\perp} \| \leq \beta, \|Z_{12}\| \leq z_{12}, \|Z_{21}\| \leq z_{21}\big\}.
	\end{split}
	\end{equation}
	In addition, we also define
	\begin{equation}\label{ineq:perturbation_lower_Frobenius}
	\begin{split}
	\mathcal{G}_{\alpha, \beta, z_{21}, z_{12}, \tilde{z}_{21}, \tilde{z}_{12}} = \big\{(X, Z): & \|Z_{21}\|_F\leq \tilde{z}_{21}, \|Z_{12}\|_F\leq \tilde{z}_{12},\\
	& (X, Z)\in \mathcal{F}_{r, \alpha, \beta, z_{21}, z_{12}}\big\}.
	\end{split}
	\end{equation}

\begin{Theorem}[Perturbation Lower Bound]\label{th:perturbation_lower}
If $\alpha^2 \leq \beta^2 + z_{12}^2\wedge z_{21}^2$ and $r \leq \frac{p_1\wedge p_2}{2}$, then
	\begin{equation}\label{ineq:perturbation_lower_spectral_1}
	\inf_{\tilde V}\sup_{(X, Z)\in \mathcal{F}} \|\sin\Theta (V, \tilde V)\| \geq \frac{1}{2\sqrt{2}}.
	\end{equation}
	\begin{itemize}[leftmargin=*]
	\item Provided that $\alpha^2 > \beta^2 + z_{12}^2+ z_{21}^2$, $r \leq \frac{p_1\wedge p_2}{2}$ we have the following lower bound for all estimate $\tilde V\in O_{p_2\times r}$ based on the observations $\hat X$,
	\begin{equation}\label{ineq:perturbation_lower_spectral}
	\inf_{\tilde V}\sup_{(X, Z)\in \mathcal{F}} \|\sin\Theta (V, \tilde V)\| \geq \frac{1}{8\sqrt{10}}\left(\frac{\alpha z_{12} + \beta z_{21}}{\alpha^2 - \beta^2 - z_{12}^2\wedge z_{21}^2} \wedge 1\right).
	\end{equation}
	In particular, if $X = \alpha UV^\intercal + \beta U_{\perp}V_{\perp}^\intercal$ and $Z = z_{12} UV_{\perp}^\intercal + z_{21}U_{\perp} V^\intercal$, then $(X, Z)\in \mathcal{F}$ and
	\begin{equation*}
	\frac{1}{\sqrt{10}}\left(\frac{\alpha z_{12} + \beta z_{21}}{\alpha^2 - \beta^2 - z_{21}^2\wedge z_{12}^2} \wedge 1\right) \leq \|\sin \Theta(V, \hat{V})\| \leq \left(\frac{\alpha z_{12} + \beta z_{21}}{\alpha^2 - \beta^2 - z_{21}^2\wedge z_{12}^2} \wedge 1\right),
	\end{equation*}
	when $\hat{U}, \hat{V}$ are the leading $r$ left and right singular vectors of $\hat{X}$.

\item	Provided that $\alpha^2 > \beta^2 + z_{12}^2+ z_{21}^2$, $\tilde{z}^2_{21} \leq rz^2_{21}$, $\tilde{z}^2_{12} \leq rz^2_{12}$, $r\leq \frac{p_1\wedge p_2}{2}$, we have the following lower bound for all estimator $\tilde V_1 \in \mathbb{O}_{p_2\times r}$ based on the observations $\hat X$,
	\begin{equation}
	\inf_{\tilde{V}_1}\sup_{(X, Z)\in \mathcal{G}} \|\sin\Theta (V, \tilde V)\|_F \geq \frac{1}{8\sqrt{10}}\left(\frac{\alpha \tilde{z}_{12} + \beta \tilde{z}_{21}}{\alpha^2 - \beta^2 - z_{12}^2\wedge z_{21}^2} \wedge \sqrt{r}\right).
	\end{equation}
	In particular, if $X = \alpha UV^\intercal + \beta U_{\perp}V_{\perp}^\intercal$, $Z = \tilde{z}_{12} UV_{\perp}^\intercal + \tilde{z}_{21}U_{\perp} V^\intercal$, then $(X, Z) \in \mathcal{G}$ and
	$$\frac{1}{\sqrt{10}} \left(\frac{\alpha \tilde{z}_{12} + \beta \tilde{z}_{21}}{\alpha^2 - \beta^2 - z_{21}^2\wedge z_{12}^2} \wedge \sqrt{r} \right) \leq \|\sin\Theta(V, \hat{V})\| \leq \left(\frac{\alpha \tilde{z}_{12} + \beta \tilde{z}_{21}}{\alpha^2 - \beta^2 - z_{21}^2\wedge z_{12}^2} \wedge \sqrt{r}\right),$$
	where $\hat{U}, \hat{V}$ are respectively the leading $r$ left and right singular vectors of $\hat{X}$.
\end{itemize}
\end{Theorem}

%
%
%

The following Proposition \ref{lm:perturbation}, which provides upper bounds for the $\sin \Theta$ distances between leading singular vectors of a matrix $A$ and arbitrary orthogonal columns $W$, can be viewed as another version of Theorem \ref{th:perturbation}. For some applications,  applying Proposition \ref{lm:perturbation} might be more convenient than using Theorem \ref{th:perturbation} directly.
\begin{Proposition}\label{lm:perturbation}
	Suppose $A\in \mathbb{R}^{p_1\times p_2}$, $ \tilde{V} = [V ~ V_\perp]\in\mathbb{O}_{p_2}$ are right singular vectors of $A$, $V\in \mathbb{O}_{p_2, r}$, $V_\perp \in \mathbb{O}_{p_2, p_2-r}$ correspond to the first $r$ and last $(p_2-r)$ singular vectors respectively. $\tilde{W} = [W ~ W_\perp]\in\mathbb{O}_{p_2}$ is any orthogonal matrix with $W\in \mathbb{O}_{p_2, r}, W_\perp \in \mathbb{O}_{p_2, p_2-r}$. Given that $\sigma_r (AW) > \sigma_{r+1}(A)$, we have
	\begin{equation}\label{ineq:lm_V_W_spectral_norm}
	\left\|\sin\Theta(V, W)\right\| \leq \frac{\sigma_{r}(AW) \|\bbP_{(AW)} A W_\perp\|}{\sigma^2_r(AW) - \sigma^2_{r+1}(A)} \wedge 1,
	\end{equation}
	\begin{equation}\label{ineq:lm_V_W_Frobenius}
	\left\|\sin\Theta(V, W)\right\|_F \leq \frac{\sigma_{r}(AW) \|\bbP_{(AW)} A W_\perp \|_F}{\sigma^2_r(AW) - \sigma^2_{r+1}(A)} \wedge \sqrt{r}.
	\end{equation}
\end{Proposition}

	It is also of practical interest to provide perturbation bounds for a given subset of singular vectors and in particular for a given singular vector. The following Corollary \ref{cr:perturbation_single_singular_vector} provides the one-sided perturbation bound for $\hat{U}_{[:, i:j]}$ and $\hat{V}_{[:, i:j]}$ when there are significant  gaps between the $(i-1)$-st and $i$-th and between the $j$-th and $(j+1)$-st  singular values and the perturbation is bounded. Particularly when $i = j$, Corollary \ref{cr:perturbation_single_singular_vector} provides the upper bound for the perturbation of the $i$-th left and right singular vectors of $\hat{X}$, $\hat{u}_i$ and $\hat{v}_i$.
	
	\begin{Corollary}[Perturbation bounds for individual singular vectors]\label{cr:perturbation_single_singular_vector}
		Suppose $X, \hat{X}$, and $Z$ are given as \eqref{eq:X_decomopsition}-\eqref{eq:Z_decomposition}. For any $k \geq 1$, let $U_{(k)} = U_{[:, 1:k]}\in \mathbb{O}_{p_1, k}$, $V_{(k)} = V_{[:, 1:k]} \in \mathbb{O}_{p_2, k}$, and $U_{(k)\perp} \in \mathbb{O}_{p_1, p_1-k}, V_{(k)\perp} \in \mathbb{O}_{p_2, p_2-k}$ be the orthogonal complements. Denote
		\begin{equation*}
		\alpha^{(k)} = \sigma_{\min}(U_{(k)}^{\intercal}\hat{X}V_{(k)}),\quad \beta^{(k)} =  \|U^\intercal_{(k)\perp}\hat{X}V_{(k)\perp}\|,
		\end{equation*}
		\begin{equation*}
		z_{12}^{(k)} = \left\|U_{(k)}^{\intercal}ZV_{(k)\perp}\right\|, \quad z_{21}^{(k)} = \left\|U_{(k)\perp}^{\intercal}ZV_{(k)}\right\|,
		\end{equation*}
		for $k=1,\ldots, p_1\wedge p_2$. We further define $\alpha^{(0)} = \infty, \beta^{(0)} = \|\hat{X}\|, z_{12}^{(0)} = z_{21}^{(0)} = 0$. For $1\leq i \leq j \leq p_1\wedge p_2$, provided that $(\alpha^{(i-1)})^2 > (\beta^{(i-1)})^2 + (z_{12}^{(i-1)})^2 \wedge (z_{21}^{(i-1)})^2$ and $(\alpha^{(j)})^2 > (\beta^{(j)})^2 + (z_{12}^{(j)})^2 \wedge (z_{21}^{(j)})^2$,	
		we have
		\begin{equation*}
		\left\|\sin\Theta(\hat{V}_{[:, i:j]}, V_{[:, i:j]})\right\| \leq \left\{\sum_{k\in\{i-1, j\}}\left(\frac{\left(\alpha^{(k)} z_{12}^{(k)} + \beta^{(k)} z_{21}^{(k)}\right)}{(\alpha^{(k)})^2 - (\beta^{(k)})^2 - (z_{21}^{(k)})^2\wedge (z_{12}^{(k)})^2}\right)^2\right\}^{1/2} \wedge 1,
		\end{equation*}
		\begin{equation*}
		\left\|\sin\Theta(\hat{U}_{[:, i:j]}, U_{[:, i:j]})\right\| \leq \left\{\sum_{k\in\{i-1, j\}}\left(\frac{\left(\alpha^{(k)} z_{21}^{(k)} + \beta^{(k)} z_{12}^{(k)}\right)}{(\alpha^{(k)})^2 - (\beta^{(k)})^2 - (z_{21}^{(k)})^2\wedge (z_{12}^{(k)})^2}\right)^2\right\}^{1/2} \wedge 1.
		\end{equation*}
		In particular, for any integer $1\leq i \leq p_1\wedge p_2$, if $(\alpha^{(i-1)})^2 > (\beta^{(i-1)})^2 + (z_{12}^{(i-1)})^2 \wedge (z_{21}^{(i-1)})^2$ and $(\alpha^{(i)})^2 > (\beta^{(i)})^2 + (z_{12}^{(i)})^2 \wedge (z_{21}^{(i)})^2$, $u_i ,\hat{u}_i, v_i, \hat{v}_i$, i.e. the $i$-th singular vectors of $X$ and $\hat{X}$, are different by
		\begin{equation*}
		\sqrt{1 - \left(v_i^{\intercal}\hat{v}_i\right)^2} \leq \left\{\sum_{k=i-1}^{i}\left(\frac{\left(\alpha^{(k)} z_{12}^{(k)} + \beta^{(k)} z_{21}^{(k)}\right)}{(\alpha^{(k)})^2 - (\beta^{(k)})^2 - (z_{21}^{(k)})^2\wedge (z_{12}^{(k)})^2}\right)^2\right\}^{1/2} \wedge 1,
		\end{equation*}
		\begin{equation*}
		\sqrt{1 - \left(u_i^{\intercal}\hat{u}_i\right)^2} \leq \left\{\sum_{k=i-1}^{i}\left(\frac{\left(\alpha^{(k)} z_{21}^{(k)} + \beta^{(k)} z_{12}^{(k)}\right)}{(\alpha^{(k)})^2 - (\beta^{(k)})^2 - (z_{21}^{(k)})^2\wedge (z_{12}^{(k)})^2}\right)^2\right\}^{1/2} \wedge 1.
		\end{equation*}

	\end{Corollary}

\begin{Remark}
	The upper bound given in Corollary \ref{cr:perturbation_single_singular_vector} is rate-optimal over the following set of $(X, Z)$ pairs,
	\begin{equation*}
	\begin{split}
	& \mathcal{H}_{\alpha^{(i-1)}, \beta^{(i-1)}, z_{12}^{(i-1)}, z_{21}^{(i-1)}, \alpha^{(j)}, \beta^{(j)}, z_{12}^{(j)}, z_{21}^{(j)}}\\
	= & \left\{(X, Z): \begin{array}{l}
	\sigma_{\min}(U_{(k)}^{\intercal}\hat{X}V_{(k)}) \geq \alpha^{(k)}, \|U^\intercal_{(k)\perp}\hat{X}V_{(k)\perp}\| \leq \beta^{(k)},\\
	\left\|U_{(k)}^{\intercal}ZV_{(k)\perp}\right\| \leq z_{12}^{(k)},  \left\|U_{(k)\perp}^{\intercal}ZV_{(k)}\right\| \leq z_{21}^{(k)}
	\end{array}, k \in \{i-1, j\} \right\}.
	\end{split}
	\end{equation*}
	The detailed analysis can be carried out similarly to the one for Theorem \ref{th:perturbation_lower}.
\end{Remark}

\subsection{Comparisons with Wedin's sin $\mathbf{\Theta}$ Theorem}
\label{comparison.sec}

Theorems  \ref{th:perturbation} and \ref{th:perturbation_lower} together establish separate rate-optimal perturbation bounds for the left and right singular subspaces. We now compare the results with the well-known Wedin's $\sin \Theta$ Theorem, which gives uniform upper bounds for the singular subspaces on both sides.  
Specifically, using the same notation as in Section \ref{sec.perturbation}, Wedin's $\sin \Theta$ Theorem  states that
if $\sigma_{\min} (\hat\Sigma_1) - \sigma_{\max}(\Sigma_2) = \delta > 0$, then
\begin{equation*}
\max\left\{\|\sin\Theta(V, \hat V) \|, \|\sin\Theta(U, \hat U)\|\right\} \leq \frac{\max\left\{\|Z\hat V\|, \|\hat U^{\intercal}Z\|\right\}}{\delta},
\end{equation*}
\begin{equation*}
\max\left\{\|\sin\Theta(V, \hat V) \|_F, \|\sin\Theta(U, \hat U)\|_F\right\} \leq \frac{\max\left\{\|Z\hat V\|_F, \|\hat U^{\intercal}Z\|_F\right\}}{\delta}.
\end{equation*}
When $X$ are $Z$ are symmetric, Theorem \ref{th:perturbation}, Proposition \ref{lm:perturbation}, 
 and Wedin's $\sin\Theta$ theorem provide similar upper bound for singular subspace perturbation.

As mentioned in the introduction, the uniform bound on both left and right singular subspaces in Wedin's $\sin \Theta$ Theorem might be sub-optimal in some cases when $X$ or $Z$ are asymmetric. For example, in the setting discussed in Remark \ref{rm:illustration_scenario}, applying Wedin's theorem leads to 
$$\max\left\{\|\sin\Theta(V, \hat{V})\|, \|\sin\Theta(U, \hat{U})\|\right\} \leq \frac{C\max\{\sqrt{p_1}, \sqrt{p_2}\}}{\sigma_r(X)}, $$
which is sub-optimal for $\|\sin\Theta(U, \hat{U})\|$ if $p_2 \gg p_1$.

\section{Low-rank Matrix Denoising and Singular-Space Estimation}
\label{sec.denoising}

In this section, we apply the perturbation bounds given in Theorem \ref{th:perturbation} for low-rank matrix denoising. It can be seen that the new perturbation bounds are particularly powerful when the matrix dimensions differ significantly. We also establish a matching lower bound for low-rank matrix denoising which shows that the results are rate-optimal. 

As mentioned in the introduction, accurate recovery of  a low-rank matrix based on noisy observations  has a wide range of applications, including  magnetic resonance imaging (MRI) and relaxometry. See, e.g.,  \cite{candes2013unbiased,shabalin2013reconstruction} and the reference therein.  This problem is also important in the context of dimensional reduction. Suppose one observes a low-rank matrix with additive noise,
$$Y = X + Z,$$
where $X = U\Sigma V^{\intercal} \in \mathbb{R}^{p_1\times p_2}$ is a low-rank matrix with $U\in \mathbb{O}_{p_1, r}, V\in \mathbb{O}_{p_2, r}, $ and $\Sigma = \diag\{\sigma_1(X), \ldots,  \sigma_r(X)\} \in \mathbb{R}^{r\times r}$, and $Z\in \mathbb{R}^{p_1\times p_2}$ is an i.i.d. mean-zero sub-Gaussian matrix. The goal is to estimate the underlying low-rank matrix $X$ or its singular values or singular vectors. 

This problem has been actively studied. For example, \cite{bura2008distribution}, \cite{capitaine2009largest}, \cite{benaych2012singular}, \cite{shabalin2013reconstruction} focused on the asymptotic distributions of single singular value and vector when $p_1$, $p_2$ and the singular values grows proportionally. \cite{vu2011singular} discussed the squared matrix perturbed by i.i.d Bernoulli matrix and derived an upper bound on the rotation angle of singular vectors. \cite{o2013random} further generalized the results in \cite{vu2011singular} and proposed a trio-concentrated random matrix perturbation setting. Recently, \cite{wang2015singular} provides the $\ell_{\infty}$ distance under relatively complicated settings when matrix is perturbed by i.i.d. Gaussian noise. \cite{donoho2014minimax,gavish2014optimal,candes2013unbiased} studied the algorithm for recovering $X$, where singular value thresholding (SVT) and hard singular value thresholding (HSVT), stated as
\begin{equation}
\begin{split}
& SVT(Y)_\lambda = \argmin_{X} \left\{\frac{1}{2}\|Y - X\|_F^2 + \lambda \|X\|_\ast\right\},\\
& HSVT(Y)_{\lambda} = \argmin_{X} \left\{\frac{1}{2}\|Y-X\|_F^2 + \lambda \rank(X)\right\}
\end{split}
\end{equation}
were proposed. The optimal choice of thresholding level $\lambda^\ast$ was further discussed in \cite{donoho2014minimax} and \cite{gavish2014optimal}. Especially, \cite{donoho2014minimax} proves that
$$\inf_{\hat{X}} \sup_{\substack{X\in \mathbb{R}^{p_1\times p_2}\\\rank(X) \leq r}} \E\left\|\hat{X} - X\right\|_F^2 \asymp r (p_1 + p_2),$$
when $Z$ is i.i.d. standard normal random matrix. If one defines the class of rank-$r$ matrices, $\mathcal{F}_{r, t} = \left\{X\in \mathbb{R}^{p_1\times p_2}: \sigma_r(X) \geq t\right\}$, the following upper bound for the relative error is an immediate consequence of our results,
\begin{equation}\label{ineq:upper_relative_error}
\sup_{X\in \mathcal{F}_{r, t}}\E\frac{\|\hat{X} - X\|_F^2}{\|X\|_F^2} \leq  \frac{C(p_1+p_2)}{t^2}\wedge 1,
\end{equation}	
where 
$$\hat{X} = \left\{\begin{array}{ll}
SVT(Y)_{\lambda^\ast}, & \text{if } t^2 \geq C(p_1+p_2),\\
0, & \text{if } t^2 < C(p_1+p_2).
\end{array}\right.$$

In the following discussion, we assume that the entries of $Z=(Z_{ij})$ have unit variance (which can be simply achieved by normalization). To be more precise, we define the class of distributions $\mathcal{G}_{\tau}$ for some $\tau>0$ as follows.
\begin{equation}
\text{If } Z\sim \mathcal{G}_\tau, \text{ then } \E Z = 0, \Var(Z) = 1, \E\exp(tZ) \leq \exp(\tau t), ~ \forall t\in \mathbb{R}.
\end{equation}
The distribution of the entries of $Z$, $Z_{ij}$, is assumed to satisfy 
$$Z_{ij} \overset{iid}{\sim}\mathcal{G}_{\tau}, \quad 1\leq i \leq p_1, 1\leq j \leq p_2.$$

Suppose $\hat{U}$ and $\hat{V}$ are respectively the first $r$ left and right singular vectors of $Y$. We use $\hat{U}$ and $\hat{V}$ as the estimators of $U$ and $V$  respectively. Then the perturbation bounds for singular spaces 
yield the following results.

\begin{Theorem}[Upper Bound]\label{th:denoising}
	Suppose $X= U\Sigma V^{\intercal}\in \mathbb{R}^{p_1\times p_2}$ is of rank-$r$. There exists constants $C>0$ that only depends on $\tau$ such that 
	\begin{equation*}
	\begin{split}
	& \E\|\sin\Theta(V, \hat V)\|^2 \leq \frac{Cp_2(\sigma_r^2(X)+p_1)}{\sigma_r^4(X)}\wedge 1,\\
	& \E\|\sin\Theta(V, \hat V)\|_F^2 \leq \frac{Cp_2r(\sigma_r^2(X) + p_1)}{\sigma_r^4(X)}\wedge r.
	\end{split}
	\end{equation*}
	\begin{equation*}
	\begin{split}
	&\E\|\sin\Theta(U, \hat U)\|^2 \leq \frac{Cp_1(\sigma_r^2(X)+p_2)}{\sigma_r^4(X)}\wedge 1,\\
	&\E\|\sin\Theta(U, \hat U)\|_F^2 \leq \frac{Cp_1r(\sigma_r^2(X) + p_2)}{\sigma_r^4(X)}\wedge r.
	\end{split}
	\end{equation*}
\end{Theorem}
Theorem \ref{th:denoising} provides a non-trivial perturbation upper bound for $\sin\Theta(V, \hat V)$ (or $\sin\Theta(U, \hat U)$) if there exists a constant $C_{\rm gap}>0$ such that 
\[
\sigma_r^2 \geq C_{\rm gap}((p_1p_2)^{1\over 2} + p_2)
\]
(or $\sigma_r^2 \geq C_{\rm gap}((p_1p_2)^{1\over 2} + p_1)$). In contrast, Wedin's $\sin \Theta$ Theorem requires the singular value gap $\sigma_r^2(X) \geq C_{\rm gap} (p_1 + p_2) $, which shows the power of the proposed unilateral perturbation bound. 

Furthermore, the upper bounds in Theorem \ref{th:denoising} are rate-sharp in the sense that the following matching lower bounds hold. To the best of our knowledge, this is the first result that gives different optimal rates for  the left and right singular spaces under the same perturbation. 

\begin{Theorem}[Lower Bound]\label{th:denoising_lower}
	Define the following class of low-rank matrices
	\begin{equation}
	\mathcal{F}_{r, t} = \left\{X\in \mathbb{R}^{p_1\times p_2}: \sigma_r(X) \geq t\right\}.
	\end{equation}
	If $r \leq \frac{p_1}{16}\wedge \frac{p_2}{2}$, then
	\begin{equation*}\label{ineq:denoising_lower_spectral}
	\inf_{\tilde V} \sup_{X\in \mathcal{F}_{r, t}} \E\|\sin\Theta(V, \tilde V)\|^2 \geq 
	c \left(\frac{p_2(t^2+p_1)}{t^4}\wedge 1\right), 
	\end{equation*}
	\begin{equation*}\label{ineq:denoising_lower_Frobenius}
	\inf_{\tilde V} \sup_{X\in \mathcal{F}_{r, t}} \E\|\sin\Theta(V, \tilde V)\|_F^2 \geq 
	c \left(\frac{p_2r(t^2+p_1)}{t^4}\wedge r\right). 
	\end{equation*}
	\begin{equation*}
	\inf_{\tilde V} \sup_{X\in \mathcal{F}_{r, t}} \E\|\sin\Theta(U, \tilde U)\|^2 \geq 
	c \left(\frac{p_1(t^2+p_2)}{t^4}\wedge 1\right), 
	\end{equation*}
	\begin{equation*}
	\inf_{\tilde V} \sup_{X\in \mathcal{F}_{r, t}} \E\|\sin\Theta(U, \tilde U)\|_F^2 \geq 
	c \left(\frac{p_1r(t^2+p_2)}{t^4}\wedge r\right). 
	\end{equation*}
\end{Theorem}

\begin{Remark}
	Using similar technical arguments, we can also obtain the following lower bound for estimating the low-rank matrix $X$ over $\mathcal{F}_{r, t}$ under a relative error loss,
	\begin{equation}\label{eq:mathcal_F}
	\inf_{\tilde{X}} \sup_{X\in \mathcal{F}_{r, t}} \E\frac{\|\tilde{X} - X\|^2_F}{\|X\|_F^2} \geq c\left(\frac{p_1+ p_2}{t^2} \wedge 1\right).
	\end{equation}
	Combining equations \eqref{ineq:upper_relative_error} and \eqref{eq:mathcal_F} yields the minimax optimal rate for relative error in matrix denoising,
	\begin{equation*}
	\inf_{\tilde{X}} \sup_{X\in \mathcal{F}_{r, t}} \E\frac{\|\tilde{X} - X\|^2_F}{\|X\|_F^2} \asymp c\left(\frac{p_1+ p_2}{t^2} \wedge 1\right).
	\end{equation*}
	An interesting fact is that
	$$ c\left(\frac{p_1+p_2}{t^2}\wedge 1\right) \asymp c\left(\frac{p_2(t^2 + p_1)}{t^4}\wedge 1\right) + c\left(\frac{p_1(t^2 + p_2)}{t^4}\wedge 1\right), $$
	which yields directly 
	$$\inf_{\tilde U} \sup_{X\in \mathcal{F}_{r, t}} \E\|\sin\Theta(U, \tilde U)\|^2 + \inf_{\tilde V} \sup_{X\in \mathcal{F}_{r, t}} \E\|\sin\Theta(V, \tilde V)\|^2 \asymp \inf_{\tilde{X}} \sup_{X\in \mathcal{F}_{r, t}} \E\frac{\|\tilde{X} - X\|^2_F}{\|X\|_F}. $$
	Hence, for the class of $\mathcal{F}_{r,t}$, one can stably recover $X$ in relative Frobenius norm loss if and only if one can stably recover both $U$ and $V$ in spectral $\sin\Theta$ norm.
	
\end{Remark}

Another interesting aspect of Theorems \ref{th:denoising} and \ref{th:denoising_lower} is that, when $p_1\gg p_2$, $(p_1p_2)^{1\over 2} \ll t^2 \ll p_1$, there is no stable algorithm for recovery of either the left singular space $U$ or whole matrix $X$ in the sense that there exists uniform constant $c>0$ such that
\begin{equation*}
\inf_{\tilde{U}} \sup_{X\in \mathcal{F}_{r, t}} \E\left\|\sin\Theta(U, \tilde{U})\right\|^2 \geq c, \quad \inf_{\tilde{X}} \sup_{X\in \mathcal{F}_{r, t}} \E\frac{\|\tilde{X} - X\|_F^2}{\|X\|_F^2} \geq c.
\end{equation*}
In fact, for $X = tUV^{\intercal} \in \mathcal{F}_{r, t}$, if we simply apply SVT or HSVT algorithms with optimal choice of $\lambda$ as proposed in \cite{donoho2014minimax} and \cite{gavish2014optimal}, with high probability, $SVT_\lambda(\hat{X}) = HSVT_\lambda(\hat{X}) = 0.$
On the other hand, the spectral method does provide a consistent recovery of the right singular-space according to Theorem \ref{th:denoising}. 
\begin{equation*}
\E\left\|\sin\Theta(V, \hat{V})\right\|^2 \to 0.
\end{equation*}	
This phenomenon is well demonstrated by the simulation result (Table \ref{tb:simu_U_V}) provided in Section \ref{sec.intro}.

\section{High-dimensional Clustering}
\label{clustering.sec}

Unsupervised learning, or clustering, is an ubiquitous problem in statistics and machine learning \citep{hastie2009elements}. The perturbation bounds given in Theorem \ref{th:perturbation} as well as the results in Theorems \ref{th:denoising} and \ref{th:denoising_lower} have a direct implication  in high-dimensional clustering.  Suppose the locations of $n$ points, $X = [X_1\cdots X_n]\in \mathbb{R}^{p\times n}$, which lie in a certain $r$-dimensional subspace $\mathcal{S}$ in $\mathbb{R}^p$,  are observed with noise  
$$Y_{i} = X_{i} + \varepsilon_{i},  \quad i=1,\cdots, n. $$
Here $X_i\in \mathcal{S}\subseteq \mathbb{R}^{p}$ are fixed coordinates, $\varepsilon_i\in \mathbb{R}^p$ are random noises. The goal is to cluster the observations $Y$. Let the SVD of $X$ be given by $X = U\Sigma V^{\intercal}$, where $U\in \mathbb{O}_{p, r}$, $V\in \mathbb{O}_{n, r}$, and $\Sigma \in \mathbb{R}^{r\times r}$.  When $p\gg n$, applying the standard algorithms (e.g. $k$-means) directly to the coordinates $Y$ may lead to sub-optimal results with expensive computational costs due to the high-dimensionality. A better approach is to first perform dimension reduction by computing the SVD of $Y$ directly or on its random projections, and then carry out clustering based on the first $r$ right singular vectors $\hat{V}\in \mathbb{O}_{n, r}$. See, e.g. \cite{feldman2013turning} and \cite{boutsidis2015randomized}, and the reference therein. It is important to note that the left singular space $U$ are not directly used in the clustering procedure. Thus Theorem \ref{th:denoising} is more suitable for the analysis of the clustering method than Wedin's $\sin \Theta$ theorem as the method main depends on the accuracy of $\hat{V}$ as an estimate of $V$. 

Let us consider the following two-class clustering problem in more detail (see \cite{hastie2009elements,azizyan2013minimax,jin2014important,jin2015phase}).  Suppose $l_i \in \{-1, 1\}$, $i=1, ..., n$, are indicators representing the class label of the $n$-th nodes and let $\mu\in \mathbb{R}^p$ be a fixed vector. Suppose one observes $Y = [Y_1, \cdots, Y_n]$ where
\begin{equation*}
Y_i = l_i \mu + Z_i, \quad Z_i \overset{iid}{\sim} N(0, I_p), \quad 1\leq i \leq n,
\end{equation*}
where neither the labels $l_i$ nor the mean vector $\mu$ are observable. The goal is to cluster the data into two groups. The accuracy of any clustering algorithm is measured by the misclassification rate
\begin{equation}\label{eq:misclassification_rate}
\mathcal{M}(l, \hat{l}) := \frac{1}{n}\min_{\pi }\left|\left\{i: l_i \neq \pi(\hat{l}_i)\right\}\right|.
\end{equation}
Here $\pi$ is any permutations on $\{-1, 1\}$, as any permutation of the labels $\{-1, 1\}$ does not change the clustering outcome. %

In this case, $EY_i$ is either $\mu$ or $-\mu$, which lies on a straight line.  A simple PCA-based clustering method is to set 
\begin{equation}
\label{PCA-clustering}
\hat{l} = \sgn(\hat{v}), 
\end{equation}
where $\hat{v}\in \mathbb{R}^n$ is the first right singular vector of $Y$. We now apply the $\sin\Theta$ upper bound in Theorem \ref{th:denoising} to analyze  the performance guarantees of this clustering method. We are particularly interested in the high-dimensional case where $ p \geq n$. The case where $p < n$ can be handled similarly.

\begin{Theorem}\label{th:cluster_upper}
	Suppose $p \geq n$, $\pi$ is any permutation on $\{-1, 1\}$.
	When $\|\mu\|_2 \geq C_{\rm gap}(p/n)^{1\over 4}$ for some large constant $C_{\rm gap} >0$, then for some other constant $C>0$ the mis-classification rate for the PCA-based clustering method $\hat l$ given in \eqref{PCA-clustering} satisfies
	\begin{equation*}
	\E \mathcal{M}(\hat{l}, l) \leq C \frac{n\|\mu\|_2^2 + p}{n\|\mu\|_2^4}.
	\end{equation*}
\end{Theorem}
It is intuitively clear that the clustering accuracy depends on the signal strength $\|\mu\|_2$.  The stronger the signal, the easier to cluster. In particular, Theorem \ref{th:cluster_upper} requires the minimum signal strength condition $\|\mu\|_2 \geq C_{\rm gap}(p/n)^{1\over 4}$. The following lower bound result shows that this condition is necessary for consistent clustering:  When the condition $\|\mu\|_2 \geq C_{\rm gap}(p/n)^{1\over 4}$ does not hold, it is not possible to essentially do better than random guessing.

\begin{Theorem}\label{th:cluster_lower}
	Suppose $p\geq n$, there exists $c_{gap}, C_n>0$ such that if $n\geq C_n$,
	\begin{equation*}
	\inf_{\tilde{l}} \sup_{\substack{\mu: \|\mu\|_2 \leq c_{gap}(p/n)^{1\over 4}\\ l\in \{-1, 1\}^n}}  \E \mathcal{M}(\tilde{l}, l) \geq \frac{1}{8}.
	\end{equation*}
\end{Theorem}
\begin{Remark}
	\cite{azizyan2013minimax} considered a similar setting when $n\geq p$, $l_i$'s are i.i.d. Rademacher variables and derived rates of convergence for both the upper and lower bounds with a logarithmic gap between the upper and lower bounds. In contrast, with the help of the newly obtained perturbation bounds, we are able to establish the optimal misclassification rate for high-dimensional setting when $n\leq p$.
	
	Moreover, \cite{jin2014important} and \cite{jin2015phase} considered the sparse and highly structured setting, where the contrast mean vector $\mu$ is assumed to be sparse and the nonzero coordinates are all equal. Their method is based on feature selection and PCA. Our setting is close to the  ``less sparse/weak signal" case in \cite{jin2015phase}. In this case, they introduced a simple aggregation method with 
	$$\hat{l}^{(sa)} = {\rm sgn}(X\hat{\mu}), $$
	where $\hat{\mu} = \argmax_{\mu \in \{-1, 0, 1\}^{p}} \|X\mu\|_q$ for some $q>0$. 	
	The statistical limit, i.e. the necessary condition for obtaining correct labels for most of the points, is $\|\mu\|_2 > C$ in their setting, 
	which is smaller than the boundary $\|\mu\|_2 > C(p/n)^{1\over 4}$ in Theorem \ref{th:cluster_upper}. As shown in Theorem \ref{th:cluster_lower} the bound $\|\mu\|_2 > C(p/n)^{1\over 4}$ is necessary.
	The reason for this difference is that they focused on  highly structured contrast mean vector $\mu$ which only takes two values $\{0, \nu\}$. In contrast, we considered the general $\mu \in \mathbb{R}^{p}$, which leads to stronger condition and larger statistical limit. Moreover, the simple aggregation algorithm is computational difficult for a general signal $\mu$, thus the PCA-based method considered in this paper is preferred under the general dense $\mu$ setting.
\end{Remark}

\section{Canonical Correlation Analysis}
\label{sec.CCA}

In this section, we consider an application of the perturbation bounds given in Theorem \ref{th:perturbation}  to the canonical correlation analysis (CCA), which is one of the most important tools in multivariate analysis in exploring the relationship between two sets of variables \citep{hotelling1936relations, anderson2003introduction,witten2009penalized,gao2014sparse, gao2015minimax,ma2016subspace}.  Given two random vectors $X$ and $Y$ with a certain joint distribution, the CCA first looks for the pair of vectors $\alpha^{(1)}\in \mathbb{R}^{p_1}, \beta^{(2)}\in \mathbb{R}^{p_2}$ that maximize ${\rm corr}((\alpha^{(1)})^{\intercal}X, (\beta^{(1)})^{\intercal}Y)$. After obtaining the first pair of canonical directions, one can further obtain the second pair $\alpha^{(2)}\in \mathbb{R}^{p_1}, \beta^{(2)}\in \mathbb{R}^{p_2}$ such that $\Cov((\alpha^{(1)})^{\intercal}X, (\alpha^{(2)})^{\intercal}X) = \Cov((\beta^{(1)})^{\intercal}Y, (\beta^{(2)})^{\intercal}Y) =0$, and\\ ${\rm Corr}((\alpha^{(2)})^{\intercal}X, (\beta^{(2)})^{\intercal}Y)$ is maximized. The higher order canonical directions can be obtained by repeating this process. If $(X, Y)$ is further assumed to have joint covariance, say
$$\Cov\begin{pmatrix}
X\\
Y
\end{pmatrix} = \Sigma = \begin{bmatrix}
\Sigma_{X} & \Sigma_{XY}\\
\Sigma_{YX} & \Sigma_{YY}
\end{bmatrix}, $$
the population canonical correlation directions can be inductively defined as the following optimization problem. For $k= 1, 2, \cdots$,
\begin{equation*}
\begin{split}
(\alpha^{(k)}, \beta^{(k)}) = \argmax_{a\in \mathbb{R}^{p_1}, b\in \mathbb{R}^{p_2}} &\quad a^{\intercal}\Sigma_{XY}b,\\
\text{subject to} & \quad a^{\intercal}\Sigma_Xa = b^{\intercal}\Sigma_Y b = 1,\\
&\quad a^{\intercal}\Sigma_X \alpha^{(l)} = b^{\intercal}\Sigma_Y \beta^{(l)} = 0, \forall 1\leq l\leq k-1.
\end{split}
\end{equation*}
A more explicit form for the canonical correlation directions is given in \cite{hotelling1936relations}: $(\Sigma_X^{1\over 2}\alpha^{(k)}, \Sigma_Y^{1\over 2}\beta^{(k)})$ is the $k$-th pair of singular vectors of $\Sigma_X^{-{1\over 2}} \Sigma_{XY}\Sigma_Y^{-{1\over 2}}$. We combine the leading $r$ population canonical correlation directions and write 
$$A = \left[\alpha^{(1)} \cdots \alpha^{(r)}\right],\quad B = \left[\beta^{(1)}\cdots \beta^{(r)}\right].$$

Suppose one observes i.i.d.samples $(X_i^{\intercal}, Y_i^{\intercal})^{\intercal} \sim N(0, \Sigma)$. Then the sample covariance and cross-covariance for $X$ and $Y$ can be calculated as
$$\hat{\Sigma}_X = \frac{1}{n} \sum_{i=1}^n X_i X_i^{\intercal},\quad \hat{\Sigma}_Y = \frac{1}{n} \sum_{i=1}^n Y_iY_i^{\intercal}, \quad \hat{\Sigma}_{XY} = \frac{1}{n}\sum_{i=1}^n X_iY_i^{\intercal}. $$
The standard approach to estimate the canonical correlation directions $\alpha^{(k)}$, $\beta^{(k)}$ is via the SVD of $\hat{\Sigma}_X^{-{1\over 2}} \hat{\Sigma}_{XY} \hat{\Sigma}_Y^{-{1\over 2}}$
$$\hat{\Sigma}_X^{-{1\over 2}} \hat{\Sigma}_{XY} \hat{\Sigma}_Y^{-{1\over 2}} = \hat{U} \hat{S}\hat{\V} = \sum_{k=1}^{p_1\wedge p_2} \hat U_{[:, k]} \hat S_{kk} \hat U_{[:, k]}^{\intercal}. $$
Then the leading $r$ sample canonical correlation directions can be calculated as
\begin{equation}\label{eq:CCA-hatA-hatB}
\begin{split}
&\hat{A} = \hat{\Sigma}_X^{-{1\over 2}}\hat U_{[:, 1:r]},\quad \hat A = [\hat{\alpha}^{(1)}, \hat{\alpha}^{(2)},\cdots, \hat{\alpha}^{(r)}], \\
& \hat {B}  = \hat{\Sigma}_Y^{-{1\over 2}}\hat V_{[:, 1:r]}, \quad \hat B = [\hat{\beta}^{(1)}, \hat{\beta}^{(2)},\cdots, \hat{\beta}^{(r)}].
\end{split}
\end{equation}
$\hat{A}, \hat{B}$ are consistent estimators for the first $r$ left and right canonical directions in the classical fixed dimension case. 

Let $X^\ast \in \mathbb{R}^{p_1}$ be an independent copy of the original sample $X$, we define  the following two losses to measure the accuracy of the estimator of the canonical correlation directions,
\begin{align}\label{eq:L_sp_A}
L_{\rm sp}(\hat A, A) &= \min_{O\in \mathbb{O}_r}\max_{v\in \mathbb{R}^{r}, \|v\|_2=1} \E_{X^\ast} \left((\hat{A}O v)^{\intercal} X^\ast - (Av)^{\intercal} X^\ast \right)^2, \\
L_F (\hat{A}, A) &= \min_{O\in \mathbb{O}_r} \E_{X^\ast}\|(\hat{A}O)^{\intercal} X^\ast - A^{\intercal} X^\ast\|^2_2.
\end{align}
These two losses quantify how well the estimator $(\hat{A} O)^{\intercal} X^\ast$ can predict the values of the canonical variables $A^{\intercal} X^\ast$, where $O\in \mathbb{O}_{r}$ is a rotation matrix as the objects of interest here are the directions. 

The following theorem gives the upper bound for one side of the canonical correlation directions. The main technical tool is the perturbation bounds given in Section \ref{sec.main}.
\begin{Theorem}\label{th:CCA}
	Suppose $(X_i, Y_i) \sim N(0, \Sigma), i=1,\cdots, n$, where $S = \Sigma_X^{-{1\over 2}}\Sigma_{XY}\Sigma_Y^{-{1\over 2}}$ is of rank-$r$. Suppose $\hat A\in \mathbb{R}^{p_1\times r}$ is given by \eqref{eq:CCA-hatA-hatB}.  Then there exist uniform constants $C_{\rm gap}, C, c>0$ such that whenever $\sigma_r(S)^2 \geq \frac{C_{\rm gap}((p_1p_2)^{1\over 2} + p_1+ p_2^{3/2}n^{-{1\over 2}})}{n}$,
	\begin{equation*}
	\P\left(L_{\rm sp}(\hat A, A) \leq \frac{C p_1(n\sigma_r^2(S)+p_2)}{n^2\sigma_r^4(S)} \right)\geq 1 - C\exp(-cp_1\wedge p_2),
	\end{equation*}
	\begin{equation*}
	\P\left(L_F(\hat A, A) \leq \frac{C p_1r(n\sigma_r^2(S)+p_2)}{n^2\sigma_r^4(S)} \right)\geq 1 - C\exp(-cp_1\wedge p_2).	
	\end{equation*}
	The results for $\hat{B}$ can be stated similarly.
\end{Theorem}

\begin{Remark}
	\cite{chen2013sparse} and \cite{gao2014sparse,gao2015minimax} considered sparse CCA, where the canonical correlation directions $A$ and $B$ are assumed to be jointly sparse. In particular, \cite{chen2013sparse} and \cite{gao2015minimax} proposed estimators under different settings and provided a unified rate-optimal bound for jointly  estimating left and right canonical correlations. \cite{gao2014sparse} proposed another computationally feasible estimators $\hat{A}^\ast$ and $ \hat{B}^\ast$ and provided a minimax rate-optimal bound for $L_F(\hat{A}^\ast, A)$ under regularity conditions that can also be used to prove the consistency of $\hat{B}^\ast$.
	
	Now consider the setting where $p_2 \gg p_1$, $\frac{p_2}{n} \gg \sigma_{r}^2(S) = t^2 \gg \frac{(p_1p_2)^{1\over 2}}{n}$. The lower bound result in Theorem 3.3 by \cite{gao2014sparse} implies that there is no consistent estimator for the right canonical correlation directions $B$. While Theorem \ref{th:CCA} given above shows that the left sample canonical correlation directions $\hat A$ are a consistent estimator of $A$. This interesting phenomena again shows the merit of our proposed unilateral perturbation bound.
	
	It is also interesting to develop the lower bounds for $\hat{A}$ and $\hat{B}$. The best known result,  given in Theorem 3.2 in \cite{gao2014sparse}, is the following two-sided lower bound for both $\hat{A}$ and $\hat{B}$ in Frobenius norm loss:
	\begin{equation*}
	\begin{split}
	\inf_{\hat{A}, \hat{B}} \sup_{A, B} \P\left\{\max\left\{L_F\left(\hat{A}, A\right), L_F\left(\hat{B}, B\right)\right\} \geq c\left(\frac{r(p_1+p_2)}{n\sigma_{\min}^2(S)}\wedge 1\right)\right\} \geq 0.8.
	\end{split}
	\end{equation*}	
	Establishing the matching one-sided  lower bound for Theorem \ref{th:CCA} is technical challenging. We leave it for future research.
\end{Remark}

\section{Simulations}
\label{sec.simulation}


In this section, we carry out numerical experiments to further illustrate the advantages of the separate bounds for the left and right singular subspaces over the uniform bounds. As mentioned earlier, in a range of cases, especially when the numbers of rows and columns of the matrix differ significantly, it is even possible that the singular space on one side can be stably recovered, while the other side cannot. To illustrates this point, we specifically perform simulation studies in matrix denoising, high-dimensional clustering, and canonical correlation analysis.

We first consider the matrix denoising model discussed in Section \ref{sec.denoising}. Let $X = tUV^{\intercal}\in \mathbb{R}^{p_1\times p_2}$, where $t\in \mathbb{R}$,  $U$ and $V$ are $p_1\times r$ and $p_2\times r$ random uniform orthonormal columns with respect to the Haar measure. Let  the perturbation $Z=(Z_{ij})_{p_1\times p_2}$  be randomly generated with $Z_{ij}\overset{iid}{\sim}N(0, 1)$. We calculate the SVD of $X+Z$ and form the first $r$ left and right singular vectors as $\hat{U}$ and $\hat{V}$. The average losses in Frobenius and spectral $\sin \Theta$ distances for both the left and right singular space estimates with 1,000 repetitions are given in Table \ref{tb:simu_U_V} for various values of $(p_1, p_2, r, t)$. It can be easily seen from this experiment that the left and right singular perturbation bounds behave very distinctly when $p_1 \gg p_2$.
 
\begin{table}[htbp]
	\label{tb:simu_U_V}
	\begin{center}
		\begin{tabular}{rcccc}\hline
			$(p_1, p_2, r, t)$ & $\|\sin\Theta(\hat{U}, U)\|^2$ & $\|\sin\Theta(\hat{V}, V)\|^2$ & $\|\sin\Theta(\hat{U}, U)\|_F^2$ & $\|\sin\Theta(\hat{V}, V)\|_F^2$\\\hline
			(100, 10, 2, 15) & 0.3512  &  0.0669 &   0.6252 &  0.0934\\
			(100, 10, 2, 30) & 0.1120  &  0.0139 &   0.1984 &   0.0196\\
			(100, 20, 5, 20) & 0.2711 &   0.0930 &   0.9993 &  0.2347\\
			(100, 20, 5, 40) & 0.0770 &   0.0195 &   0.2835 &  0.0508\\
			(1000, 20, 5, 30) & 0.5838  &  0.0699  &  2.6693  &  0.1786\\
			(1000, 20, 10, 100) & 0.1060 &   0.0036 &   0.9007  &  0.0109\\
			(1000, 200, 10, 50) & 0.3456 &   0.0797 &   2.9430 &   0.4863\\
			(1000, 200, 50, 100) & 0.1289 &   0.0205  &  4.3614  &  0.2731 \\
			\hline
		\end{tabular}
		\caption{Average losses in Frobenius and spectral $\sin \Theta$ distances for both the left and right singular space changes after Gaussian noise perturbations.}
	\end{center}
\end{table}

We then consider the high-dimensional clustering model studied in Section \ref{clustering.sec}. Let $\tilde{\mu} {\sim} N(0, I_p)$ and $\mu = t(p/n)^{1/4}\cdot \tilde{\mu} / \|\tilde{\mu}\|_2 \in \mathbb{R}^p$, where $t = \|\mu\|_2$ essentially represents the signal strength. The group label $l\in \mathbb{R}^{n}$ is randomly generated as
$$l_i \overset{iid}{\sim} \left\{\begin{array}{ll}
1, & \text{with probability } \rho,\\
-1, & \text{with probability } 1 - \rho.
\end{array}\right.$$
Based on $n$ i.i.d. observations: $Y_i = l_i\mu + Z_i, Z_i\overset{iid}{\sim}N(0, I_p), i =1,\ldots, n$, we apply the proposed estimator \eqref{PCA-clustering} to estimate $l$. The results for different values of $(n, p, t, \rho)$ are provided in Table \ref{tb:clustering}. It can be seen that the numerical results match our theoretical analysis -- the proposed $\hat{l}$ achieves good performance roughly when $t\geq C(p/n)^{1/4}$.
\begin{table}[htbp]
	\label{tb:clustering}
	\begin{center}
		\begin{tabular}{@{}r|cccccc@{}}\hline
			\backslashbox{$(p, t, \rho)$}{$n$} & 5 & 10 & 20  & 50 & 100 & 200 \\\hline
			(100, 1, 1/2) &     0.2100  &  0.1485 &   0.0690 &   0.0494 &   0.0440 &   0.0333\\
			(100, 1, 3/4) &     0.2150  &  0.1590 &   0.0680 &   0.0468 &   0.0422 &   0.0290\\
			(100, 3, 1/2) &    0.0019  &  0.0005   &   0.0000  &  0.0000 &  0.0000  & 0.0000 \\
			(100, 3, 3/4) &   0.0020  &   0.0005  &    0.0000  & 0.0000  & 0.0000 & 0.0000\\
			(1000, 1, 1/2) &   0.3260  &   0.3510  &  0.3594 &   0.2855  &  0.2691 &   0.1364 \\
			(1000, 1, 3/4) &   0.3610 &   0.3610  &  0.3462  &  0.3057  &  0.2696   & 0.1410 \\
			(1000, 3, 1/2) &  0.1370  &  0.0485  &  0.0066 &   0.0019  &  0.0013 &   0.0003\\
			(1000, 3, 3/4) &  0.1160   & 0.0425  &  0.0046 &   0.0019  &  0.0018   & 0.0006 \\			
			\hline
		\end{tabular}
		\caption{Average misclassification rate for different settings.}
	\end{center}
\end{table}

We finally investigate the numerical performance of canonical correlation analysis particularly when the dimensions of two samples differ significantly. Suppose $\Sigma_X = I_{p_1} + \frac{1}{2\|Z_{p_1} + Z_{p_1}^{\intercal}\|}\left(Z_{p_1} + Z_{p_1}^{\intercal}\right), \Sigma_Y = I_{p_2} + \frac{1}{2\|Z_{p_2} + Z_{p_2}^{\intercal}\|}\left(Z_{p_2} + Z_{p_2}^{\intercal}\right)$, $\Sigma_{XY} = \Sigma_X^{1/2}\cdot \left(t UV^{\intercal}\right) \Sigma_Y^{1/2}$, where $Z_{p_1}$ and $Z_{p_2}$ are i.i.d. Gaussian matrices; $U\in \mathbb{O}_{p_1, r}, V\in \mathbb{O}_{p_2, r}$ are random orthogonal matrices. With $n$ pairs of observations
$$\begin{pmatrix}
X_i\\
Y_i
\end{pmatrix} \overset{iid}{\sim} N(0, \Sigma), \quad \Sigma = \begin{bmatrix}
\Sigma_X & \Sigma_{XY}\\
\Sigma_{XY}^{\intercal} & \Sigma_{Y}
\end{bmatrix},\quad i=1,\ldots, n, $$
we apply the procedure discussed in Section \ref{sec.CCA} to obtain $\hat{A}$ and $\hat{B}$, i.e. the estimates for left and right canonical correlation directions. Since the exact losses in $L_{\rm sp}(\cdot, \cdot)$, $L_{F}(\cdot, \cdot)$ metrics \eqref{eq:L_sp_A} involves difficult optimization, we instead measure the losses in 
$$\left\|\sin\Theta(\hat{U}, U)\right\|, \left\|\sin\Theta(\hat{U}, U)\right\|_F, \left\|\sin\Theta(\hat{V}, V)\right\|,\text{ and }\left\|\sin\Theta(\hat{V}, V)\right\|_F.$$ 
Here $U, V, \hat{U}, \hat{V}$ are the first $r$ left and right singular vectors of $\Sigma_X^{-1/2}\Sigma_{XY}\Sigma_Y^{-1/2}$ and $\hat{\Sigma}_X^{-1/2}\hat{\Sigma}_{XY}\hat{\Sigma}_Y^{-1/2}$, respectively. It is shown in Step 1 of the proof for Theorem \ref{th:CCA} that these measures are equivalent to $L_{\rm sp}$ and $L_{F}$. The results under various choices of $(p_1, p_2, n, t)$ are collected in Table \ref{tb:CCA}. It can be easily seen that the performance of the right canonical direction estimation is much better than the left ones when $p_1$ is much larger than $p_2$, which is consistent with the theoretical results in Theorem \ref{th:CCA} and illustrates the power of the newly proposed perturbation bound results.
\begin{table}[htbp]
	\begin{center}
		\begin{tabular}{rcccc}\hline
			$(p_1, p_2, r, t)$ & $\|\sin\Theta(\hat{U}_S, U_S)\|$ & $\|\sin\Theta(\hat{U}_S, U_S)\|_F$ & $\|\sin\Theta(\hat{V}_S, V_S)\|$ & $\|\sin\Theta(\hat{V}_S, V_S)\|_F$ \\\hline
			(30, 10, 100, .8) & 0.3194  &  0.6609  &  0.1571 &   0.2530\\
			(30, 10, 200, .5) & 0.5348  &  1.1111  &  0.3343  &  0.5256\\
			(100, 10, 200, .8) & 0.4103  &  1.0145 &   0.1120 &   0.1825\\
			(100, 10, 500, .5) & 0.5183  &  1.2821  &  0.1614 &   0.2606\\
			(200, 20, 500, .8) & 0.3239 &   0.8428 &   0.0746 &   0.1442\\
			(200, 20, 800, .5) & 0.5834 &   1.5155 &   0.2423 &   0.4605\\
			(500, 50, 1000, .8) & 0.3875 &   1.0515 &   0.1091 &   0.2472\\
			(500, 50, 2000, .5) & 0.5677 &   1.5467 &   0.2216  &  0.4910\\
			\hline
		\end{tabular}
	\end{center}
\caption{Average losses in $L_{\rm sp}(\cdot, \cdot)$ and $L_{F}(\cdot, \cdot)$ metrics for the left and right canonical directions.}
	\label{tb:CCA}
\end{table}

\section{Discussions}
\label{sec.discussion}

We have established in the present paper new and rate-optimal perturbation bounds, measured in both spectral and Frobenius $\sin \Theta$ distances,  for the left and right singular subspaces separately. 
These perturbation bounds are widely applicable to the analysis of many high-dimensional problems. In particular, we applied the perturbation bounds to study three important problems in high-dimensional statistics: low-rank matrix denoising and singular space estimation, high-dimensional clustering and CCA.
As mentioned in the introduction, in addition to these problems and possible extensions discussed in the previous sections, the obtained perturbation bounds can be used in a range of other applications including {\it community detection in bipartite networks,  multidimensional scaling,  cross-covariance matrix estimation, and singular space estimation for matrix completion}. We briefly discuss these problems here.

An interesting application of the perturbation bounds given in Section \ref{sec.main} is {\it community detection in bipartite graphs}. Community detection in networks has attracted much recent attention. The focus of the current community detection literature has been mainly  on unipartite graph (i.e.,  there are only one type of nodes). However in some applications, the nodes can be divided into different types and only the interactions between the different types of nodes are available or of interest, such as people vs. committees, Facebook users vs. public pages (see \cite{melamed2014community,alzahrani2016community}).
The observations on the connectivity of the network between two types of nodes can be described by an adjacency matrix $A$, where  $A_{ij} = 1$ if the $i$-th Type 1 node and $j$-th Type 2 node are connected, and $A_{ij} = 0$ otherwise. 
The spectral method is one of the most commonly used approaches in the literature with theoretical guarantees \citep{rohe2011spectral,lei2015consistency}. In a bipartite network, the left and right singular subspaces could behave very differently from each other. Our perturbation bounds can be used for community detection in bipartite graph and potentially lead to sharper results in some settings.
 
Another possible application lies in {\it multidimensional scaling (MDS)} with distance matrix between two sets of points. MDS is a popular method of visualizing the data points embedded in low-dimensional space based on the distance matrices \citep{borg2005modern}. Traditionally MDS deals with unipartite distance matrix, where all distances between any pairs of points are observed. In some applications, the data points are from two groups and one is only able to observe its biparitite distance matrix formed by the pairwise distances between points from different groups.  As the SVD is a commonly used technique for dimension reduction in MDS, the  perturbation bounds developed in this paper can be potentially used for the analysis of MDS with bipartite distance matrix. 

In some applications, the {\it cross-covariance matrix}, not the overall covariance matrix,  is of particular interest.  \citep{cai2015large} considered multiple testing of cross-covariances in the context of the phenome-wide association studies (PheWAS). Suppose $X\in \mathbb{R}^{p_1}$ and $Y\in \mathbb{R}^{p_2}$ are jointly distributed with covariance matrix $\Sigma$. Given $n$ i.i.d. samples $(X_i, Y_i)$, $i=1,\cdots, n$, from the joint distribution, one wishes to make statistical inference for the cross-covariance matrix $\Sigma_{XY}$.  If $\Sigma_{XY}$ has low-rank structure, the perturbation bounds established in Section \ref{sec.main} could be potentially applied to make statistical inference for $\Sigma_{XY}$.

{\it Matrix completion,} whose central goal is to recover a large low-rank matrix based on a limited number of observable entries, has been widely studied in the last decade. Among various methods for matrix completion, spectral method is fast, easy to implement and achieves good performance \citep{keshavan2010matrix,chatterjee2014matrix,cho2015asymptotic}. The new perturbation bounds can be potentially used for singular space estimation under the matrix completion setting to yield better results. 

In addition to the aforementioned problems, {\it high-dimensional clustering with correlated features} is an important extension of the  problem of clustering with independent features considered in the present paper. Specifically, based on $n$ observations $Y_i = l_i\mu + Z_i \in \mathbb{R}^p$, $i = 1, \ldots, n$, where $Z_i \overset{iid}{\sim} N(0, \Sigma)$, one aims to recover the unknown labels $\{l_i\}_{i=1}^n$. When $\Sigma$ is known or can be well estimated, one can transform $\tilde{Y}_i = \Sigma^{-1/2} Y_i, i=1,\ldots, n$ and perform the spectral method on $\left\{\tilde{Y}_i\right\}_{i=1}^n$. It would be an interesting and challenging problem to consider the general setting where $\Sigma$ is unknown. We leave this for future research.

\section{Proofs}
\label{sec.proofs}

We prove the main results in Sections \ref{sec.main}, \ref{sec.denoising} and \ref{clustering.sec} in this section. The proofs for CCA and the additional technical results are given in the supplementary materials.

\subsection{Proofs of General Unilateral Perturbation Bounds}
\label{sec.proof.perturbationBD}

Some technical tools are needed to prove Theorems \ref{th:perturbation}, \ref{th:perturbation_lower} and Proposition \ref{lm:perturbation}. In particular, we need a few useful properties of $\sin \Theta$ distances given below in Lemma \ref{lm:sin_Theta_distance}. Specifically, Lemma \ref{lm:sin_Theta_distance} provides some more convenient expressions than the definitions for the $\sin \Theta$ distances. It also shows that they are indeed distances as they satisfy triangle inequality. Some other widely used metrics for orthogonal spaces, including 
\begin{equation}\label{eq:D_sp_D_F}
D_{\rm sp}(\hat V, V) = \inf_{O\in \mathbb{O}_{r}}\|\hat V - VO \|, \quad D_{F}(\hat V, V) = \inf_{O\in \mathbb{O}_r} \|\hat V - VO\|_F, 
\end{equation}
\begin{equation}\label{eq:VVT-VV}
\left\|\hat{V}\hat{V}^{\intercal} - VV^{\intercal}\right\|, \quad \left\|\hat{V}\hat{V}^{\intercal} - VV^{\intercal}\right\|_F. 
\end{equation}
are shown to be equivalent to the $\sin \Theta$ distances. 

\begin{Lemma}[Properties of the $\sin \Theta$ Distances]
	\label{lm:sin_Theta_distance} 
	The following properties hold for the $\sin \Theta$ distances.	
	\begin{enumerate}
		\item {\it (Equivalent Expressions)} Suppose $V, \hat V\in \mathbb{O}_{p, r}$. If $V_{\perp}$ is an orthogonal extension of $V$, namely $[V~V_{\perp}] \in \mathbb{O}_p$, we have the following equivalent forms for $\|\sin\Theta(\hat V, V)\|$ and $\|\sin\Theta(\hat V, V)\|_F$,
		\begin{equation}\label{eq:def_sin_Theta_spectral}
		\|\sin\Theta(\hat V, V)\| = \sqrt{1 - \sigma_{\min}^2(\hat V^{\intercal} V)} = \|\hat V^{\intercal} V_{\perp}\|,
		\end{equation}
		\begin{equation}\label{eq:def_sin_Theta_Frobenius}
		\|\sin\Theta(\hat V, V)\|_F = \sqrt{r - \|V^{\intercal}\hat V\|_F^2} = \|\hat V^{\intercal} V_{\perp}\|_F.
		\end{equation}
		\item {\it (Triangle Inequality)} For any $V_1, V_2, V_3\in \mathbb{O}_{p, r}$,
		\begin{equation}\label{ineq:triangle_spectral}
		\|\sin\Theta(V_2, V_3)\| \leq \|\sin\Theta(V_1, V_2)\| + \|\sin\Theta(V_1, V_3)\|,
		\end{equation}
		\begin{equation}\label{ineq:triangle_Frobenius}
		\|\sin\Theta(V_2, V_3)\|_F \leq \|\sin\Theta(V_1, V_2)\|_F + \|\sin\Theta(V_1, V_3)\|_F.
		\end{equation}
		\item {\it (Equivalence with Other Metrics)} The metrics defined as \eqref{eq:D_sp_D_F} and \eqref{eq:VVT-VV} are equivalent to $\sin \Theta$ distances as the following inequalities hold
		\begin{equation*}
		\|\sin\Theta(\hat V, V)\| \leq D_{\rm sp}(\hat V, V) \leq \sqrt{2}\|\sin\Theta(\hat V, V)\|,
		\end{equation*}
		\begin{equation*}
		\|\sin\Theta(\hat V, V)\|_F \leq D_{F}(\hat V, V) \leq \sqrt{2}\|\sin\Theta(\hat V, V)\|_F,
		\end{equation*}	
		\begin{equation*}
		\left\|\sin\Theta(\hat{V}, V)\right\|\leq \left\|\hat{V}\hat{V}^{\intercal} - VV^{\intercal}\right\| \leq 2\left\|\sin\Theta(\hat{V}, V)\right\|,
		\end{equation*}
		\begin{equation*}
		\left\|\hat{V}\hat{V}^{\intercal} - VV^{\intercal}\right\|_F = \sqrt{2}\left\|\sin\Theta(\hat{V}, V)\right\|_F.
		\end{equation*}
	\end{enumerate}
\end{Lemma}

\ \par

\begin{proof}[Proof of Proposition \ref{lm:perturbation}] First, we can rotate the right singular space by right multiplying the whole matrices $A, V^{\intercal}, W^{\intercal}$ by $[W~W_{\perp}]$ without changing the singular values and left singular vectors. Thus without loss of generality, we assume that 
$$[W~W_{\perp}] = I_{p_2}.$$
Next, we further calculate the SVD: $AW = A_{[:, 1:r]} := \bar{U}\bar{\Sigma}\bar{V}^{\intercal}$, where $\bar{U} \in \mathbb{O}_{p_1, r}, \bar{\Sigma}\in \mathbb{R}^{r\times r}$, $\bar{V} \in \mathbb{O}_r$, and rotate the left singular space by left multiplying the whole matrix $A$ by $[\bar{U} ~ \bar{U}_{\perp}]^{\intercal}$, then rotate the right singular space by right multiplying $A_{[:, 1:r]}$ by $\bar{V}$. After this rotation, the singular structure of $A$, $AW$ are unchanged. Again without loss of generality, we can assume that $[\bar{U} ~ \bar{U}_{\perp}]^{\intercal} = I_{p_1}, \bar{V} = I_r.$
After these two steps of rotations, the formation of $A$ is much simplified,
\begin{equation}
A = \begin{blockarray}{ccccc}
& & r & & p_2-r\\
\begin{block}{c[cccc]}
& \sigma_1(AW) & & & \\
r & & \ddots & & \bar{U}^{\intercal} AW_{\perp}\\
& & & \sigma_r(AW) & \\
p_1 - r & & 0 & & \bar{U}_\perp^{\intercal} A W_{\perp}\\
\end{block}
\end{blockarray},
\end{equation}
while the problem we are considering is still without loss of generality. For convenience, denote 
\begin{equation}\label{eq:tilde U A W = y}
\left(\bar{U}^{\intercal} AW_{\perp}\right)^{\intercal} = [y^{(1)} ~ y^{(2)} \cdots y^{(r)}], \quad y^{(1)},\cdots, y^{(r)}\in \mathbb{R}^{p_2-r}.
\end{equation}
We can further compute that
\begin{equation}
A^{\intercal}A = \begin{blockarray}{ccccc}
& & r & & p_2-r\\
\begin{block}{c[cccc]}
& \sigma_1^2(AW) & & & \sigma_1(AW) y^{(1)\intercal} \\
r & & \ddots & & \vdots \\
& & & \sigma_r^2(AW) & \sigma_r(AW) y^{(r)\intercal} \\
p_2 - r & \sigma_1(AW) y^{(1)} & \cdots  & \sigma_r(AW) y^{(r)} & (A W_{\perp})^{\intercal}A W_{\perp}\\
\end{block}
\end{blockarray}.
\end{equation}
By basic theory in algebra, the $i$-th eigenvalue of $A^{\intercal}A$ is equal to $\sigma_i^2(A)$, and the $i$-th eigenvector of $A^{\intercal}A$ is equal to the $i$-th right singular vector of $A$ (up-to-sign). Suppose the singular vectors of $A$ are $\tilde{V} = [v^{(1)}, v^{(2)}, \cdots, v^{(p_2)}]$, where the singular values can be further decomposed into two parts as
\begin{equation}\label{eq:def_beta_alpha}
v^{(k)} = \begin{blockarray}{cc}
\begin{block}{c[c]}
r & \alpha^{(k)}\\
p_2 - r& \beta^{(k)}\\
\end{block}
\end{blockarray},\quad \text{or equivalently}, \quad \alpha^{(k)} =W^{\intercal}v^{(k)}, \beta^{(k)} = W_{\perp}^{\intercal}v^{(k)}.
\end{equation}
By observing the $i$-th entry of $A^{\intercal}A v^{(k)} = \sigma_k^2(A) v^{(k)}$, we know for $1\leq i\leq r$, $r+1 \leq k\leq p_2$,
\begin{equation}\label{eq:alpha_i^k}
\begin{split}
& \left(\sigma_i^2(AW) - \sigma_k^2(A) \right) \alpha^{(k)}_i + \sigma_i(AW) y^{(i)\intercal} \beta^{(k)} = 0,\\
\Rightarrow \quad & \alpha_i^{(k)} = \frac{-\sigma_i(AW)}{\sigma_i^2(AW) - \sigma_k^2(A)} y^{(i)\intercal}\beta^{(k)}.
\end{split}
\end{equation}
Recall the assumption that 
\begin{equation}\label{ineq:sigma_1(AW)geq}
\sigma_1(AW) \geq \cdots \geq \sigma_r(AW) > \sigma_{r+1}(A) \geq \cdots \sigma_{p_2}(A)\geq 0. 
\end{equation} 
Also $\frac{x}{x^2- y^2} = \frac{1}{x - y^2/x}$ is a decreasing function for $x$ and a increasing function for $y$ when $x>y \geq 0$, so
\begin{equation}\label{ineq:sigma_i(AW)}
\frac{\sigma_i(AW)}{\sigma_i^2(AW) - \sigma_k^2(A)} \leq \frac{\sigma_r(AW)}{\sigma_r^2(AW) - \sigma_{r+1}^2(A)},\quad 1\leq i\leq r, r+1\leq k \leq p_2.
\end{equation} 
Since $[\beta^{(r+1)} ~ \cdots ~ \beta^{(p_2)}]$ is the submatrix of the orthogonal matrix $V$,
\begin{equation}\label{ineq:beta_spectral_norm}
\left\|[\beta^{(r+1)} ~ \cdots ~ \beta^{(p_2)}]\right\| \leq 1.
\end{equation}
Now we can give an upper bound for the Frobenius norm of $[\alpha^{(r+1)} ~ \cdots ~ \alpha^{(p_2)}]$
\begin{equation*}
\begin{split}
& \left\|\left[\alpha^{(r+1)} ~ \cdots ~ \alpha^{(p_2)}\right]\right\|_F^2 = \sum_{i=1}^{r}\sum_{k=r+1}^{p_2} \left(\alpha_i^{(k)}\right)^2\\
\overset{\eqref{ineq:sigma_i(AW)}}{\leq} & \frac{\sigma_r^2(AW)}{\left(\sigma_r^2(AW) - \sigma_{r+1}^2(A)\right)^2} \sum_{i=1}^r \sum_{k=r+1}^{p_2} \left(y^{(i)\intercal}\beta^{(k)}\right)^2\\
\leq  & \frac{\sigma_r^2(AW)}{\left(\sigma_r^2(AW) - \sigma_{r+1}^2(A)\right)^2} \left\|[y_1 ~ \cdots y_r]^{\intercal}\right\|_F^2 \left\|[\beta^{(r+1)} ~ \cdots ~ \beta^{(p_2)}]\right\|^2\\
\overset{\eqref{eq:tilde U A W = y}\eqref{ineq:beta_spectral_norm}}{\leq} & \frac{\sigma_r^2(AW)}{\left(\sigma_r^2(AW) - \sigma_{r+1}^2(A)\right)^2} \left\|\bar{U}^{\intercal} AW_{\perp}\right\|_F^2.
\end{split}
\end{equation*}
It is more complicated to give a upper bound for the spectral norm of $[\alpha^{(r+1)} ~ \cdots ~ \alpha^{(p_2)}]$. Suppose $s = (s_{r+1},\cdots, s_{p_2})\in \mathbb{R}^{p_2-r}$ is any vector with $\|s\|_2 = 1$. Based on \eqref{eq:alpha_i^k}, 
\begin{equation*}
\begin{split}
 & \sum_{k=r+1}^{p_2} s_k\alpha^{(k)}_i\\
 = & \sum_{k=r+1}^{p_2}  \frac{-s_k\sigma_i(AW)y^{(i)\intercal}\beta^{(k)}}{\sigma_i^2(AW) - \sigma_k^2(A)} = \sum_{k=r+1}^{p_2}\frac{-s_k}{\sigma_i(AW)}\frac{1}{1-\sigma_k^2(A)/\sigma_i(AW)^2}y^{(i)\intercal} \beta^{(k)}\\
\overset{\eqref{ineq:sigma_1(AW)geq}}{=} & \sum_{k=r+1}^{p_2}\sum_{l=0}^\infty \frac{-s_k\sigma_k^{2l}(A)}{\sigma_i^{2l+1}(AW)}y^{(i)\intercal}\beta^{(k)} = \sum_{l=0}^\infty\frac{-y^{(i)\intercal}}{\sigma_i^{2l+1}(AW)}\left(\sum_{k= r+1}^{p_2}s_k\sigma_k^{2l}(A)\beta^{(k)}\right).  \\
\end{split}
\end{equation*}
Hence,
\begin{equation*}
\begin{split}
\left\|\sum_{k=r+1}^{p_2}s_k \alpha^{(k)}\right\|_2 \leq & \sum_{l=0}^\infty \left\|\begin{bmatrix}
y^{(1)\intercal}/\sigma_1^{2l+1}(AW)\\
\vdots\\
y^{(r)\intercal}/\sigma_r^{2l+1}(AW)
\end{bmatrix}\cdot\left(\sum_{k= r+1}^{p_2}s_k\sigma_k^{2l}(A)\beta^{(k)}\right)\right\|_2\\
\leq & \sum_{l=0}^\infty \frac{\left\|[y^{(1)} ~ y^{(2)} ~ \cdots ~ y^{(r)}]\right\|}{\sigma_r^{2l+1}(AW)}\cdot \left\|[\beta^{(r+1)} ~ \beta^{(r+2)} ~ \cdots ~ \beta^{(p_2)}]\right\|\\
& \cdot \left\|\left(s_{r+1}\sigma^{2l}_{r+1}(A), \cdots, s_{p_2}\sigma_{p_2}^{2l}(A) \right)\right\|_2\\
\overset{\eqref{eq:tilde U A W = y}\eqref{ineq:beta_spectral_norm}\eqref{ineq:sigma_1(AW)geq}}{\leq}& \sum_{l=0}^\infty \frac{\|\tilde{U}^{\intercal}AW_{\perp}\|}{\sigma_{r}^{2l+1}(AW)} \cdot \sigma_{r+1}^{2l}(A) \|s\|_2\\
= & \frac{\|\tilde{U}^{\intercal}AW_{\perp}\|\sigma_{r}(AW)}{\sigma_{r}^2(AW) - \sigma^2_{r+1}(A)},
\end{split}
\end{equation*}
which implies 
\begin{equation*}
\|[\alpha^{(r+1)}~\cdots~\alpha^{(p_2)}]\| \leq \frac{\|\tilde{U}^{\intercal}AW_{\perp}\|\sigma_{r}(AW)}{\sigma_{r}^2(AW) - \sigma^2_{r+1}(A)}.
\end{equation*}
Note the definition of $\alpha^{(i)}$ in \eqref{eq:def_beta_alpha}, we know 
$$[\alpha^{(r+1)} \alpha^{(r+2)} \cdots \alpha^{(p_2)}] = \tilde{V}_{[1:r, (r+1):p_2]} = (V_{\perp})_{[1:r, :]}.$$ 
Thus,
\begin{equation*}
\begin{split}
\left\|\sin\Theta(V, W)\right\| \overset{\eqref{eq:def_sin_Theta_spectral}}{=} \left\|W^{\intercal}V_{\perp}\right\| = \left\|[\alpha^{(r+1)}\cdots \alpha^{(p_2)}]\right\| \leq \frac{\|\tilde{U}^{\intercal}AW_{\perp}\|\sigma_{r}(AW)}{\sigma_{r}^2(AW) - \sigma^2_{r+1}(A)},
\end{split}
\end{equation*}
\begin{equation}
\begin{split}
& \left\|\sin\Theta(V, W)\right\|_F^2 \overset{\eqref{eq:def_sin_Theta_Frobenius}}{=} \|W^{\intercal}V_{\perp}\|_F^2\\
= & \|[\alpha^{(r+1)}~\cdots~\alpha^{(p_2)}]\|_F^2 \leq \frac{\|\tilde{U}^{\intercal}AW_{\perp}\|_F^2\sigma_{r}^2(AW)}{\left(\sigma_{r}^2(AW) - \sigma^2_{r+1}(A)\right)^2}.	
\end{split}
\end{equation}
Finally, since $\bar{U}$ is the left singular vectors of $AW$, 
\begin{equation}\label{eq:P_{(AW)}}
\|\bar{U}^{\intercal}AW_\perp \| = \|\bbP_{(AW)}AW_{\perp}\|, \quad \|\bar{U}^{\intercal}AW_{\perp}\|_F = \|\bbP_{(AW)}AW_{\perp}\|.
\end{equation}
The upper bounds 1 in \eqref{ineq:lm_V_W_spectral_norm} and $\sqrt{r}$ on \eqref{ineq:lm_V_W_Frobenius} are trivial. Therefore, we have finished the proof of Proposition \ref{lm:perturbation}. 
\end{proof}

\ \par

\begin{proof}[Proof of Theorem \ref{th:perturbation}] Before proving this theorem, we introduce the following lemma on the inequalities of the singular values in the perturbed matrix.
\begin{Lemma}\label{lm:X_Y_sv}
	Suppose $X\in \mathbb{R}^{p\times n}$, $Y\in \mathbb{R}^{p\times n}$, $rank(X) = a$, $rank(Y) = b$,
	\begin{enumerate}
		\item $\sigma_{a+b+1-r}(X+Y) \leq \min(\sigma_{a+1-r}(X), \sigma_{b+1-r}(Y))$ for $r\geq 1$;
		\item if we further have $X^{\intercal}Y = 0$ or $XY^{\intercal} = 0$, we must have $a+b\leq n\wedge p$, and
		$$\sigma_r^2(X+Y) \geq \max(\sigma_r^2(X), \sigma_r^2(Y))$$ 
		for any $r\geq 1$. Also, 
		$$\sigma_1^2(X+Y) \leq \sigma_1^2(X) + \sigma_1^2(Y). $$
	\end{enumerate}
\end{Lemma}
The proof of Lemma \ref{lm:X_Y_sv} is provided in the supplementary materials. Applying Lemma \ref{lm:X_Y_sv}, we get
\begin{equation}\label{ineq:sigma_min_hat_XV}
\begin{split}
 & \sigma_{\min}^2(\hat{X}V) = \sigma_r^2(\hat X V) = \sigma_{r}^2(UU^{\intercal}\hat{X}V + U_{\perp}U_{\perp}^{\intercal}\hat{X} V)\\
 \geq & \sigma_{r}^2(UU^{\intercal}\hat XV) = \alpha^2, \quad (\text{by Lemma \ref{lm:X_Y_sv} Part 2}.) 
\end{split}
\end{equation}
Since $U, V$ have $r$ columns, $\rank(\hat{X}VV^{\intercal}), \rank(UU^{\intercal}\hat{X}) \leq r$. Also since $\hat{X} = U_{\perp}U_{\perp}^{\intercal}\hat{X} + UU^{\intercal}\hat{X} = \hat{X}V_{\perp}V_{\perp}^{\intercal} + \hat{X}VV^{\intercal}$, we have
\begin{equation*}
\begin{split}
\sigma_{r+1}^2(\hat{X}) \leq & \min\left\{\sigma_1^2(U_{\perp}U_{\perp}^{\intercal} \hat X), \sigma_1^2(\hat X V_{\perp}V_{\perp}^{\intercal})\right\} \quad (\text{by Lemma \ref{lm:X_Y_sv} Part 1.})\\
= & \min\left\{\sigma_1^2(Z_{21} + U_{\perp}^{\intercal}\hat X V_{\perp}), \sigma_1^2 (Z_{12} + U_{\perp}^{\intercal}\hat X V_{\perp})\right\} \\
\leq & (\beta^2 + z_{12}^2)\wedge (\beta^2 + z_{21}^2) \quad (\text{by Lemma \ref{lm:X_Y_sv} Part 2.})\\
= & \beta^2 + z_{12}^2\wedge z_{21}^2.
\end{split}
\end{equation*}
We shall also note the fact that for any matrix $A\in \mathbb{R}^{p\times r}$ with $r\leq p$, denote the SVD as $A = U_A\Sigma_AV_A^{\intercal}$, then
\begin{equation}\label{eq:singular_projection_inverse}
\left\|A\left(A^{\intercal}A\right)^{\dagger}\right\| = \left\|U_A\Sigma_A V_A^{\intercal}\left(V_A\Sigma_A^2V_A^{\intercal}\right)^{\dagger}\right\| = \left\|U_A\Sigma_A^{\dagger}V_A^{\intercal}\right\| \leq \sigma_{\min}^{-1}(A).
\end{equation}
Thus,
\begin{equation*}
\begin{split}
& \|\bbP_{(\hat X V)} \hat X V_{\perp} \| = \|\bbP_{(\hat X V)} \bbP_{U} \hat X V_{\perp} + \bbP_{(\hat X V)}\bbP_{U_{\perp}} \hat X V_{\perp}\|\\
\leq & \|\bbP_{(\hat X V)} U U^\intercal \hat X V_{\perp}\| + \|\bbP_{(\hat{X}V)} U_{\perp}U_{\perp}^\intercal \hat X V_{\perp}\|\\
\leq & \|U^\intercal \hat{X}V_\perp\| + \|\hat{X}V[(\hat{X}V)^\intercal (\hat{X}V)]^{-1}(\hat{X}V)^\intercal U_{\perp}U_{\perp}^\intercal \hat{X}V_{\perp}\|\\
\leq & \|U^\intercal \hat{X} V_{\perp}\| + \|\hat{X}V[(\hat{X}V)^\intercal (\hat{X}V)]^{-1}\|\cdot \|U_{\perp}^\intercal \hat{X}V \|\cdot \|U_{\perp}^\intercal \hat{X}V_{\perp}\|\\
\overset{\eqref{eq:singular_projection_inverse}}{\leq} & \|U^\intercal ZV_{\perp}\| + \frac{1}{\sigma_{\min}(\hat{X}V)} \|U^\intercal_{\perp}ZV\|\cdot \|U_{\perp}^\intercal \hat{X}V_{\perp}\|\\
\overset{\eqref{ineq:sigma_min_hat_XV}}{\leq} & z_{12} + \frac{\beta}{\alpha}z_{21} = \frac{\alpha z_{12} + \beta z_{21}}{\alpha}.\\
\end{split}
\end{equation*}
Similarly,
\begin{equation*}
\begin{split}
\|P_{(\hat X V)} \hat X V_{\perp} \|_F &  \leq \frac{\alpha \|Z_{12}\|_F + \beta \|Z_{21}\|_F}{\alpha}.
\end{split}
\end{equation*}
Next, applying Proposition \ref{lm:perturbation} by setting $A = \hat X$, $\tilde{W} = [V~ V_{\perp}]$, $\tilde{V} = [\hat V ~ \hat V_{\perp}]$, we could obtain \eqref{ineq:th-main_spectral}. 
\end{proof}

\bibliographystyle{apa}
\bibliography{reference_add}

\newpage

\newpage

\setcounter{page}{1}
\setcounter{section}{0}


\centerline{\Large\bf SUPPLEMENT TO ``RATE-OPTIMAL PERTURBATION}\smallskip
\centerline{\Large\bf  BOUNDS FOR SINGULAR SUBSPACES WITH APPLICATIONS}\smallskip
\centerline{\Large\bf TO HIGH-DIMENSIONAL STATISTICS"}
\bigskip

\centerline{\large BY T. TONY CAI \footnote{The research was supported in part by NSF FRG Grant DMS-0854973, NSF Grant DMS-1208982 and NIH Grant R01 CA127334-05.} \, and \, ANRU ZHANG}
\smallskip
\centerline{\large University of Pennsylvania ~ and ~ University of Wisconsin-Madison}

\bigskip\bigskip


In this supplementary material, we provide the proofs for Theorem \ref{th:perturbation_lower}, Corollary \ref{cr:perturbation_single_singular_vector}, matrix denoising, high-dimensional clustering, canonical correlation analysis and all the technical lemmas.

\section{Additional Proofs}

\ \par

\begin{proof}[Proof of Theorem \ref{th:perturbation_lower}]
	
	The construction of lower bound relies on the following design of 2-by-2 blocks.
	
	\begin{Lemma}[SVD of 2-by-2 matrices]\label{lm:2-by-2} 
		~
		\begin{enumerate}
			\item Suppose 2-by-2 matrix $B$ satisfies 
			\begin{equation*}
			B = \begin{bmatrix}
			a & b\\
			0 & d
			\end{bmatrix}, \quad a, b, d \geq 0, \quad a^2 \leq b^2 + d^2.
			\end{equation*}
			Suppose $V = \begin{bmatrix}
			v_{11} & v_{12}\\
			v_{21} & v_{22}
			\end{bmatrix}$ is the right singular vectors of $B$, then 
			\begin{equation}
			|v_{12}| = |v_{21}| \geq \frac{1}{\sqrt{2}}.
			\end{equation}
			\item Suppose 2-by-2 matrix $A$ satisfies
			\begin{equation}
			A = \begin{bmatrix}
			a & b\\
			c & d
			\end{bmatrix},\quad a,b,c,d \geq 0,\quad  a^2 > d^2 + b^2 +  c^2.
			\end{equation}
			Suppose $V = \begin{bmatrix}
			v_{11} & v_{12}\\
			v_{21} & v_{22}
			\end{bmatrix}$ is the right singular vectors of $A$, then 
			\begin{equation*}
			|v_{12}| = |v_{21}| \geq \frac{1}{\sqrt{10}}\left(\frac{ab+cd}{a^2 - d^2 - b^2\wedge c^2} \wedge 1\right).
			\end{equation*}
		\end{enumerate}
	\end{Lemma}
	The proof of Lemma \ref{lm:2-by-2} is provided later. 
	
	\begin{itemize}[leftmargin=*]
		\item Now we consider the situation when $\alpha^2 \leq \beta^2 + z_{12}^2 \wedge z_{21}^2$.  Clearly, $\alpha^2 \leq \beta^2 + z_{12}^2$ under this setting. We can write down the singular value decomposition for the following matrix,
		$$
		\begin{bmatrix}
		\alpha & z_{12}\\
		0  & \beta
		\end{bmatrix} = \begin{bmatrix}
		u_{11} & u_{12}\\
		u_{21} & u_{22}
		\end{bmatrix}\cdot 
		\begin{bmatrix}
		\sigma_1 & 0\\
		0 & \sigma_2
		\end{bmatrix}\cdot 
		\begin{bmatrix}
		v_{11} & v_{12}\\
		v_{21} & v_{22}
		\end{bmatrix}^\intercal,$$
		By the first part of Lemma \ref{lm:2-by-2}, we have 
		\begin{equation}\label{ineq:v_12_case1}
		|v_{12}| = |v_{21}| \geq \frac{1}{\sqrt{2}}.
		\end{equation} We construct the following matrices
		\begin{equation*}
		X_1 =  \begin{blockarray}{cccc}
		& r & r & p_2 - 2r\\
		\begin{block}{c[ccc]}
		r & \sigma_1 u_{11} v_{11} I_{r} & \sigma_1 u_{11} v_{21} I_{r} & 0\\
		r & \sigma_1 u_{21} v_{11} I_{r} & \sigma_1 u_{21} v_{21} I_{r} & 0 \\
		p_1 - 2r& 0 & 0 & 0\\
		\end{block}
		\end{blockarray},\\
		\end{equation*}
		\begin{equation*}
		Z_1 = \begin{blockarray}{cccc}
		& r & r & p_2 - 2r\\
		\begin{block}{c[ccc]}
		r & \sigma_2 u_{12} v_{12} I_{r} & \sigma_2 u_{12} v_{22} I_{r} & 0\\
		r & \sigma_2 u_{22} v_{12} I_{r} & \sigma_2 u_{22} v_{22} I_{r} & 0 \\
		p_1 - 2r& 0 & 0 & 0\\
		\end{block}
		\end{blockarray};
		\end{equation*}
		\begin{equation}
		X_2 =  \begin{blockarray}{cccc}
		& r & r & p_2 - 2r\\
		\begin{block}{c[ccc]}
		r & \alpha I_{r} & 0 & 0\\
		r & 0 & 0 & 0 \\
		p_1 - 2r& 0 & 0 & 0\\
		\end{block}
		\end{blockarray},\quad Z_2 =  \begin{blockarray}{cccc}
		& r & r & p_2 - 2r\\
		\begin{block}{c[ccc]}
		r & 0 & z_{12} I_r & 0\\
		r & 0 & \beta I_r & 0 \\
		p_1 - 2r& 0 & 0 & 0\\
		\end{block}
		\end{blockarray}.
		\end{equation}
		We can see $\rank(X_1) = \rank(X_2) = r$,
		$$X_1 + Z_1 = X_2 +Z_2 = \hat X = \begin{blockarray}{cccc}
		& r & r & p_2 - 2r\\
		\begin{block}{c[ccc]}
		r & \alpha I_{r} & z_{12} I_{r} & 0\\
		r & 0 & \beta I_r & 0 \\
		p_1 - 2r& 0 & 0 & 0\\
		\end{block}
		\end{blockarray}. $$
		It is easy to check that both $(X_1, Z_1)$ and $(X_2, Z_2)$ are both in $\mathcal{F}_{r, \alpha, \beta, z_{12}, z_{21}}$. Assume $V_1$, $V_2$ are the first $r$ singular vectors of $X_1$ and $X_2$, respectively. Based on the structure of $X_1, X_2$, we know
		$$V_1 = \begin{blockarray}{cc}
		\begin{block}{c[c]}
		r & v_{11} I_r\\
		r & v_{21} I_r\\
		p_2 - 2r & 0\\
		\end{block}
		\end{blockarray},\quad V_2 = \begin{blockarray}{cc}
		\begin{block}{c[c]}
		r & I_r\\
		p_2 - r & 0\\
		\end{block}
		\end{blockarray}. $$
		Now based on the observations $\hat X$, for any estimator $\tilde{V}$ for the right singular space, we have
		\begin{equation}
		\begin{split}
		& \max\left\{\|\sin\Theta(\tilde{V}, V_1)\|, \|\sin\Theta(\tilde{V}, V_{2})\|\right\}\\
		\geq & \frac{1}{2} \left(\|\sin\Theta(\tilde{V}, V_1)\| + \|\sin\Theta(\tilde{V}, V_2)\|\right)\\
		\overset{\text{Lemma \ref{lm:sin_Theta_distance}}}{\geq} & \frac{1}{2} \left\|\sin\Theta({V}_1, V_2)\right\| \overset{\eqref{ineq:v_12_case1}}{\geq} \frac{1}{2\sqrt{2}},
		\end{split}
		\end{equation}
		which gives us \eqref{ineq:perturbation_lower_spectral_1}.
		
		\item Next, we consider the situation when $\alpha^2 > \beta^2 + z_{12}^2 + z_{21}^2$. We first assume $\alpha z_{12} \geq \beta z_{21}$. Since $\alpha^2 > \beta^2 + z_{12}^2 + z_{21}^2$, we have
		$$\alpha z_{12} \geq (\alpha z_{12} + \beta z_{21})/2, \quad \alpha^2  - \beta^2 \leq 2(\alpha^2 - \beta^2 - z_{12}^2 \wedge z_{21}^2). $$
		Suppose we have the following singular value decomposition for 2-by-2 matrix
		$$
		\begin{bmatrix}
		\alpha & z_{12}\\
		0 & \beta
		\end{bmatrix} = \begin{bmatrix}
		u_{11} & u_{12}\\
		u_{21} & u_{22}
		\end{bmatrix}\cdot 
		\begin{bmatrix}
		\sigma_1 & 0\\
		0 & \sigma_2
		\end{bmatrix}\cdot 
		\begin{bmatrix}
		v_{11} & v_{12}\\
		v_{21} & v_{22}
		\end{bmatrix}^\intercal,$$
		by Lemma \ref{lm:2-by-2}, we have
		\begin{equation}\label{ineq:v_12,v_21}
		\begin{split}
		|v_{12}| = |v_{21}| \geq & \frac{1}{\sqrt{10}} \left(\frac{\alpha z_{12}}{\alpha^2 - \beta^2} \wedge 1\right)\geq  \frac{1}{\sqrt{10}}\left(\frac{(\alpha z_{12} + \beta z_{21})/2}{2(\alpha^2 - \beta^2 - z_{12}^2\wedge z_{21}^2)}\wedge 1\right)\\
		\geq & \frac{1}{4\sqrt{10}}\left(\frac{\alpha z_{12} + \beta z_{21}}{\alpha^2 - \beta^2 - z_{12}^2\wedge z_{21}^2}\wedge 1\right).\\ 
		\end{split}
		\end{equation}
		We construct the following matrices
		\begin{equation*}
		X_1 =  \begin{blockarray}{cccc}
		& r & r & p_2 - 2r\\
		\begin{block}{c[ccc]}
		r & \sigma_1 u_{11} v_{11} I_{r} & \sigma_1 u_{11} v_{21} I_{r} & 0\\
		r & \sigma_1 u_{21} v_{11} I_{r} & \sigma_1 u_{21} v_{21} I_{r} & 0 \\
		p_1 - 2r& 0 & 0 & 0\\
		\end{block}
		\end{blockarray},\\
		\end{equation*}
		\begin{equation*}
		Z_1 = \begin{blockarray}{cccc}
		& r & r & p_2 - 2r\\
		\begin{block}{c[ccc]}
		r & \sigma_2 u_{12} v_{12} I_{r} & \sigma_2 u_{12} v_{22} I_{r} & 0\\
		r & \sigma_2 u_{22} v_{12} I_{r} & \sigma_2 u_{22} v_{22} I_{r} & 0 \\
		p_1 - 2r& 0 & 0 & 0\\
		\end{block}
		\end{blockarray};
		\end{equation*}
		\begin{equation}
		X_2 =  \begin{blockarray}{cccc}
		& r & r & p_2 - 2r\\
		\begin{block}{c[ccc]}
		r & \alpha I_{r} & 0 & 0\\
		r & 0 & 0 & 0 \\
		p_1 - 2r& 0 & 0 & 0\\
		\end{block}
		\end{blockarray},\quad Z_2 =  \begin{blockarray}{cccc}
		& r & r & p_2 - 2r\\
		\begin{block}{c[ccc]}
		r & 0 & z_{12} I_r & 0\\
		r & 0 & \beta I_r & 0 \\
		p_1 - 2r& 0 & 0 & 0\\
		\end{block}
		\end{blockarray}.
		\end{equation}
		We can see $\rank(X_1) = \rank(X_2) = r$,
		$$X_1 + Z_1 = X_2 +Z_2 = \hat X = \begin{blockarray}{cccc}
		& r & r & p_2 - 2r\\
		\begin{block}{c[ccc]}
		r & \alpha I_{r} & z_{12} I_{r} & 0\\
		r & 0 & \beta I_r & 0 \\
		p_1 - 2r& 0 & 0 & 0\\
		\end{block}
		\end{blockarray}. $$
		It is easy to check that both $(X_1, Z_1)$ and $(X_2, Z_2)$ are both in $\mathcal{F}_{r, \alpha, \beta, z_{12}, z_{21}}$. Assume $V_1$, $V_2$ are the first $r$ singular vectors of $X_1$ and $X_2$, respectively. Based on the structure of $X_1, X_2$, we know
		$$V_1 = \begin{blockarray}{cc}
		\begin{block}{c[c]}
		r & v_{11} I_r\\
		r & v_{21} I_r\\
		p_2 - 2r & 0\\
		\end{block}
		\end{blockarray},\quad V_2 = \begin{blockarray}{cc}
		\begin{block}{c[c]}
		r & I_r\\
		p_2 - r & 0\\
		\end{block}
		\end{blockarray}. $$
		Now based on the observations $\hat X$, for any estimator $\tilde{V}$ for the right singular space, we have
		\begin{equation}
		\begin{split}
		& \max\left\{\|\sin\Theta(\tilde{V}, V_1)\|, \|\sin\Theta(\tilde{V}, V_{2})\|\right\}\\
		\geq & \frac{1}{2} \left(\|\sin\Theta(\tilde{V}, V_1)\| + \|\sin\Theta(\tilde{V}, V_2)\|\right)\\
		\overset{\text{Lemma \ref{lm:sin_Theta_distance}}}{\geq} & \frac{1}{2} \left\|\sin\Theta({V}_1, V_2)\right\| \overset{\eqref{ineq:v_12,v_21}}{\geq} \frac{1}{8\sqrt{10}} \left(\frac{\alpha z_{12} + \beta z_{21}}{\alpha^2 - \beta^2 - z_{12}^2\wedge z_{21}^2}\wedge 1 \right),
		\end{split}
		\end{equation}
		which gives us \eqref{ineq:perturbation_lower_spectral}. 
		
		For the situation where $\alpha z_{12} \leq \beta z_{21}$, we consider the singular decomposition of the 2-by-2 matrix
		$$
		\begin{bmatrix}
		\alpha & 0\\
		z_{21} & \beta
		\end{bmatrix} = \begin{bmatrix}
		u_{11} & u_{12}\\
		u_{21} & u_{22}
		\end{bmatrix}\cdot 
		\begin{bmatrix}
		\sigma_1 & 0\\
		0 & \sigma_2
		\end{bmatrix}\cdot 
		\begin{bmatrix}
		v_{11} & v_{12}\\
		v_{21} & v_{22}
		\end{bmatrix}^\intercal,$$
		we can construct $X_1, Z_1, X_2, Z_2$ similarly as the previous case, and derive \eqref{ineq:perturbation_lower_spectral}.
		
		Then we consider the worst-case matrices. Define 
		$$X_3 = \begin{blockarray}{cccc}
		& r & r & p_2 - 2r\\
		\begin{block}{c[ccc]}
		r & \alpha I_{r} & 0 & 0\\
		r & 0 & 0 & 0 \\
		p_1 - 2r& 0 & 0 & 0\\
		\end{block}
		\end{blockarray},\quad Z_3 =  \begin{blockarray}{cccc}
		& r & r & p_2 - 2r\\
		\begin{block}{c[ccc]}
		r & 0 & z_{12} I_r & 0\\
		r & z_{21}I_r & \beta I_r & 0 \\
		p_1 - 2r& 0 & 0 & 0\\
		\end{block}
		\end{blockarray},
		$$ 
		and consider the singular value decomposition of $X_3 + Z_3$, which is essentially the singular value decomposition of the following 2-by-2 matrix,
		$$
		\begin{bmatrix}
		\alpha & z_{12}\\
		z_{21} & \beta
		\end{bmatrix} = \begin{bmatrix}
		u_{11} & u_{12}\\
		u_{21} & u_{22}
		\end{bmatrix}\cdot 
		\begin{bmatrix}
		\sigma_1 & 0\\
		0 & \sigma_2
		\end{bmatrix}\cdot 
		\begin{bmatrix}
		v_{11} & v_{12}\\
		v_{21} & v_{22}
		\end{bmatrix}^\intercal,$$
		by Lemma \ref{lm:2-by-2}, we have
		\begin{equation}
		\begin{split}
		|v_{12}| = |v_{21}| \geq \frac{1}{\sqrt{10}}\left(\frac{\alpha z_{12} + \beta z_{21}}{\alpha^2 - \beta^2 - z_{12}^2\wedge z_{21}^2}\wedge 1\right).\\ 
		\end{split}
		\end{equation}
		If $V_3$ and $\hat{V}_3$ are the leading $r$ right singular vectors of $X_3$ and $X_3+Z_3$ respectively, then
		$$\left\|\sin\Theta(\hat{V}_3, V_3)\right\| = |v_{12}| \geq \frac{1}{\sqrt{10}} \left(\frac{\alpha z_{12} + \beta z_{21}}{\alpha^2 - \beta^2 - z_{12}^2\wedge z_{21}^2}\wedge 1 \right), $$
		and by upper bound result in Theorem \ref{th:perturbation}, 
		$$\left\|\sin\Theta(\hat{V}_3, V_3)\right\| \leq \frac{\alpha z_{12} + \beta z_{21}}{\alpha^2 - \beta^2 - z_{12}^2\wedge z_{21}^2}\wedge 1.$$

		\item For the proof of the Frobenius norm loss lower bound \eqref{ineq:perturbation_lower_Frobenius}, the construction of pairs $(X_1, Z_1), (X_2, Z_2)$ is essentially the same. We just need to construct similar pairs of $(X_1, Z_1)$ and $(X_2, Z_2)$ by replacing $z_{12}, z_{21}$ by $\tilde{z}_{12}/\sqrt{r}, \tilde{z}_{21}/\sqrt{r}$.
		Based on the similar calculation, we can finish the proof of Theorem \ref{th:perturbation_lower}. 
	\end{itemize}
\end{proof}

\subsection{Proof of Corollary \ref{cr:perturbation_single_singular_vector}} By Theorem \ref{th:perturbation}, we have
\begin{equation}\label{ineq:hat_V_i-1}
\left\|\sin\Theta(\hat{V}_{(i-1)}, V_{(i-1)})\right\| \leq \frac{\alpha^{(i-1)}z_{12}^{(i-1)} + \beta^{(i-1)}z_{21}^{(i-1)}}{(\alpha^{(i-1)})^2 - (\beta^{(i-1)})^2 - (z_{12}^{(i-1)})^2\wedge (z_{21}^{(i-1)})^2}\wedge 1,
\end{equation}
\begin{equation}\label{ineq:hat_V_j}
\left\|\sin\Theta(\hat{V}_{(j)}, V_{(j)})\right\| \leq \frac{\alpha^{(j)}z_{12}^{(j)} + \beta^{(j)}z_{21}^{(j)}}{(\alpha^{(j)})^2 - (\beta^{(j)})^2 - (z_{12}^{(j)})^2\wedge (z_{21}^{(j)})^2}\wedge 1.
\end{equation}
Next, based on Lemma \ref{lm:sin_Theta_distance}, we know
\begin{equation*}
\begin{split}
& \left\|\sin\Theta(\hat{V}_{[:, i:j]}, V_{[:, i:j]})\right\| = \left\|V_{[:, i:j]}^{\intercal}\hat{V}_{[:, i:j]\perp}\right\|\\
= & \left\|\begin{bmatrix}
V_{[:, i:j]}^{\intercal}\hat{V}_{[:, 1:(i-1)]}\\
V_{[:, i:j]}^{\intercal}\hat{V}_{[:, (j+1):p_2]}
\end{bmatrix}\right\| \leq \left(\|V_{[:, i:j]}^{\intercal}\hat{V}_{[:, 1:(i-1)]}\|^2 + \|V_{[:, i:j]}^{\intercal}\hat{V}_{[:, (j+1):p_2]}\|^2\right)^{1/2}\\
\leq & \left(\|V_{[:, i:p_2]}^{\intercal}\hat{V}_{[:, 1:(i-1)]}\|^2 + \|V_{[:, 1:j]}^{\intercal}\hat{V}_{[:, (j+1):p_2]}\|^2\right)^{1/2}\\
= & \left(\|V_{[:, 1:(i-1)]\perp}^{\intercal}\hat{V}_{[:, 1:(i-1)]}\|^2 + \|V_{[:, 1:j]}^{\intercal}\hat{V}_{[:, 1:j]\perp}\|^2\right)^{1/2}\\
= & \left(\|\sin\Theta(\hat{V}_{(i-1)}, V_{(i-1)})\|^2 + \|\sin\Theta(\hat{V}_{(j)}, V_{(j)})\|^2\right)^{1/2}
\end{split}
\end{equation*}
Particularly when $j=i$, by Lemma \ref{lm:sin_Theta_distance} we have $\|\sin\Theta(v_i, \hat{v}_i)\| = \sqrt{1 - (\hat{v}_i^\intercal v_i)^2}.$ Combining \eqref{ineq:hat_V_i-1}, \eqref{ineq:hat_V_j}, and the inequality above, we have finished the proof for this corollary.

\subsection{Proofs for Matrix Denoising}

We prove all results in Section \ref{sec.denoising} in this section.

\begin{proof}[Proof of Theorem \ref{th:denoising}] 
	We need some technical results for the proof of this theorem. Specifically, the following lemma relating to random matrix theory plays an important part.
	\begin{Lemma}[Properties related to Random Matrix Theory]\label{lm:denoising}
		Suppose $X \in \mathbb{R}^{p_1 \times p_2}$ is a rank-$r$ matrix with right singular space as $V\in \mathbb{O}_{n, r}$, $Z\in \mathbb{R}^{p_1\times p_2}, Z\overset{iid}{\sim} \mathcal{G}_{\tau}$ is an i.i.d. sub-Gaussian random matrix. $Y = X+Z$. Then there exists constants $C, c$ only depending on $\tau$ such that for any $x>0$,
		\begin{equation}\label{ineq:sigma(YV)}
		\P\left(\sigma_r^2(YV) \geq (\sigma_r^2(X) + p_1)(1-x)\right) \leq C\exp\left\{Cr - c\left(\sigma_r^2(X)+p_1\right)x^2\wedge x\right\}, 
		\end{equation}
		\begin{equation}\label{ineq:sigma(Y)}
		\P\left(\sigma_{r+1}^2(Y) \leq p_1(1+x)\right) \leq C\exp\left\{Cp_2 - cp_1\cdot x^2\wedge x\right\}.
		\end{equation} 
		Moreover, there exists $C_{\rm gap}, C, c$ which only depends on $\tau$, such that whenever $\sigma_r^2(X) \geq C_{\rm gap}p_2$, for any $x>0$ we have
		\begin{equation}\label{ineq:P_YV(YV_perp)}
		\begin{split}
		&\P\left(\left\|\bbP_{YV}YV_{\perp}\right\| \leq x\right)\\
		\geq & 1 - C\exp\left\{Cp_2 - c\min(x^2, \sqrt{\sigma_r^2(X)+p_1}x)\right\} - C\exp\left\{-c(\sigma_r^2(X)+p_1)\right\}.
		\end{split}
		\end{equation}
	\end{Lemma}

	The following lemma provides an upper bound for matrix spectral norm based on an $\varepsilon$-net argument.
	\begin{Lemma}[$\varepsilon$-net Argument for Unit Ball]\label{lm:epsilon_net}
		For any $p\geq1$, denote $\mathbb{B}^{p} = \{x: x\in \mathbb{R}^{p}, \|x\|_2\leq 1\}$ as the $p$-dimensional unit ball. Suppose $K\in \mathbb{R}^{p_1\times p_2}$ is a random matrix. Then we have for $t > 0$,
		\begin{equation}
		\P\left(\left\|K\right\| \geq 3t \right) \leq 7^{p_1+p_2}\cdot \max_{u \in \mathbb{B}^{p_1}, v\in \mathbb{B}^{p_2}} \P\left(\left|u^{\intercal}K v\right| \geq t\right).
		\end{equation}
	\end{Lemma}
	
	\ \par
	
	Now we start to prove Theorem \ref{th:denoising}. We only need to focus on the losses of $\hat{V}$ since the results for $\hat{U}$ are symmetric. Besides, we only need to prove the spectral norm loss, as $\sin\Theta(\hat{V}, V)$ is a $r\times r$ matrix $\|\sin\Theta(V, \hat{V})\|_F^2\leq r\|\sin\Theta(V, \hat{V})\|^2$. Throughout the proof, we use $C$ and $c$ to represent generic ``large" and ``small" constants, respectively. These constants $C, c$ are uniform and only relying on $\tau$, while the actual values may vary in different formulas. 
	
	Next, we focus on the scenario that $\sigma_r^2(X)\geq C_{\rm gap}((p_1p_2)^{1\over 2} + p_2)$ for some large constant $C_{\rm gap}>0$ only relying on $\tau$. The other case will be considered later. By Lemma \ref{lm:denoising}, there exists constants $C, c$ only depending on $\tau$ such that
	\begin{equation*}
	\begin{split}
	& \P\left(\sigma_r^2(YV) \leq \sigma_r^2(X) + p_1 - \frac{\sigma_r^2(X)}{3}\right)\\
	\leq & C\exp\left\{Cr - c \min\left(\sigma_r^2(X), \frac{\sigma_r^4(X)}{\sigma_r^2(X) + p_1}\right)\right\}, 
	\end{split}
	\end{equation*}
	\begin{equation*}
	\begin{split}
	\P\left(\sigma_{r+1}^2(Y) \geq p_1 + \frac{1}{3}\sigma_r^2(X)\right) \leq C\exp\left\{Cp_2 - c\min\left(\sigma_r^2(X), \frac{\sigma_r^4(X)}{p_1}\right)\right\},
	\end{split}
	\end{equation*}
	\begin{equation}\label{ineq:tail_probability_PYVYV_perp}
	\begin{split}
	& \P\left\{\left\|\bbP_{YV}YV_{\perp}\right\|\geq x\right\}\\ 
	\leq & C\exp\left\{Cp_2 - c\min\left(x^2, x\sqrt{\sigma_r^2(X)+p_1}\right)\right\} + C\exp\left\{-c(\sigma_r^2(X)+p_1)\right\}.
	\end{split}
	\end{equation}
	When $C_{\rm gap}$ is large enough, it holds that $\sigma_r^4(X)/(\sigma_r^2(X)+p_1) \geq Cp_2$. 
	Then 
	$$ c\min\left(\frac{\sigma_r^4(X)}{\sigma_r^2(X) + p_1}, \sigma_r^2(X)\right) - Cr \geq \frac{\sigma_r^4(X)}{\sigma_r^2(X)+p_1} - Cr \geq \frac{c}{2}\frac{\sigma_r^4(X)}{\sigma_r^2(X)+p_1}; $$
	\begin{equation*}
	\begin{split}
	c\min\left(\sigma_r^2(X), \frac{\sigma_r^4(X)}{p_1}\right) - Cp_2 \geq c\frac{\sigma_r^4(X)}{\sigma_r^2(X)+p_1} - Cp_2 \geq \frac{c}{2} \frac{\sigma_r^4(x)}{\sigma_r^2(X)+p_1}.
	\end{split}
	\end{equation*}
	\begin{equation*}
	\begin{split}
	& c\min(x^2, \sqrt{\sigma_r^2(X)+p_1}x) - Cp_2 = c(\sigma_r^2(X) + p_1) - Cp_2 \geq \frac{c}{2}\left(\sigma_r^2(X)+p_1\right),\\
	\geq & \frac{c}{2}\frac{\sigma_r^4(X)}{\sigma_r^2(X)+p_1}\quad  \text{if $x=\sqrt{\sigma_r^2(X)+p_1}$.}
	\end{split}
	\end{equation*}
	
	To sum up, if we denote the event $Q$ as
	\begin{equation*}
	\begin{split}
	Q = \Big\{ & \sigma_r^2(YV) \geq \sigma_r^2(X) + p_1 - \frac{\sigma_r^2(X)}{3},
	\sigma_{r+1}^2(Y) \leq p_1 + \frac{1}{3}\sigma_r^2(X),\\
	& \left\|\bbP_{YV}YV_{\perp}\right\|\leq \sqrt{\sigma_r^2(X)+p_1}\Big\},
	\end{split}
	\end{equation*} 
	when $\sigma_r^2(X)\geq C_{\rm gap}((p_1p_2)^{1\over 2} + p_2)$ for some large constant $C_{\rm gap}>0$,
	\begin{equation}\label{ineq:th_denoising_Q}
	\begin{split}
	& \P\left(Q^c\right)\leq C\exp\left\{-c\frac{\sigma_r^4(X)}{\sigma_r^2(X)+p_1}\right\}.
	\end{split}
	\end{equation}
	Under event $Q$, we can apply Proposition \ref{lm:perturbation} and obtain
	\begin{equation*}
	\|\sin\Theta(\hat{V}, V)\|^2 \leq \frac{\sigma_r^2(YV)\|\bbP_{YV}YV_{\perp}\|^2}{(\sigma_r^2(YV) - \sigma_{r+1}(Y))^2} \leq C\frac{(\sigma_r^2(X)+p_1)\|\bbP_{YV}YV_{\perp}\|^2}{\sigma_r^4(X)}. 
	\end{equation*}
	Here we used the fact that $\frac{x^2}{(x^2-y^2)^2}$ is a decreasing function of $x$ and increasing function of $y$ when $x>y\geq 0$. 
	
	Next we shall note that $\|\sin\Theta(\hat{V}, V)\| \leq 1$ for any $\hat{V}, V \in \mathbb{O}_{p_2, r}$. Therefore,
	\begin{equation}\label{ineq:th_denoising_risk_decompose}
	\begin{split}
	& \E\left\|\sin\Theta(\hat{V}, V)\right\|^2 = \E\left\|\sin\Theta(\hat{V}, V)\right\|^21_Q + \E\left\|\sin\Theta(\hat{V}, V)\right\|^21_{Q^c}\\
	\leq & \frac{C(\sigma_r^2(X)+p_1)}{\sigma_r^4(X)}\E\|\bbP_{YV}YV_{\perp}\|^21_Q + \P(Q^c).
	\end{split}
	\end{equation}
	By basic property of exponential function, 
	\begin{equation}\label{ineq:th_denoising_term1}
	\P(Q^c) \overset{\eqref{ineq:th_denoising_Q}}{\leq} \exp\left\{-c\frac{\sigma_r^4(X)}{\sigma_r^2(X)+p_1}\right\} \leq C\frac{\sigma_r^2(X)+p_1}{\sigma_r^4(X)} \leq C\frac{p_2(\sigma_r^2(X)+p_1)}{\sigma_r^4(X)}.
	\end{equation}
	It remains to consider $\E\left\|\bbP_{YV}YV_{\perp}\right\|^21_Q$. Denote $T = \left\|\bbP_{YV}YV_{\perp}\right\|$. Applying Lemma \ref{lm:denoising} again, we have for some constant $C_x$ to be determined a little while later that
	\begin{equation*}
	\begin{split}
	\E T^21_{Q} \leq & \E T^2 1_{\{T^2 \leq \sigma_r^2(X)+p_1\}} = \int_{0}^\infty \P\left(T^2 1_{\{T^2 \leq \sigma_r^2(X)+p_1\}}\geq t\right)dt \\
	\leq & C_xp_2 + \int_{C_xp_2}^{\sigma_r^2(X)+p_1} \P\left(T^21_{\{T^2 \leq \sigma_r^2(X)+p_1\}}\geq t\right)dt\\
	\overset{\eqref{ineq:tail_probability_PYVYV_perp}}{\leq} & C_xp_2 + \int_{C_xp_2}^{\sigma_r^2(X)+p_1} C\left(\exp\left\{Cp_2 - ct\right\}dt + \exp\left\{-c(\sigma_r^2(X)+p_1)\right\}\right)dt\\
	\leq & C_xp_2 + C(\sigma_r^2(X)+p_1)\exp\left\{-c(\sigma_r^2(X)+p_1)\right\}\\
	& + C\exp(Cp_2)\cdot \exp\left(-cC_xp_2\right)\frac{1}{c}\\
	\leq & C_xp_2 + C + \frac{C}{c}\exp((C - cC_x)p_2).
	\end{split}
	\end{equation*}
	As we could see we can choose $C_x$ large enough, but only relying on other constants $C,c$ in the inequalities above, to ensure that
	\begin{equation}\label{ineq:th_denoising_term2}
	\E T^21_Q \leq Cp_2
	\end{equation}
	for large constant $C>0$ as long as $p_2 \geq 1$. Now, combining \eqref{ineq:th_denoising_risk_decompose}, \eqref{ineq:th_denoising_term1}, \eqref{ineq:th_denoising_term2} as well as the trivial upper bound 1, we obtain
	$$\E\left\|\sin\Theta(\hat{V}, V)\right\|^2 \leq \frac{Cp_2(\sigma_r^2(X)+p_1)}{\sigma_r^4(X)} \wedge 1. $$
	as long as $\sigma_r^2(X)\geq C_{\rm gap}\left((p_1p_2)^{1\over 2} + p_2\right)$ for some large enough $C_{\rm gap}$. 
	
	Finally when $\sigma_r^2(X)<C_{\rm gap}\left((p_1p_2)^{1\over 2}+p_2\right)$, we have 
	\begin{equation}
	\begin{split}
	& \frac{p_2(\sigma_r^2(X)+p_1)}{\sigma_r^4(X)} \geq \frac{p_2(p_1+C_{\rm gap}(p_1p_2)^{1\over 2}+C_{\rm gap}p_2)}{C_{\rm gap}^2(p_1p_2 + p_2^2 + 2p_1^{1\over 2}p_2^{3/2})}\\
	= & \frac{p_1 + C_{\rm gap}(p_1p_2)^{1\over 2} + C_{\rm gap}p_2}{C_{\rm gap}^2\left(p_1+2(p_1p_2)^{1\over 2} + p_2\right)} \geq \min\left(1, \frac{1}{C_{\rm gap}}\right),
	\end{split}
	\end{equation}
	then
	$$\E\left\|\sin\theta(\hat{V}, V)\right\|^2 \leq 1 \leq \frac{Cp_2(\sigma_r^2(X)+p_1)}{\sigma_r^4(X)}\wedge 1$$
	when $C \geq \min^{-1}\left(1, 1/C_{\rm gap}\right)$. In summary, no matter what value $\sigma_r^2(X)$ takes, we always have
	$$\E\left\|\sin\Theta(\hat{V}, V)\right\|^2 \leq \frac{Cp_2(\sigma_r^2(X)+p_1)}{\sigma_r^4(X)}\wedge 1. $$

\end{proof}

\ \par

\begin{proof}[Proof of Theorem \ref{th:denoising_lower}.] Since $\|\sin\Theta(\hat{V}, V)\| \geq \frac{1}{\sqrt{r}}\|\sin\Theta(\hat{V}, V)\|_F$, we only need to prove the Frobenius norm lower bound. Particularly we focus on the case with Gaussian noise with mean 0 and variance 1, i.e. $Z_{ij}\overset{iid}{\sim}N(0, 1)$. The technical tool we use to develop such lower bound include the generalized Fano's Lemma.
	
	One interesting aspect of this singular space estimation problem is that, multiple sampling distribution $P_X$ correspond to one target parameter $V$. In order to proceed, we select one ``representative", either a distribution $P_V$ or a mixture distribution $\bar{P}_{V, t}$ for each $V$, and consider the estimation with samples from ``representative" distribution. In order to bridge the representative estimation and the original estimation problem, we introduce the following lower bound lemma. 
	\begin{Lemma}\label{lm:lower_bound_transformation}
		Let $\{P_{\theta}: \theta \in \Theta\}$ be a set of probability measures in Euclidean space $\mathbb{R}^p$, let $T: \Theta \to \Phi$ be a function which maps $\theta$ into another metric space $\phi \in (\Phi, d)$. For each $\phi \in \Phi$, denote $co(\mathcal{P}_{\phi})$ as the convex hull of the set of probability measures $\mathcal{P}_{\phi} = \{P_\theta: T(\theta) = \phi\}$. If we choose a representative $P_{\phi}$ in each $co(P_{\phi})$, then for any estimator $\hat{\phi}$ based on the sample generated from $P_{\theta}$ with $\phi = T(\theta)$, and any estimator $\hat{\phi}'$ based on the samples generated from $P_{\theta_{\phi}}$, we have
		\begin{equation}
		\inf_{\hat{\phi}}\sup_{\substack{\theta \in \Theta}} \E_{P_{\theta}}[d^2(T(\theta), \hat{\phi})] \geq 	\inf_{\hat{\phi}}\sup_{\phi \in \Phi} \E_{P_{\phi}}[d^2(\phi, \hat{\phi})]. 
		\end{equation}
	\end{Lemma}
	We will prove this theorem separately in two cases according to $t$: $t^2\leq p_1/4 $ or $t^2 > p_1/4$. 
	\begin{itemize}
		\item First we consider when $t^2 \leq p_1/4$. For each $V\in \mathbb{O}_{p_2, r}$, we define the following class of density $P_Y$, where $Y = X+Z$ and the right singular space of $X$ is $V$.
		\begin{equation}\label{eq:mathcal_P_V,t}
		\hskip-5mm \mathcal{P}_{V, t} = \left\{P_{Y}: \begin{array}{ll}
		Y\in \mathbb{R}^{p_1\times p_2}, Y= X+Z,  X \text{ is fixed},  Z\overset{iid}{\sim} N(0, 1),\\ X\in \mathcal{F}_{r, t}, \text{the right singular vectors of $X$ is } VO, O\in \mathbb{O}_{r}
		\end{array} \right\}.
		\end{equation}
		For each $V\in \mathbb{O}_{p_2, r}$, we construct the following Gaussian mixture measure $\bar{P}_{V, t} \in \mathbb{R}^{p_1\times p_2}$.
		\begin{equation}\label{eq:bar P_Vt}
		\begin{split}
		\bar{P}_{V, t}(Y) = C_{V, t}\int_{W\in \mathbb{R}^{p_1\times r}: \sigma_{\min}(W) \geq 1/2} &\frac{1}{(2\pi)^{p_1p_2/2}}\exp(-\|Y - 2t WV^{\intercal}\|_F^2/2 )\\
		& \cdot (\frac{p_1}{2\pi})^{p_1r/2}\exp(-p_1\|W\|_F^2/2) dW.
		\end{split}
		\end{equation}
		Here $C_{V, t}$ is the constant which normalizes the integral and makes $\bar{P}_{V, t}$ a valid probability density. To be specific
		$$C_{V, t}^{-1} = \int_Y P_{V, t}(Y)dY = \P\left(\sigma_{\min}(W)\geq 1/2\Big| W\in \mathbb{R}^{p_1\times r}, W\overset{iid}{\sim}N(0, 1/p_1)\right). $$
		Moreover, since $2tWV^{\intercal}$ is rank-$r$ with the least singular value no less than $t$ in the event that $\sigma_{\min}(W) \geq 1/2$, $\bar{P}_{V, t}$ is a mixture density of $\mathcal{P}_{V, t}$, 
		i.e. $\bar{P}_{V, t} \in co (\mathcal{P}_{V, t})$. 
		
		The following lemma, whose proof is provided in the supplementary materials, gives a upper bound for the KL divergence between  $\bar{P}_{V, t}$ and $\bar{P}_{V', t}$.
		\begin{Lemma}\label{lm:D_KL_upper}
			Under the assumption of Theorem \ref{th:denoising_lower} and $t^2 \leq p_1/4$, for any $V, V' \in \mathbb{O}_{p_2, r}$, we have
			\begin{equation}\label{eq:D_KL_upper}
			D(\bar P_{V, t}|| \bar P_{V', t}) \leq \frac{8t^4}{4t^2+p_1}\|\sin\Theta(V, V')\|_F^2 + C_{KL},
			\end{equation}
			where $C_{KL}$ is some uniform constant.
		\end{Lemma}

		We then consider the metric $\left(\mathbb{O}_{p_2, r}, \|\sin\Theta(\cdot, \cdot)\|_F\right)$. 
		Especially we consider the ball with radius $0 < \varepsilon<\sqrt{2r}$ and center $V\in\mathbb{O}_{p_2, r}$
		$$ B(V, \varepsilon) = \left\{V'\in \mathbb{O}_{p_2, r}: \left\|\sin\Theta(V', V)\right\|_F\leq \varepsilon\right\}. $$
		Note that $r \leq p_2/2$, $\|\sin\Theta(V', V)\|_F = \frac{1}{\sqrt{2}} \|V'V^\intercal - VV^\intercal\|_F$, based on Lemma 1 in \cite{cai2013sparse}, one can show for any $\alpha\in (0, 1)$, $\varepsilon \in (0, \sqrt{2r})$, 
		there exists $V_1,\cdots, V_m \subseteq B(V, \varepsilon)$, such that
		$$m\geq \left(\frac{c}{\alpha} \right)^{r(p_2 - r)}, \quad \min_{1\leq i\neq j\leq m} \|\sin\Theta(V_i, V_j)\|_F \geq \alpha \varepsilon.$$ 
		By $\{V_1,\cdots, V_m\}\subseteq B(V, \varepsilon)$, $\|\sin\Theta(V_i, V_j)\| \leq 2\varepsilon$, along with \eqref{eq:D_KL_upper},
		$$D\left(\bar{P}_{V_i, t}||\bar{P}_{V_j, t}\right) \leq \frac{32\varepsilon^2t^4}{4t^2 + p_1} + C_{KL}.$$
		Fano's Lemma \citep{yu1997assouad} leads to the following lower bound
		\begin{equation}
		\inf_{\hat{V}} \sup_{V \in \Theta} \E_{\bar{P}_{V, t}} \left\|\sin\Theta(\hat{V}, V)\right\|_F^2 \geq \alpha^2 \varepsilon^2 \left(1 - \frac{\frac{32\varepsilon^2t^4}{4t^2 + p_1} + C_{KL} +\log 2}{r(p_2-r)\log \frac{c}{\alpha}}\right).
		\end{equation}
		We particularly select 
		$$\alpha = c\exp(-(1+C_{KL}  + \log(2))), \varepsilon = \sqrt{\frac{r(p_2-r)(4t^2 + p_1)}{32t^4}} \wedge \sqrt{2r}.$$ 
		Note that $r \leq p_2/2$, we further have
		\begin{equation}
		\inf_{\hat{V}} \sup_{V \in \mathbb{O}_{p_2, r}} \E_{\bar{P}_{V, t}} \left\|\sin\Theta(\hat{V}, V)\right\|_F^2 \geq c\left(\frac{rp_2 (t^2 + p_1)}{t^4} \wedge r\right).
		\end{equation}
		Finally, note that $\bar{P}_{V,t}$ is a mixture distribution from $\mathcal{P}_{V, t}$ defined in \eqref{eq:mathcal_P_V,t}. Lemma \ref{lm:lower_bound_transformation} implies
		\begin{equation}
		\inf_{\hat{V}} \sup_{X\in \mathcal{F}_{r, t}} \E\left\|\sin\Theta(\hat{V}, V)\right\|_F^2 \geq \inf_{\hat{V}} \sup_{V \in \Theta} \E_{\bar{P}_{V, t}} \left\|\sin\Theta(\hat{V}, V)\right\|_F^2.
		\end{equation}
		The two inequalities above together imply the desired lower bound.
		
		\item Then we consider when $ t^2 > p_1/4$. This case is simpler than the previous as we do not have to mix the multivariate Gaussian measures. Suppose 
		$$U_0 = \begin{bmatrix}
		I_r\\
		0
		\end{bmatrix} \in \mathbb{O}_{p_1, r}. $$
		We introduce
		\begin{equation}
		X_V = t U_0 V^{\intercal}, \quad V \in \mathbb{O}_{p_2, r},
		\end{equation}
		and denote $P_V$ as the probability measure of $Y$ when $Y = X_V+Z, Z\in \mathbb{R}^{p_1\times p_2}$, $Z\overset{iid}{\sim} N(0, 1)$. Based on \eqref{eq:KL_Gaussian}, we have
		\begin{equation}\label{eq:D(P_V||P_V')}
		D(P_V|| P_{V'}) = \frac{1}{2}\left\|tU_0V^{\intercal} - tU_0(V')^{\intercal}\right\|_F^2 =  \frac{t^2}{2}\left\|V - V'\right\|_F^2. 
		\end{equation}
		Based on the same procedure Step 5 in the case $t^2<p_1/4$, one can construct the ball of radius $\varepsilon$ centered at $V_0\in \mathbb{O}_{p_2, r}$,
		$$B(V_0, \varepsilon) = \left\{V': \|\sin\Theta(V', V_0)\|_F \leq \varepsilon\right\}. $$
		and for $0<\alpha<1$, there exist $\{V_1',\cdots, V_m'\}\subseteq B(V_0, \varepsilon)$ such that
		$$m\geq \left(\frac{c}{\alpha}\right)^{r(p_2-r)}, \quad \max_{1\leq i\neq j\leq m}\left\|\sin\theta(V_i', V_j') \right\|_F \geq \varepsilon\alpha. $$
		By basic property of $\sin \Theta$ distances (Lemma \ref{lm:sin_Theta_distance} in this paper), we can find $O_i\in\mathbb{O}_r$ such that
		$$\left\|V_0 - V_i'O_i\right\|_F \leq \sqrt{2}\left\|\sin\Theta(V_0, V_i')\right\|\leq \sqrt{2}\varepsilon $$
		Denote $V_i = V_i' O_i \in \mathbb{O}_{p_2, r}$, then
		\begin{equation*}
		\begin{split}
		& \max_{1\leq i\neq j\leq m} D\left(P_{V_i}|| P_{V_j}\right) \overset{\eqref{eq:D(P_V||P_V')}}{=}\max_{1\leq i\neq j\leq m} \frac{t^2}{2}\|V_i - V_j\|_F^2\\
		\leq & \frac{t^2}{2} \max_{1\leq i\neq j\leq m} 2\left(\left\|V_0 - V_i\right\|_F^2 + \left\|V_0 - V_i\right\|_F^2\right) \leq 4t^2\varepsilon^2.
		\end{split}
		\end{equation*}
		Follow the same procedure as Step 5 in the case $t^2<p_1/4$, we have
		\begin{equation}
		\inf_{\hat{V}} \sup_{V \in \mathbb{O}_{p_2,r}} \E_{P_{V}} \left\|\sin\Theta(\hat{V}, V)\right\|_F^2 \geq \alpha^2 \varepsilon^2 \left(1 - \frac{4\varepsilon^2t^2 +\log 2}{r(p_2-r)\log \frac{c}{\alpha}}\right).
		\end{equation}
		for any $\alpha \in (0, 1)$, $\varepsilon < \sqrt{2r}$. By selecting $\alpha = \frac{c}{4}$,  $\varepsilon^2 = \sqrt{\frac{r(p_2-r)}{2t^2}}\wedge \sqrt{2r}$, we have
		\begin{equation*}
		\inf_{\hat{V}} \sup_{V \in \Theta} \E_{\bar{P}_{V}} \left\|\sin\Theta(\hat{V}, V)\right\|_F^2 \geq c\left(\frac{p_2r}{t^2}\wedge r\right).
		\end{equation*}
		Based on the assumption that $t^2 > p_1/4$, 
		$$\frac{p_2r}{t^2} \geq \frac{p_2r(t^2/2 + p_1/8)}{t^4} \geq \frac{1}{8}\left(\frac{p_2r(t^2+p_1)}{t^4}\right). $$
		Similarly by Lemma \ref{lm:lower_bound_transformation}, we have
		\begin{equation*}
		\inf_{\hat{V}} \sup_{X \in \mathcal{F}_{r, t}} \E \left\|\sin\Theta(\hat{V}, V)\right\|_F^2 \geq c\left(\frac{p_2r(t^2+p_1)}{t^4}\wedge r\right).
		\end{equation*}
	\end{itemize}
	Finally, combining these two cases for $t^2<p_1/4$ and $t^2>p_1/4$, we have finished the proof of Theorem \ref{th:denoising_lower}. 
\end{proof}

\subsection{Proofs in High-dimensional Clustering}

\ \par

\begin{proof}[Proof of Theorem \ref{th:cluster_upper}] 
	Since $EY = \mu l^\intercal = \sqrt{n}\|\mu\|_2\cdot \frac{l}{\sqrt{n}}$, by Theorem \ref{th:denoising}, we know
	\begin{equation*}
	\E\left(\left\|\sin\Theta\left(\frac{l}{\sqrt{n}}, \hat{v}\right)\right\|^2\right) \leq \frac{Cn\left((\sqrt{n}\|\mu\|_2)^2 + p\right)}{(\sqrt{n}\|\mu\|_2)^4} \leq \frac{C(n\|\mu\|_2^2 + p)}{n\|\mu\|_2^4}.
	\end{equation*}
	In addition, the third part of Lemma \ref{lm:sin_Theta_distance} implies
	$$\E \left(\min_{o\in \{-1, 1\}} \|\hat{v} - o \cdot l/\sqrt{n} \|_2^2\right) \leq 2\E\left(\left\|\sin\Theta\left(\frac{l}{\sqrt{n}}, \hat{v}\right)\right\|^2\Big| l\right) \leq \frac{C(n\|\mu\|_2^2 + p)}{n\|\mu\|_2^4}. $$
	Finally, by definition that $\hat{l} = {\rm sgn}(\hat{v})$ and $l_i/\sqrt{n} = \pm 1/\sqrt{n}$, we can obtain
	\begin{equation*}
	\begin{split}
	& \frac{1}{n}\E\min_{\pi} \left|\left\{i: l_i \neq \pi(\hat{l}_i)\right\}\right| \leq \frac{1}{n}\E\min_{\pi} \left|\left\{i: \frac{l_i}{\sqrt{n}} \neq  \pi({\rm sgn}(\hat{v}_i))/\sqrt{n}\right\}\right|\\
	\leq & \frac{1}{n}\E\min_{o\in \{-1, 1\}}\left|\left\{i: \left|\frac{l_i}{\sqrt{n}} - o\hat{v}_i\right| \geq 1/\sqrt{n}\right\}\right| \leq \frac{1}{n} \E \min_{o\in \{-1, 1\}} \sum_{i=1}^n \left(l_i/\sqrt{n} - o\hat{v}_i\right)^2n\\
	= & \E \left(\min_{o\in \{-1, 1\}} \|\hat{v} - o \cdot l/\sqrt{n} \|_2^2 \right) \leq \frac{C(n\|\mu\|_2^2 + p)}{n\|\mu\|_2^4}.
	\end{split}
	\end{equation*}
	which has finished the proof of this theorem.
\end{proof}

\ \par

\begin{proof}[Proof of Theorem \ref{th:cluster_lower}] We first consider the metric space $\{-1, 1\}^n$, where each pair of $l$ and $-l$ are considered as the same element. For any $l_1, l_2 \in \{-1, 1\}^n$, it is easy to see that
	$$\mathcal{M}(l_1, l_2) = \frac{1}{2n} \min\{\|l_1 - l_2\|_1, \|l_1+l_2\|_1\}, $$
	where $\mathcal{M}$ is defined in \eqref{eq:misclassification_rate}. For any three elements $l_1, l_2, l_3 \in \{-1, 1\}^n$, since
	$$\|l_1 - l_2\|_1 \leq \min\left\{\|l_1 - l_3\|_1 + \|l_3 - l_2\|_1, \|l_1 + l_3\|_1 + \|l_3 + l_2\|_1 \right\}, $$
	$$\|l_1 + l_2\|_1 \leq \min\left\{\|l_1 - l_3\|_1 + \|l_3 + l_2\|_1, \|l_1 + l_3\|_1 + \|l_3 - l_2\|_1 \right\}, $$
	we have
	\begin{equation*}
	\begin{split}
	& \mathcal{M}(l_1, l_2) = \frac{1}{2n} \min\left\{\|l_1 - l_2\|_1, \|l_1 + l_2\|_1\right\}\leq\\
	\leq & \frac{1}{2n} \min\{\|l_1 - l_3\|_1 + \|l_3 - l_2\|_1, \|l_1 + l_3\|_1 + \|l_3 + l_2\|_1,\\
	& + \|l_1 - l_3\|_1 + \|l_3 + l_2\|_1, \|l_1 + l_3\|_1 + \|l_3 - l_2\|_1\} = \mathcal{M}(l_1, l_3 ) +\mathcal{M}(l_2, l_3).
	\end{split}
	\end{equation*}
	In other words, $\mathcal{M}$ is a metric in the space of $\{-1, 1\}^n$. Similarly to Lemma 4 in \cite{yu1997assouad}, we can show that there exists universal positive constant $c_0, C_n'$, such that when $n\geq C_n'$, there exists a subset $A$ of $\{-1, + 1\}^n$ satisfying
	\begin{equation}\label{eq:property_A_net}
	|A| \geq \exp(c_0n), \text{ and } \forall l, l' \in \{-1, + 1\}, l\neq l', \text{ we have } \mathcal{M}(l, l') \geq 1/3.
	\end{equation}
	Let $C_{KL}$ be the uniform constant in Lemma \ref{lm:D_KL_upper}. Now, we set $c_{gap} = (c_0/300)^{1\over 4}, C_n = \frac{36}{c_0\log 2} \vee \frac{12C_{KL}}{c_0} \vee C_n'$ and denote $t^2 = c_{gap} (p/n)^{1\over 4}$. Suppose $n\geq C_n$. Next, we need to discuss separately according to the value of $t^2$. 
	\begin{itemize}[leftmargin=*]
		\item When $nt^2 \leq  p/4$, for each $l\in \{-1, 1\}^n$, similarly to the proof to Theorem \ref{th:cluster_lower}, we define the following class of density $P_Y$, where $Y = X+Z$ and the right singular space of $X$ is $l/\sqrt{n}$,
		\begin{equation*}
		\begin{split}
		\mathcal{P}_{l, t} = \Big\{P_{Y}:
		& Y\in \mathbb{R}^{p\times n}, Y= X+Z,  X \text{ is fixed}, \\
		& Z\overset{iid}{\sim} N(0, 1), X = \mu l^\intercal, \|\mu\|_2 \geq t \Big\}.
		\end{split}
		\end{equation*}
		We construct the following Gaussian mixture measure $\bar{P}_{l, t}$,
		\begin{equation}
		\begin{split}
		\bar{P}_{l, t}(Y) = C_{l, t}\int_{\mu_0\in \mathbb{R}^{p}: \|\mu_0\| \geq 1/2} &\frac{1}{(2\pi)^{pn/2}}\exp(-\|Y - 2t \mu_0l^{\intercal}\|_F^2/2 )\\
		& \cdot \left(\frac{p}{2\pi}\right)^{p_1r/2}\exp(-p\|\mu_0\|_2^2/2) d\mu_0.
		\end{split}
		\end{equation}
		Here $C_{l, t}$ is the constant which normalize the integral and make $\bar{P}_{l, t}$ a valid probability density. To be specific, $C_{l, t}^{-1} = \int_Y P_{l, t}(Y)dY. $
		Moreover, since $\|2t\mu_0\| \geq t$ in the event that $\|\mu_0\|_2 \geq 1/2$, $\bar{P}_{l, t}$ is a mixture density in $\mathcal{P}_{l, t}$, i.e. $\bar{P}_{l, t} \in co (\mathcal{P}_{l, t})$. By Lemma \ref{lm:D_KL_upper}, for any two $l, l'\in \{-1, 1\}^n$, the KL-divergence between $\bar{P}_{l, t}$ and $\bar{P}_{l', t}$ have the following upper bound
		\begin{equation*}
		\begin{split}
		& D(\bar{P}_{l, t}|| \bar{P}_{l', t}) \leq \frac{8n^2t^4}{4nt^2 + p}\left\|\sin\Theta(l/\sqrt{n}, l'/\sqrt{n})\right\|_F^2 + C_{KL}\\
		\leq & \frac{8n^2t^4}{4nt^2+p} + C_{KL} \leq \frac{8n^2c_{gap}^4 (p/n)}{p} + C_{KL} = 8nc_{gap}^4 + C_{KL}
		\end{split}
		\end{equation*}
		Then the generalized Fano's lemma (Lemma 3 in \cite{yu1997assouad}) along with inequality \eqref{eq:property_A_net} and Lemma \ref{lm:lower_bound_transformation} lead to
		\begin{equation}
		\begin{split}
		& \inf_{\hat{l}}\max_{\substack{\|\mu\|_2 \leq c_{gap}(p/n)^{1 \over 4}\\ l \in \{-1, 1\}^n}} \E\mathcal{M}(\hat{l}, l) \geq \inf_{\hat{l}}\max_{\substack{\|\mu\|_2 \leq c_{gap}(p/n)^{1 \over 4}\\ l \in A}} \E\mathcal{M}(\hat{l}, l)\\
		\overset{\eqref{eq:property_A_net}}{\geq} & \frac{1}{6}\left(1 - \frac{8nc_{gap}^4 + C_{KL} + \log 2}{c_0n}\right)\geq \frac{1}{8},
		\end{split}
		\end{equation}
		where the last inequality holds since $n\geq 36/(c_0 \log 2), n\geq 36C_{KL}/c_0$, and $c_{gap} \leq (c_0/300)^{1\over 4}$. 
		\item When $nt^2 >  p/4$, we fix $\mu = (t, 0,\ldots, 0)^{\intercal}\in \mathbb{R}^{p},$ introduce 
		$$X_l = \mu l^\intercal, \quad l \in \{-1, 1\}^n,$$
		and denote $P_l$ as the probability of $Y$ if $Y = X_l + Z$ with $Z\overset{iid}{\sim}N(0, 1)$. Based on the calculation in Theorem \ref{th:denoising_lower}, the KL-divergence between $P_l$ and $P_{l'}$ satisfies
		$$D(P_{l}|| P_{l'}) = \frac{t^2}{2}\|l - l'\|_2^2 \leq 2t^2n. $$
		Applying the generalized Fano's lemma on $A$, we have
		$$\inf_{\hat{l}} \max_{l \in A} \E \mathcal{M}(\hat{l}, l) \geq \frac{1}{6} \left(1 - \frac{2t^2n+\log 2}{c_0n}\right). $$
		When $nt^2 > p/4$, $t = c_{gap}(p/n)^{1\over 4}$, we know
		$$\frac{1}{2}(p/n)^{1\over 2} \leq t = c_{gap}(p/n)^{1\over 4}, \quad \text{thus} \quad t \leq \frac{(c_{gap}(p/n)^{1\over 4})^2}{\frac{1}{2}(p/n)^{1\over 2}} \leq 2c_{gap}^2. $$
		Provided that $n \geq 36/(c_0 \log 2), c_{gap}\leq (c_0/300)^{1\over 4}$, we have
		$$\frac{1}{6}\left(1 - \frac{2t^2n + \log 2}{c_0n}\right) \geq \frac{1}{6}\left(1 - \frac{16c_{gap}^4}{c_0} - \frac{\log 2}{c_0n}\right) \geq \frac{1}{8},$$
		which implies
		$$\inf_{\hat{l}} \max_{\substack{\|\mu\|_2 \leq c_{gap}(p/n)^{1 \over 4}\\ l \in \{-1, 1\}^n}} \E \mathcal{M}(\hat{l}, l) \geq \frac{1}{8}.$$
	\end{itemize}
	To sum up, we have finished the proof of this theorem.

\end{proof}

\subsection*{Proofs in Canonical Correlation Analysis}
\begin{proof}[Proof of Theorem \ref{th:CCA}] 
	The proof of Theorem \ref{th:CCA} is relatively complicated, which we shall divide into steps.
\begin{enumerate}
	\item {\it (Analysis of Loss)} Suppose the SVD of $\hat{S}$ and $S$ are
	\begin{equation} 
	\begin{split}
	\hat{S} = & \hat {U}_S\hat{\Sigma}_S \hat{V}_S^{\intercal}, \quad \hat{U}_S \in \mathbb{O}_{p_1}, \hat{\Sigma}_S\in \mathbb{R}^{p_1\times p_2}, \hat{V}_S\in \mathbb{O}_{p_2} \\
	S  = & U_S\Sigma_SV_S^{\intercal}, \quad U_S\in \mathbb{O}_{p_1, r}, \Sigma_S\in \mathbb{R}^{r\times r}, V_S\in \mathbb{O}_{p_2, r}\\ 
	& \text{(we shall note $rank(S) = r$.)}
	\end{split}
	\end{equation}
	respectively. Recall that $A = \Sigma_{X}^{-{1\over 2}}U_S$, $\hat{A} = \hat{\Sigma}_X^{-{1\over 2}}\hat{U}_{S,[:, 1:r]}$. Invertible multiplication to all $X$'s or $Y$'s does not change the loss of the procedure, thus without loss of generality we could assume that $\Sigma_X = I_{p_1}$. Under this assumption, we have following expansions for the loss $L_F$ and $L_{\rm sp}$.
	\begin{equation*}
	\begin{split}
	& L_F(\hat A, A)\\
	 = & \min_{O\in \mathbb{O}_{r}}\E_{X^\ast}\|(\hat{A}O - A)^{\intercal}X^\ast\|_2^2\\
	= & \min_{O\in \mathbb{O}_{r}}\E_{X^\ast}\tr\left((\hat{A}O - A)^{\intercal}X^\ast (X^\ast)^{\intercal}(\hat{A}O - A)\right)\\
	= & \min_{O\in \mathbb{O}^r} \tr\left((\hat{\Sigma}_X^{-{1\over 2}}\hat{U}_{S,[:,1:r]}O - U_S)^{\intercal}I_{p_1}(\hat{\Sigma}_X^{-{1\over 2}}\hat{U}_{S,[:,1:r]}O-U_S)\right)\\
	= & \min_{O\in \mathbb{O}_r}\left\|\hat\Sigma_X^{-{1\over 2}} \hat{U}_{S, [:,1:r]}O - U_S\right\|_F^2\\
	\leq & 2\min_{O\in \mathbb{O}_r}\left\{\left\|\hat{U}_{S, [:, 1:r]} O - U_S\right\|^2_F + \left\|(\hat{\Sigma}^{-{1\over 2}}_X - I)\hat{U}_{S, [:, 1:r]}O\right\|_F^2\right\}\\
	\leq & 4\left\|\sin\Theta(\hat{U}_{S, [:, 1:r]}, U_S)\right\|_F^2 + 2\left\|(\hat{\Sigma}^{-{1\over 2}}_X - I)\hat{U}_{S, [:, 1:r]}O\right\|_F^2 \text{ (by Lemma \ref{lm:sin_Theta_distance}).}
	\end{split}
	\end{equation*}
	Similarly,
	\begin{equation}\label{ineq:CCA_loss}
	\begin{split}
	L_{\rm sp}(\hat A, A) = & \min_{O\in \mathbb{O}_r}\max_{v\in \mathbb{R}^r, \|v\|_2=1} \E_{X^\ast}\left((\hat{A}Ov)^{\intercal}X^\ast - (Av)^{\intercal}X^\ast\right)^2\\
	= & \min_{O\in \mathbb{O}_r}\max_{v\in \mathbb{R}^r, \|v\|_2=1} (\hat{A}Ov - Av)^{\intercal}\Sigma_X (\hat{A}Ov - Av)\\
	= & \min_{O\in \mathbb{O}_r}\left\|\hat{A}O - A\right\|^2 = \min_{O\in \mathbb{O}_{r}}\left\|\hat{\Sigma}_X^{-\frac{1}{2}}\hat{U}_{S, [:, 1:r]}O - U_S\right\|\\
	\leq & 4\left\|\sin\Theta(\hat{U}_{S, [:, 1:r]}, U_S)\right\|^2 + 2\left\|\hat{\Sigma}_X^{-{1\over 2}} - I\right\|^2.
	\end{split}
	\end{equation}
	We use the bold symbols $\X\in \mathbb{R}^{p_1\times n}, \Y\in \mathbb{R}^{p_2\times n}$ to denote those $n$ samples. Since $\hat{\Sigma} = \frac{1}{n}\X\X^{\intercal}$, where $\X\in\mathbb{R}^{p_1\times n}$ is a i.i.d. Gaussian matrix, by random matrix theory (Corollary 5.35 in \cite{vershynin2012introduction}), there exists constant $C, c$ such that
	\begin{equation}
	1 - C\sqrt{p_1/n} \leq \sigma_{\min}(\hat{\Sigma}_X)\leq \sigma_{\max}(\hat{\Sigma}_X) \leq 1 + C\sqrt{p_1/n}.
	\end{equation}
	with probability at least $1-C\exp(cp_1)$. Since 
	$$1 \overset{\text{property of correlation matrix}}{\geq} \sigma_r^2(S)\overset{\text{problem assumption}}{\geq} \frac{C_{\rm gap}p_1}{n},$$ 
	we can find large $C_{\rm gap}>0$ to ensure $n\geq Cp_1$ for any large $C>0$. In addition, we could have
	\begin{equation}\label{ineq:hat_X-I sp}
	\P\left(\left\|\hat{\Sigma}_X^{-{1\over 2}} - I\right\|^2 \leq Cp_1/n \leq \frac{Cp_1}{n\sigma_r^2(S)} \right) \geq 1 - C\exp(-cp_1).
	\end{equation}
	Moreover, as $\hat{U}_{S, [:, 1:r]}O\in \mathbb{O}_{p_1, r}$, we have 
	$$r\left\|\hat{\Sigma}_X^{-{1\over 2}} - I\right\|^2 \geq \left\|(\hat{\Sigma}_X^{-{1\over 2}} - I)\hat{U}_{S, [:, 1:r]}O\right\|_F^2,$$ 
	thus
	\begin{equation}\label{ineq:hat_X-I F}
	\P\left(\left\|(\hat{\Sigma}_X^{-{1\over 2}} - I)\hat{U}_{S, [:, 1:r]}O\right\|_F^2 \leq Cp_1r/n \leq \frac{Cp_1r}{n\sigma_r^2(S)} \right) \geq 1 - C\exp(-cp_1).
	\end{equation}
	Now the central goal in our analysis moves to bound
	$$\left\|\sin\Theta(\hat{U}_{S, [:, 1:r]}, U_S)\right\|_F^2 \text{ and } \left\|\sin\Theta(\hat{U}_{S, [:, 1:r]}, U_S)\right\|^2, $$
	namely to compare the singular spectrum between the population $S = \Sigma_X^{-{1\over 2}}\Sigma_{XY}\Sigma_Y^{-{1\over 2}}$ and the sample version $\hat{S} = \hat{\Sigma}_X^{-{1\over 2}}\hat{\Sigma}_{XY}\hat{\Sigma}_Y^{-{1\over 2}}$
	in $\sin \Theta$ distance. Since $\sin\Theta(\hat{U}_{S, [:, 1:r]}, U_S)$ is $r\times r$, $\left\|\sin\Theta(\hat{U}_{S, [:, 1:r]}, U_S)\right\|_F^2 \leq  \left\|\sin\Theta(\hat{U}_{S, [:, 1:r]}, U_S)\right\|^2$. We only need to consider the $\sin \Theta$ spectral distance below.
	
	\item {\it (Reformulation of the Model Set-up)}  In this step, we transform $X, Y$ to a better formulation to simplify the further analysis. First, any invertible affine transformation on $\X$ and $\Y$ separately will not essentially change the problem. We specifically do transformation as follows
	$$\X \to [U_S~U_{S\perp}]^{\intercal}\Sigma_X^{-{1\over 2}}\X, \quad \Y \to (I_{p_2} - \Sigma_S^2)^{-{1\over 2}}[V_S~V_{S\perp}]^{\intercal} \Sigma_Y^{-{1\over 2}} \Y. $$
	Simple calculation shows that after transformation, $\Var(X)$, $\Var(Y), \Cov(X, Y)$ become $I_{p_1}$, $(I_{p_2} - \Sigma_S^2)^{-1}$, $\Sigma_S(I_{p_2} - \Sigma_S^2)^{-{1\over 2}}$ respectively. Therefore, without loss of generality we can assume that
	\begin{equation}\label{eq:assumption_Sigma_X_Y}
	\Sigma_X = I_{p_1}, \quad\Sigma_Y = (I_{p_2} - \Sigma_S^2)^{-1}.
	\end{equation}
	\begin{equation}\label{eq:assumption-S-CCA}
	S \in \mathbb{R}^{p_1\times p_2} \text{ is diagonal, such that } S_{ij} = \left\{
	\begin{array}{ll}
	\sigma_i(S), & i =j =1,\cdots, r\\
	0, & \text{otherwise}
	\end{array}
	\right.
	\end{equation} 
	throughout the proof. (It will be explained a while later why we want to transform $\Sigma_Y$ into this form.)
	
	
	Since $X, Y$ are jointly Gaussian, we can relate them as $Y = W^{\intercal}X + Z$, where $W\in \mathbb{R}^{p_1\times p_2}$ is a fixed matrix, $X\in \mathbb{R}^{p_1}$ and $Z\in \mathbb{R}^{p_2}$ are independent random vectors. Based on simple calculation, $\Sigma_{XY} = \E XY^{\intercal} = \Sigma_XW$, $\Sigma_Y = W^{\intercal}\Sigma_X W + \Var(Z)$. Combining \eqref{eq:assumption_Sigma_X_Y}, we can calculate that 
	\begin{equation}
	\begin{split}
	& W = \Sigma_X^{-1}\Sigma_{XY} = S\Sigma_Y^{1\over 2}, \quad W \text{ is diagonal},\\
	& W_{ij} = \left\{\begin{array}{ll}
	S_{ii}(1 - S_{ii}^2)^{-{1\over 2}}, & \quad i=j\\
	0 & \text{otherwise}
	\end{array}\right.
	\end{split}
	\end{equation} 
	\begin{equation}
	\Var(Z) = \Sigma_Y - \Sigma_Y^{1\over 2} S^{\intercal}S\Sigma_Y^{1\over 2} = \Sigma^{1\over 2}_Y(I_{p_2} - S^{\intercal}S)\Sigma^{1\over 2}_Y = I_{p_2}.
	\end{equation}
	In other words, $Z$ is i.i.d. standard normal and $W$ is diagonal. By rescaling, the analysis could be much more simplified and easier to read.
	
	\item {\it (Expression for $\hat S$)} In this step, we move from the population to the samples and find out a useful expression for $\hat S$. We use bold symbols $\X\in \mathbb{R}^{p_1\times n}$, $\Y\in \mathbb{R}^{p_2\times n}$, $\Z \in \mathbb{R}^{p_2\times n}$ to denote the compiled data such that $\Y = W^{\intercal}\X +\Z$. Denote the singular decomposition for $\X$ and $\Y$ are 
	$$\X=\hat U_X\hat\Sigma_X\hat{V}_X^{\intercal},\quad \Y = \hat{U}_Y\hat{\Sigma}_Y \hat{V}_Y^{\intercal}, $$
	here $\hat{U}_X\in \mathbb{O}_{p_1}, \hat{\Sigma}_X\in \mathbb{R}^{p_1\times p_1}, \hat{V}_X\in \mathbb{O}_{n, p_1}$, $\hat{U}_Y \in \mathbb{O}_{p_2}, \hat{\Sigma}_Y \in \mathbb{R}^{p_2\times p_2}, \hat{V}_Y \in \mathbb{O}_{n, p_2}$. Thus,
	\begin{equation*}
	\hat{\Sigma}_X = \hat{U}_X \hat{\Sigma}_X^2 \hat U_X^{\intercal}, \quad \hat \Sigma_Y = \hat U_Y \hat{\Sigma}_Y^2 \hat U_Y^{\intercal}, \quad \hat{\Sigma}_{XY} = \hat U_X\hat{\Sigma}_X \hat{V}_X^{\intercal} \hat{V}_Y \hat{\Sigma}_Y\hat{U}_Y^{\intercal}.
	\end{equation*}
	Additionally, 
	\begin{equation*}
	\hat S = \hat{\Sigma}_X^{-{1\over 2}} \hat{\Sigma}_{XY}\hat{\Sigma}_Y^{-{1\over 2}} = \hat U_X\hat{V}_X^{\intercal}\hat V_Y \hat U_Y^{\intercal}.
	\end{equation*}
	\item {\it (Useful Characterization of the left Singular Space of $\hat{S}$)} Since $\X$ is a i.i.d. Gaussian matrix at this moment, from the random matrix theory, we know $\hat U_X$, $\hat V_X$ are randomly distributed with Haar measure on $\mathbb{O}_{p_1}$ and $\mathbb{O}_{n, p_1}$ respectively and $\hat U_X, \hat \Sigma_X, \hat V_X, Z$ are independent. We can extend $\hat V_X \hat U_X^{\intercal}\in \mathbb{O}_{n, p_1}$ to orthogonal matrix $\tilde{R}_{X} = [\hat V_X\hat U_X^{\intercal}, \hat R_{X\perp}]\in O_{n}$ such that $\hat R_{X\perp} \in \mathbb{O}_{n, n-p_1}$. Introduce $\tilde{\Z} = \Z\cdot \tilde{R}_X$,  $\tilde{\Y} = \Y\cdot \tilde{R}_X$, $\tilde{\X} = \X \tilde{R}_X$. $\tilde{\Y}^{\intercal}$ can be explicitly written as the following clean form
	\begin{equation}\label{eq:tilde_Y}
	\begin{split}
	\tilde{\Y}^{\intercal} = & \tilde{\X}^{\intercal} W + \tilde{\Z}^{\intercal} = \begin{bmatrix}
	\hat{U}_X\hat{V}_X^{\intercal}\X ^{\intercal}W\\
	\hat{R}_{X\perp}^{\intercal}\X^{\intercal}W
	\end{bmatrix} + \tilde{Z}^{\intercal} = 
	\begin{blockarray}{cc}
	& p_2 \\
	\begin{block}{c[c]}
	p_1 &  \hat U_X \hat\Sigma_X \hat U_X^{\intercal}W\\
	n-p_1 &  0 \\
	\end{block}
	\end{blockarray} + \tilde{\Z}^{\intercal}\\
	 := & \G + \tilde{\Z}^{\intercal},\\
	& \tilde{\Z}\overset{iid}{\sim}N(0, 1),\quad  \tilde{\Z}, \hat{U}_X, \hat{\Sigma}_X \text{ are independent; } \G, \tilde{\Z} \text{ are independent.}
	\end{split}
	\end{equation}
	Meanwhile, since $\hat V_Y$ is the left singular vectors for $\Y^{\intercal}$, we have
	$$[\hat V_X\hat U_X^{\intercal}, \hat R_{X\perp}]^{\intercal}\hat V_Y = \begin{bmatrix}
	\hat U_X \hat V_X^{\intercal} \hat V_Y\\
	\hat R_{X\perp}^{\intercal}\hat V_Y
	\end{bmatrix}$$ 
	is the left singular vectors of $\tilde{\Y}^{\intercal} = [\hat V_X\hat U_X^{\intercal}, \hat R_{X\perp}]^{\intercal} \Y^{\intercal}$, which yields the following important characterization for $\hat S$.
	\begin{equation}\label{eq:hat_S_characterization}
	\hat S \hat U_Y = \hat U_X\hat V_X^{\intercal} \hat{V}_Y \text{ is the first $p_1$ rows of the left singular vectors of $\tilde{\Y}^{\intercal}$.} 
	\end{equation}
	Suppose the SVD for $\tilde{\Y}$ is
	\begin{equation*}
	\tilde{\Y}^{\intercal} = \tilde{U}_Y \tilde{\Sigma}_Y \tilde{V}_Y^{\intercal},\quad \tilde{U}_Y \in \mathbb{O}_{n, p_2}, \tilde{\Sigma}_Y\in \mathbb{R}^{p_2\times p_2}, \tilde{V}_Y \in \mathbb{O}_{p_2}.
	\end{equation*}
	Further assume the SVD for the first $p_1$ rows of $\tilde{U}_Y$ is
	$$\left(\tilde{U}_Y\right)_{[1:p_1, :]} = \tilde{U}_{Y2} \tilde{\Sigma}_{Y2} \tilde{V}_{Y2}^{\intercal}, \quad \tilde{U}_{Y2} \in \mathbb{O}_{p_1}, \tilde{\Sigma}_{Y2}\in \mathbb{R}^{p_1\times p_2}, \tilde{V}_2 \in \mathbb{O}_{p_2}.$$
	By characterization \eqref{eq:hat_S_characterization} and the fact that right multiplication of $\hat{U}_Y\in \mathbb{O}_{p_2}$ does not change left singular vectors, we have
	\begin{equation}
	\begin{split}
	\hat{U}_S = \tilde{U}_{Y2}, & \text{ where}\\ 
	&	\text{$\hat{U}_S$ is the left singular vectors of $\hat{S}$,}\\ 
	&	\tilde{U}_{Y2} \text{ is the left singular space of } \left(\tilde{U}_Y\right)_{[1:p_1, ]}.\\
	\end{split}
	\end{equation}
	The characterization above is the baseline we shall use later to compare the spectrum of $\hat{S}$ and $S$.
	
	
	\item {\it (Split of $\sin \Theta$ Norm Distance)} Recall Step 1, the central goal of analysis now is to find the $\sin \Theta$ distance between the leading $r$ left singular vectors of $\hat S$ and $S$. It is easy to see that the left singular space of $S$ is
	\begin{equation}
	U_S = \begin{bmatrix}
	1 & & \\
	& \ddots & \\
	&  & 1\\
	0 & \cdots & 0
	\end{bmatrix} \in \mathbb{O}^{p_1, r}
	\end{equation}
	By the traingle inequality of $\sin \Theta$ distance (Lemma \ref{lm:sin_Theta_distance}), 
	\begin{equation}\label{ineq:CCA-sin-theta-split}
	\left\|\sin\Theta\left(\hat{U}_{S, [:, 1:r]}, U_S \right)\right\| \leq \left\|\sin\Theta\left(U_{G}, U_S \right)\right\| + \left\|\sin\Theta\left(U_{G}, \hat{U}_{S, [:, 1:r]} \right)\right\|.
	\end{equation} 
	where $U_G$ is defined a little while later in \eqref{eq:U_G}. For the next Steps 6 and 7, we try to bound the two $\sin \Theta$ distances respectively.
	
	\item {\it (Left Singular Space of $\G_{[1:p_1, 1:r]}$)} Recall the definition of $\G$ in \eqref{eq:tilde_Y}, since $W$ is with only first $r$ diagonal entries non-zero, only the submatrix $\G_{[1:p_1, 1:r]}$ of $\G$ is non-zero. Suppose
	\begin{equation}\label{eq:U_G}
	\G_{[1:p_1, 1:r]} = U_G\Sigma_G V_G^{\intercal}, \quad U_G\in \mathbb{O}_{p_1, r}, \Sigma_G\in \mathbb{R}^{r\times r}, V_G\in \mathbb{O}_{r}.
	\end{equation}
	In this section we aim to show the following bound on the $\sin \Theta$ distance from $U_G$ to $U_S$, i.e. there exists $C_0>0$ such that whenever $n>Cp_1$,
	\begin{equation}\label{ineq:distance_UG_US}
	\P\left(\left\|\sin\Theta(U_G, U_S)\right\| \geq C\sqrt{p_1/n}\right) \leq C\exp(-cp_1)
	\end{equation}
	and also
	\begin{equation}\label{ineq:sigma_min(G)}
	\P\left(\sigma_{\min}^2(\G_{[1:p_1, 1:r]}) \geq \frac{cn\sigma_r^2(S)}{\sqrt{1 - \sigma_r^2(S)}} \right) \geq 1 - C\exp(-cp_1).
	\end{equation}
	Recall
	\begin{equation}
	G_{[1:p_1, 1:r]} = \hat{U}_X\hat{\Sigma}_X \hat{U}_X^{\intercal} W_{[:, 1:r]}, \quad W_{[:, 1:r]}  = \begin{bmatrix}
	W_{11} & & \\
	&\ddots & \\
	& & W_{rr}\\
	& 0 & \\
	\end{bmatrix}.
	\end{equation}
	Thus if we split $\hat{U}_X$ as
	$$\hat{U}_X = \begin{blockarray}{cc}
	p_1 & \\
	\begin{block}{[c]c}
	\hat{U}_{X1} & r \\
	\hat{U}_{X2} & p_1 - r \\
	\end{block}
	\end{blockarray}, \quad\text{then}\quad \hat{U}_X\hat{\Sigma}_X\left(\hat{U}^{\intercal}_X\right)_{[:, 1:r]} = \begin{blockarray}{cc}
	r & \\
	\begin{block}{[c]c}
	\hat{U}_{X1} \hat{\Sigma}_X \hat{U}_{X1}^{\intercal} & r \\
	\hat{U}_{X2} \hat{\Sigma}_X \hat{U}_{X1}^{\intercal} & p_1 - r \\
	\end{block}
	\end{blockarray}.$$ 
	Furthermore, due to the diagonal structure of $W$, $\G_{[1:p_1, 1:r]}$ has the same left singular space as $\hat{U}_X\hat{\Sigma}_X \left(\hat{U}_X^{\intercal}\right)_{[:, 1:r]}$. Since $\hat{\Sigma}_X$ is the singular values of i.i.d. Gaussian matrix $X\in \mathbb{R}^{p_1\times n}$, by random matrix theory \citep{vershynin2012introduction}, 
	\begin{equation}\label{ineq:bound_Sigma_X}
	\P\left(\sqrt{n} + C\sqrt{p_1} \geq \|\hat{\Sigma}_X\| \geq \sigma_{\min}\left(\hat{\Sigma}_X\right) \geq \sqrt{n} - C\sqrt{p_1}\right) \geq 1- C\exp(-cp_1). 
	\end{equation}
	Under the event that $ \sqrt{n} + C\sqrt{p_1} \geq \|\hat{\Sigma}_X\| \geq \sigma_{\min}\left(\hat{\Sigma}_X\right) \geq \sqrt{n} - C\sqrt{p_1}$ hold and $n\geq C_{\rm gap} p_1$ for some large $C_{\rm gap}$, we have
	\begin{equation*}
	\begin{split}
	\sigma_{\min}\left(\hat{U}_{X1} \hat{\Sigma}_X \hat{U}_{X1}^{\intercal}\right) \geq \sigma_{\min}(\hat{\Sigma}_X) \geq \sqrt{n} - C\sqrt{p_1}, \quad\text{(since $\hat{U}_{X1}$ is orthogonal.)}
	\end{split}
	\end{equation*}
	\begin{equation*}
	\begin{split}
	\left\|\hat{U}_{X1} \hat{\Sigma}_X \hat{U}_{X2}^{\intercal}\right\| = \left\|\hat{U}_{X1} \left(\hat{\Sigma}_X - \sqrt{n}I\right) \hat{U}_{X2}^{\intercal}\right\| \leq C\sqrt{p_1}.
	\end{split}
	\end{equation*}
	By Lemma \ref{lm:U_svd}, the left singular space of $\hat{U}_X\hat{\Sigma}_X\left(\hat{U}^{\intercal}_X\right)_{[:, 1:r]}$, which is also the left singular vectors of $G_{[1:p_1, 1:r]}$ satisfies
	$$\sigma_{\max}^2\left(U_{G, [(r+1): n, 1:r]}\right) \leq \frac{\left\|\hat{U}_{X1} \hat{\Sigma}_X \hat{U}_{X2}^{\intercal}\right\|^2}{\sigma^2_{\min}\left(\hat{U}_{X1} \hat{\Sigma}_X \hat{U}_{X1}^{\intercal}\right) + \left\|\hat{U}_{X1} \hat{\Sigma}_X \hat{U}_{X2}^{\intercal}\right\|^2} \leq \frac{Cp_1}{n} $$
	when $n \geq C_{\rm gap}p_1$ for some large $C_{\rm gap}>0$. By the characterization of $\sin \Theta$ distance in Lemma \ref{lm:sin_Theta_distance}, we have finally proved the statement \eqref{ineq:distance_UG_US}.
	
	Since
	\begin{equation}
	\begin{split}
	& \sigma_{\min}(\G_{[1:p_1, 1:r]}) = \sigma_r\left(\hat{U}_X \hat{\Sigma}_X \hat{U}_X^{\intercal} W_{[:, 1:r]}\right)\\
	\geq &  \sigma_r\left(\hat{U}_X\hat{\Sigma}_X\hat{U}_{X1}^{\intercal}\right)\cdot \sigma_{\min}\left(W_{[:, 1:r]}\right)	\geq \sigma_{\min}(\hat{\Sigma}_X)\cdot \frac{\sigma_r(S)}{\sqrt{1-\sigma_r^2(S)}},
	\end{split}
	\end{equation}
	by \eqref{ineq:bound_Sigma_X}, \eqref{ineq:sigma_min(G)} holds.
	
	\item In this step we try to prove the following statement: there exists constant $C_{\rm gap}$, such that whenever 
	\begin{equation}\label{ineq:CCA_step7_condition}
	\sigma_{r}^2(S) \geq C_{\rm gap}\frac{(p_1p_2)^{1\over 2} + p_1 + p_2^{3/2}n^{-{1\over 2}}}{n},
	\end{equation}
	we have
	\begin{equation}\label{ineq:CCA_distance2}
	\P\left(\left\|\sin\Theta(\hat{U}_{S,[1:r, :]}, U_G)\right\|^2 \leq  \frac{Cp_1\left(n\sigma_r^2(S) + p_2\right)}{n^2\sigma_{r}^4(S)}\right) \geq 1 - C\exp(-cp_1\wedge p_2).
	\end{equation}
	We shall again note that by $1 \geq \sigma_r^2(S)$, the condition \eqref{ineq:CCA_step7_condition} also implies $n\geq C_{\rm gap}p_1$. Recall \eqref{ineq:sigma_min(G)}, with probability at least $1 - C\exp(-cp_1)$,
	\begin{equation}\label{ineq:sigma_r^2(G) = t^2}
	\sigma_r^2(\G) \geq \frac{cn\sigma_r^2(S)}{1-\sigma_r^2(S)} \geq cn\sigma_r^2(S) \geq cC_{\rm gap}\left((p_1p_2)^{1\over 2} + p_1 + p_2^{3/2}n^{-{1\over 2}}\right).
	\end{equation}
	Conditioning on $G$ with $\sigma_r^2(G)$ satisfies \eqref{ineq:sigma_r^2(G) = t^2}, set $L=\tilde{Y}^{\intercal}$, $n = n$, $p = p_2$, $d = p_1$, $r = r$, Lemma \ref{lm:CCA} leads to the following result 
	\begin{equation*}
	\begin{split}
	& \P\left(\left\|\sin\Theta(\hat{U}_{S,[:, 1:r]}, U_G)\right\|^2 \leq \frac{Cp_1(\sigma_r^2(\G) + p_2) (1+\sigma_r^2(\G)/n)}{\sigma_r^4(\G)} \Bigg |\G \right)\\
	\geq & 1 - C\exp(-cp_1\wedge p_2) .
	\end{split}
	\end{equation*}
	whenever $\sigma_r^2(\G) \geq C_{\rm gap} \left((p_1p_2)^{1\over 2} + p_1 + p_2^{3/2} n^{-{1\over 2}}\right)$ for uniform constant $C_{\rm gap}$ large enough. Then we shall also note that
	\begin{equation*}
	\begin{split}
	& \|S\|\leq 1, \Rightarrow \sigma_r^2(S) \leq 1\\
	\Rightarrow &  \frac{Cp_1\left(n\sigma_r^2(S) + p_2\right)}{n^2\sigma_{r}^4(S)} \geq \frac{Cp_1\left(n\sigma_r^2(S) + p_2\right)(1 + n\sigma_r^2(S)/n)}{n^2\sigma_{r}^4(S)}\\
	\overset{\eqref{ineq:sigma_r^2(G) = t^2}}{\Rightarrow} &\quad  \frac{Cp_1\left(n\sigma_r^2(S) + p_2\right)}{n^2\sigma_{r}^4(S)} \geq \frac{Cp_1(\sigma_r^2(\G) + p_2) (1+\sigma_r^2(\G)/n)}{\sigma_r^4(\G)}.
	\end{split}
	\end{equation*}
	The two inequalities above together implies the statement \eqref{ineq:CCA_distance2} holds for true.
	
\end{enumerate}
Finally, combining \eqref{ineq:CCA-sin-theta-split}, \eqref{ineq:distance_UG_US}, \eqref{ineq:CCA_distance2}, \eqref{ineq:hat_X-I sp}, \eqref{ineq:hat_X-I F} and \eqref{ineq:CCA_loss}, we have completed the proof of Theorem \ref{th:CCA}.
\end{proof}

\ \par

The following Lemmas \ref{lm:CCA}, \ref{lm:U_svd} and \ref{lm:random_projection} are used in the proof of Theorem \ref{th:CCA}. To be specific, Lemma \ref{lm:CCA} provides a $\sin \Theta$ upper bound for the singular vectors of a sub-matrix; Lemma \ref{lm:random_projection} gives both upper and lower bounds for the singular values of a sub-matrix; Lemma \ref{lm:random_projection} propose a spectral norm upper bound for any matrix after a random projection.

\begin{Lemma}\label{lm:CCA}
	Suppose $L\in \mathbb{R}^{n\times p}$ ($n>p$) is a non-central random matrix such that $L = G+Z$, $Z\overset{iid}{\sim}N(0, 1)$, $G$ is all zero except the top left $d\times r$ block ($d\geq r$). Suppose the SVD for $G_{[1:d, 1:r]}, L$ are 
	$$G_{[1:d, 1:r]} = U_G \Sigma_G V_G^{\intercal},\quad U_G\in \mathbb{O}_{d, r}, \Sigma_G \in \mathbb{R}^{r\times r}, V_G\in \mathbb{O}_{r};$$
	$$L = \hat U\hat\Sigma \hat V^{\intercal},\quad \hat U\in \mathbb{O}_{n, p}, \hat \Sigma \in \mathbb{R}^{p\times p}, \hat V \in \mathbb{O}_{p}.$$ 
	In addition, suppose $r < d < n$, the SVD for $\hat U_{[1:d, :]}$ is $\hat U_{[1:d, :]} = \hat U_2\hat{\Sigma}_2 \hat V_2^{\intercal}$. There exists $C_{\rm gap}, C_0>0$ such that whenever
	\begin{equation}\label{ineq:lm_CCA_assumption}
	\sigma^2_r(G) = t^2 > C_{\rm gap}((pd)^{1\over 2} + d + p^{3/2}n^{-{1\over 2}}),\quad n \geq C_0 p,
	\end{equation} 
	we have
	\begin{equation}
	\|\sin\Theta(\hat U_{2,[:, 1:r]}, U_G)\| \leq \frac{C\sqrt{d(t^2+p)(1+t^2/n)}}{t^2}
	\end{equation}
	with probability at least $1 - C\exp(-cd\wedge p)$.
\end{Lemma}

\begin{Lemma}[Spectral Bound for Partial Singular Vectors]\label{lm:U_svd}
	Suppose $L\in \mathbb{R}^{n\times p}$ with $n>p$, $L = U\Sigma V^{\intercal}$ is the SVD with $U\in \mathbb{O}_{n\times p}, \Sigma \in \mathbb{R}^{p\times p}, V\in \mathbb{O}_{p}$. Then for any subset $\Omega \subseteq\{1,\cdots, n\}$,
	\begin{equation}
	\sigma_{\max}^2(U_{[\Omega, :]}) \leq \frac{\sigma_{\max}^2(L_{[\Omega, :]})}{\sigma_{\max}^2(L_{[\Omega, :]}) + \sigma_{\min}^2(L_{[\Omega^c, :]})},
	\end{equation}
	\begin{equation}
	\sigma_{\min}^2(U_{[\Omega, :]}) \geq \frac{\sigma_{\min}^2(L_{[\Omega, :]})}{\sigma_{\min}^2(L_{[\Omega, :]}) + \sigma_{\max}^2(L_{[\Omega^c, :]})}.
	\end{equation}
\end{Lemma}

\begin{Lemma}[Spectral Norm Bound of a Random Projection]\label{lm:random_projection}
	Suppose $X \in \mathbb{R}^{n \times m}$ is a fixed matrix $\rank(X) = p$, suppose $R \in \mathbb{O}_{n \times d}$ ($n>d$) is with orthogonal columns with Haar measure. There exists uniform constant $C_0>0$ such that whenever $n \geq C_0 d$, the following bound probability hold for uniform constant $C, c>0$,
	\begin{equation}
	\P\left(\left\|R^{\intercal}X\right\|^2 \geq \frac{p + C\sqrt{pd}+Cd}{n}\|X\|\right) \leq 1- C\exp(-cd).
	\end{equation}
\end{Lemma}

\subsection*{Proofs of Technical Lemmas}

\begin{proof}[Proof of Lemma \ref{lm:sin_Theta_distance}]
\ \par
\begin{itemize}
	\item {(Expressions)}
	Suppose $V^{\intercal}\hat V = A \Sigma B^{\intercal}$ is the SVD, where $\Sigma = \diag(\sigma_1, \cdots, \sigma_r)$, $A, B\in \mathbb{O}_r$. By definition of $\sin \Theta$ distances,
	\begin{equation}\label{eq:sin_theta_expression1}
	\begin{split}
	\|\sin\Theta(\hat V, V)\| = \sin (\cos^{-1}(\sigma_r)) = \sqrt{1 - \sigma_r^2} = \sqrt{1 - \sigma_{\min}^2(V^{\intercal}\hat V)},
	\end{split}
	\end{equation}
	\begin{equation}\label{eq:sin_theta_expression_F}
	\begin{split}
	& \|\sin\Theta(\hat V, V)\|_F^2 = \sum_{i=1}^r \sin^2(\cos^{-1}(\sigma_r))\\
	= & \sum_{i=1}^r (1 - \sigma_r^2) = r - \sum_{i=1}^r \sigma_i^2 = r- \|\hat V^{\intercal} V\|_F^2.
	\end{split}
	\end{equation}
	On the other hand, since $\hat{V}$, $V$, $V_{\perp}$ are orthonormal,
	\begin{equation}\label{eq:sigma_min(hat_VV)}
	\begin{split}
	& \sigma_{\min}^2(\hat{V}^{\intercal}V) = \min_{x \in \mathbb{R}^r} \frac{\|\hat{V}^{\intercal} Vx\|_2^2}{\|x\|_2^2}\\
	= & \min_{x \in \mathbb{R}^r} \frac{\|Vx\|_2^2 - \|\hat{V}^{\intercal}_{\perp} Vx\|_2^2}{\|x\|_2^2} = 1 - \max_{x\in \mathbb{R}^r} \frac{\|\hat{V}_{\perp}Vx\|_2^2}{\|x\|_2^2} = 1 - \|\hat{V}_{\perp}^{\intercal}V\|^2,
	\end{split}
	\end{equation}
	\begin{equation}\label{eq:hatVV_F}
	\begin{split}
	& \|\hat{V}^{\intercal}V\|_F^2 = \tr\left(\hat{V}^{\intercal}VV^{\intercal}\hat{V}\right) = \tr\left(\hat{V}\hat{V}^{\intercal}VV^{\intercal}\right)\\
	= & \tr\left(VV^{\intercal} - \hat{V}_{\perp}\hat{V}_{\perp}^{\intercal}VV^{\intercal}\right) = r - \|\hat{V}^{\intercal}_{\perp}V\|_F^2,
	\end{split}
	\end{equation}
	we conclude that $\|\sin\Theta(\hat{V}, V)\| = \|\hat{V}_{\perp}^{\intercal}V\|$, $\|\sin\Theta(\hat{V}, V)\|_F = \|\hat{V}_{\perp}^{\intercal}V\|_F$.
	
	\item {(Triangle Inequality)} Next we consider the triangle inequality under spectral norm \eqref{ineq:triangle_spectral}. Denote 
	$$x= \|\sin\Theta(V_1, V_2)\|,\quad y = \|\sin\Theta(V_1, V_3)\|.$$ 
	For each $i=1,2,3$, we can expand $V_i$ to the full orthogonal matrix as $[V_i~V_{i\perp}]\in \mathbb{O}_{p}$. Thus,
	$$\sigma_{\min}(V_1^{\intercal}V_2) \overset{\eqref{eq:sin_theta_expression1}}{=} \sqrt{1 - x^2}, \quad \sigma_{\max}^2(V_{1\perp}^{\intercal}V_2) \overset{\eqref{eq:sigma_min(hat_VV)}}{=} 1-\sigma_{\min}^2(V_1^{\intercal}V_2) = x^2.$$
	Similarly
	$$\sigma_{\min}(V_1^{\intercal}V_3) = \sqrt{1 - y^2}, \quad \sigma^2_{1}(V_{1\perp}^{\intercal} V_3) =  y^2.$$
	Thus,
	\begin{equation*}
	\begin{split}
	\sigma_{\min}(V_2^{\intercal}V_3) = & \sigma_{\min}(V_2^{\intercal}V_1V_1^{\intercal}V_3 + V_2^{\intercal}V_{1\perp}V_{1\perp}^{\intercal} V_3)\\
	\geq & \sigma_{\min} (V_2^{\intercal}V_1 V_1^{\intercal}V_3) - \sigma_{\max}(V_2^{\intercal}V_{1\perp}V_{1\perp}^{\intercal}V_3)\\
	\geq & \sigma_{\min}(V_2^{\intercal}V_1)\sigma_{\min}(V_1^{\intercal}V_3) - \sigma_{\max}(V_2^{\intercal}V_{1\perp})\cdot\sigma_{\max}(V_{1\perp}^{\intercal} V_3)\\
	\geq & \sqrt{(1-x^2)(1-y^2)} - xy.
	\end{split}
	\end{equation*}
	Therefore, 
	\begin{equation}
	\|\sin\Theta(V_2, V_3)\| \leq \sqrt{1 - \left(\sqrt{(1-x^2)(1-y^2)} - xy\right)_+^2}.
	\end{equation} 
	Now, we discuss under two situations.
	\begin{enumerate}
		\item If $\sqrt{(1-x^2)(1-y^2)} \geq xy$, we have
		\begin{equation*}
		\begin{split}
		& 1 - \left(\sqrt{(1-x^2)(1-y^2)} - xy\right)_+^2\\
		= & x^2 + y^2 - x^2y^2 +2xy\sqrt{(1-x^2)(1-y^2)} - x^2y^2 \leq (x+y)^2. 
		\end{split}
		\end{equation*}
		Thus, $\|\sin\Theta(V_2, V_3)\|\leq x+y$.
		\item If $\sqrt{(1-x^2)(1-y^2)} > xy$, we have $x^2 + y^2 >1$. Provided that $0\leq x, y\leq 1$, this implies $x+ y >1$.
		Thus, $\|\sin\Theta(V_2, V_3)\| \leq 1 \leq x+y$.
	\end{enumerate}
	To sum up, we always have \eqref{ineq:triangle_spectral}. The proof of the triangle inequality under Frobenius norm is slightly simpler. 
	\begin{equation}
	\begin{split}
	& \|\sin\Theta(V_2, V_3)\|_F\\ \overset{\eqref{eq:sin_theta_expression_F}, \eqref{eq:hatVV_F}}{=} & \|V_{2\perp}^{\intercal}V_3\|_F = \|V^{\intercal}_{2\perp} (\mathbb{P}_{V_1} +\mathbb{P}_{V_{1\perp}}) V_{3}\|_F\\
	\leq & \|V_{2\perp}^{\intercal} V_1V_1^{\intercal}V_3\|_F + \|V_{2\perp}^{\intercal} V_{1\perp}V^{\intercal}_{1\perp} V_3\|_F\\
	\leq & \|V_{2\perp}^{\intercal} V_1\|_F \cdot\|V_1^{\intercal}V_3\| + \|V_{2\perp}^{\intercal} V_{1\perp}\| \cdot \|V_{1\perp}^{\intercal} V_3\|_F\\
	\leq & \|V_{2\perp}^{\intercal}V_1\|_F + \| V_{1\perp}^{\intercal} V_3\|_F\leq \|\sin\Theta(V_1, V_3)\|_F + \|\sin\Theta(V_1, V_2)\|_F.
	\end{split}
	\end{equation}
	\item {(Equivalence with Other Metrics)} Since all metrics mentioned in Lemma \ref{lm:sin_Theta_distance} is rotation invariant, i.e. for any $J\in \mathbb{O}_{p}$, $\sin\Theta(\hat{V}, V) = \sin\Theta(J\hat{V}, JV)$, so does the other metrics. Without loss of generality, we can assume that 
	$$V = \begin{bmatrix}
	I_r\\
	0_{(p-r)\times r}
	\end{bmatrix}_{p\times r}$$
	In this case,
	\begin{equation*}
	\begin{split}
	D_{\rm sp}(\hat{V}, V) = \inf_{O\in \mathbb{O}_r} \|\hat{V} - VO\| \geq \|\hat{V}_{[(r+1):p, :]}\| = \|V_{\perp}^{\intercal}\hat{V}\| = \|\sin\Theta(\hat{V}, V)\|.
	\end{split}
	\end{equation*}
	Recall $V^{\intercal}\hat{V} = A\Sigma B^{\intercal}$ is the singular decomposition. 
	\begin{equation*}
	\begin{split}
	& D_{\rm sp}(\hat{V}, V) = \inf_{O\in \mathbb{O}_r} \|\hat{V} - VO\| \leq \|\hat{V} - VAB^{\intercal}\| \\
	\leq & \sqrt{\left\|V^{\intercal}\left(\hat{V} - VAB^{\intercal}\right)\right\|^2 + \left\|V_{\perp}^{\intercal}\left(\hat{V} - VAB^{\intercal}\right)\right\|^2}\\
	= & \sqrt{\|A(\Sigma - I_r)B^{\intercal}\|^2 + \|\hat{V}_{\perp}^{\intercal}\hat{V}\|^2} \leq \sqrt{(1-\sigma_r)^2 + \|\sin\Theta(\hat{V}, V)\|^2}\\
	\leq & \sqrt{1 - \sigma_r^2 + \|\sin\Theta(\hat{V}, V)\|^2}\overset{\eqref{eq:sin_theta_expression1}}{\leq} \sqrt{2}\|\sin\Theta(\hat{V}, V)\|.
	\end{split}
	\end{equation*}
	Similarly, we can show $\|\sin\Theta(\hat{V}, V)\|_F\leq D_{F}(\hat{V}, V) \leq \sqrt{2}\|\sin\Theta(\hat{V}, V)\|_F$.
	
	For $\|\hat{V}\hat{V}^{\intercal} - VV^{\intercal}\|$, one can show
	\begin{equation*}
	\begin{split}
	& \left\|\hat{V}\hat{V}^{\intercal} - VV^{\intercal}\right\| \geq \left\|V_{\perp}^{\intercal}\hat{V}\hat{V}^{\intercal} - V_{\perp}^{\intercal}VV^{\intercal}\right\|\\
	= & \|(V_{\perp}^{\intercal}\hat{V})\hat{V}^{\intercal} \| = \|V_{\perp}^{\intercal}\hat{V}\| = \|\sin\Theta(\hat{V}, V)\| \quad \text{since $V$ is orthonormal}.
	\end{split}
	\end{equation*}
	Besides,
	\begin{equation*}
	\begin{split}
	& \hat{V}\hat{V}^{\intercal} - VV^{\intercal} = \left(\bbP_{V} + \bbP_{V_{\perp}}\right)\hat{V}\hat{V}^{\intercal}\left(\bbP_{V} + \bbP_{V_{\perp}}\right) - VV^{\intercal}\\
	= & V \left((V^{\intercal}\hat{V})(V^{\intercal}\hat{V})^{\intercal} - I_r\right)V^{\intercal}  + V_{\perp} V_{\perp}^{\intercal}\hat{V} \hat{V}^{\intercal}VV^{\intercal}\\
	& +  VV^{\intercal}\hat{V} \hat{V}^{\intercal}V_{\perp}V_{\perp}^{\intercal} + V_{\perp} V_{\perp}^{\intercal}\hat{V} \hat{V}^{\intercal}V_{\perp}V_{\perp}^{\intercal}.
	\end{split}
	\end{equation*}
	For any vector $x\in \mathbb{R}^p$, we denote $x_1 = V^{\intercal}x, x_2 = V_{\perp}^{\intercal}x$. We also denote $t = \|V_{\perp}^{\intercal}\hat{V}\|$. Recall that we have proved in part 1 that $\sigma_{\min}^2(V^{\intercal}\hat{V}) = 1 - \sigma_{\max}(V_{\perp} \hat{V})$, then $1\geq \sigma_{\max}(V^{\intercal}\hat{V}) \geq \sigma_{\min}(V^{\intercal}\hat{V}) \geq \sqrt{1-t^2}$. Thus,
	\begin{equation*}
	\begin{split}
	& \left|x^{\intercal}\left(\hat{V}\hat{V}^{\intercal} - VV^{\intercal}\right)x\right|\\
	= &\Bigg| x_1^{\intercal}\left((V^{\intercal}\hat{V})(V^{\intercal}\hat{V})^{\intercal}- I_r\right)x_1 + x_2^{\intercal}\left(V_{\perp}^{\intercal}\hat{V}\right)\hat{V}^{\intercal}V x_1\\
	& + x_1^{\intercal}\left(V^{\intercal}\hat{V}\right)\hat{V}^{\intercal}V_{\perp} x_2 + x_2^{\intercal}\left(V^{\intercal}_{\perp}\hat{V}\right)\left(\hat{V}^{\intercal}V_{\perp}\right) x_2\Bigg|\\
	\leq & t^2\|x_1\|_2^2 + 2t\|x_2\|_2\|x_1\|_2 + t^2\|x_2\|_2^2\\
	\leq & \left(t^2 + t\right)\left(\|x_1\|_2^2 + \|x_2\|_2^2\right) \leq 2t\|x\|_2^2,
	\end{split}
	\end{equation*}
	which implies
	\begin{equation*}
	\left\|\hat{V}\hat{V}^{\intercal} - VV^{\intercal}\right\| \leq 2t  = 2\left\|\sin\Theta(\hat{V}, V)\right\|.
	\end{equation*}
	This has established the equivalence between $\|\sin\theta(\hat{V}, V)\|$ and $\|\hat{V}\hat{V}^{\intercal} - VV^{\intercal}\|$. Finally for $\|\hat{V}\hat{V}^{\intercal} - VV^{\intercal}\|_F$, one has
	\begin{equation*}
	\begin{split}
	\|\hat{V}\hat{V}^{\intercal} - VV^{\intercal}\|_F^2 = & \tr\left((\hat{V}\hat{V}^{\intercal} - VV^{\intercal})^2\right)\\
	= & \tr\left(\hat{V}\hat{V}^{\intercal}\hat{V}\hat{V}^{\intercal} + VV^{\intercal}VV^{\intercal} - \hat{V}\hat{V}^{\intercal}VV^{\intercal} - VV^{\intercal}\hat{V}\hat{V}^{\intercal}\right)\\
	= & r + r - 2\|\hat{V}^{\intercal} V\|_F^2 = 2\|\sin\Theta(\hat{V}, V)\|_F^2.
	\end{split}
	\end{equation*}
\end{itemize}
\end{proof}

\ \par

\begin{proof}[Proof of Lemma \ref{lm:X_Y_sv}] 
	Before starting the proof, we introduce some useful notations. For any matrix $M\in \mathbb{R}^{p_1\times p_2}$ with the SVD $M = \sum_{i=1}^{p_1\wedge p_2}u_i \sigma_i(M) v_i^{\intercal}$ we use $M_{\max(r)}$ to denote its leading $r$ principle components, i.e. $M_{\max(r)} = \sum_{i=1}^{r} u_i \sigma_i(M) v_i^{\intercal}$; $M_{-\max(r)}$ denotes the remainder, i.e. 
	$$M_{-\max(r)} = \sum_{i=r+1}^{p_1\wedge p_2} u_i\sigma_i(M) v_i^{\intercal}.$$

\begin{enumerate}
	\item First, by a well-known fact about best low-rank matrix approximation,
	$$\sigma_{a+b+1-r} (X+Y) = \min_{M\in \mathbb{R}^{p\times n}, \rank(M) \leq a+b-r }\|X+Y - M\|. $$
	Hence,
	\begin{equation*}
	\begin{split}
	& \sigma_{a+b+1-r}(X+Y) \leq \|X+Y - (X_{\max(a-r)} + Y)\|\\
	= & \|X_{-\max(a-r)}\| = \sigma_{a+1-r}(X); 
	\end{split}
	\end{equation*}
	similarly $\sigma_{a+b+1-r}(X+Y) \leq \sigma_{b+1-r}(Y)$. 
	
	\item When we further have $X^{\intercal}Y=0$ or $XY^{\intercal} = 0$, without loss of generality we can assume $X^{\intercal} Y = 0$. Then the column space of $X$ and $Y$ are orthogonal, and $\rank(X+Y) = \rank(X) + \rank(Y) = a+b$, which means $a+b \leq n$. Next, note that
	$$(X+Y)^{\intercal}(X+Y) = X^{\intercal}X+Y^{\intercal}Y + X^{\intercal}Y + Y^{\intercal}X = X^{\intercal}X + Y^{\intercal}Y,$$
	if we note $\lambda_i(\cdot)$ as the $r$-th largest eigenvalue of the matrix, then we have
	\begin{equation*}
	\begin{split}
	& \sigma^2_i(X+Y)\\
	= & \lambda_i((X+Y)^{\intercal}(X+Y)) = \lambda_i(X^{\intercal}X + Y^{\intercal}Y)\\
	\geq & \max(\lambda_i(X^{\intercal}X), \lambda_i(Y^{\intercal}Y)) \quad \text{(since $X^{\intercal}X, Y^{\intercal}Y$ are semi-positive definite)}\\
	= & \max(\sigma_i^2(X), \sigma_i^2(Y)).
	\end{split}
	\end{equation*}
	\begin{equation*}
	\begin{split}
	\sigma_1^2(X+Y)  = & \lambda_1 ((X+Y)^{\intercal} (X+ Y)) = \|X^{\intercal} X + Y^{\intercal} Y\|\\
	\leq & \|X^{\intercal} X\| + \|Y^{\intercal} Y\| = \sigma_1^2(X) + \sigma_1^2(Y).
	\end{split}
	\end{equation*}
\end{enumerate}
\end{proof}

\begin{proof}[Proof of Lemma \ref{lm:2-by-2}]	
	
	\begin{enumerate}
	\item If $A = \begin{bmatrix}
	a & b\\
	0 & d
	\end{bmatrix}$ and $a^2 \leq b^2 + d^2$,
	$$A^\intercal A = \begin{bmatrix}
	a^2 & ab\\
	ab & b^2 + d^2
	\end{bmatrix} $$
	We can solve the two eigenvalues of $A^\intercal A$ are
	$$\{\lambda_1, \lambda_2\} = \frac{a^2 + b^2+d^2 \pm \sqrt{(b^2 + d^2 - a^2)^2 + 4a^2b^2}}{2}. $$
	$$\lambda_2 \geq \frac{a^2 + b^2 + d^2 - (b^2 + d^2 - a^2) - 2ab}{2} = a(a-b). $$
	By the definition of singular vectors, we have $(A^{\intercal}A - \lambda_2I_2)\begin{bmatrix}
	v_{12}\\
	v_{22}
	\end{bmatrix} = 0$. Thus, 
	$$(a^2 -\lambda_2 )v_{12} + ab v_{22}=0 $$
	Given $v_{12}^2 + v_{22}^2 = 1$, we have
	\begin{equation*}
	\begin{split}
	|v_{12}| = & \frac{ab}{\sqrt{(a^2-\lambda_2)^2 + a^2b^2}}\\ 
	\geq & \frac{ab}{\sqrt{(a^2 - a(a-b))^2 + a^2b^2}} = \frac{ab}{\sqrt{2a^2b^2}} = \frac{1}{\sqrt{2}}.
	\end{split}
	\end{equation*}
	\item 
	If $A = \begin{bmatrix}
	a & b\\
	c & d
	\end{bmatrix}$, 
$$A^{\intercal}A = \begin{bmatrix}
a^2+c^2 & ab+cd\\
ab+cd & b^2+d^2
\end{bmatrix}, $$
If its eigenvalues are $\lambda_1 \geq \lambda_2$, clearly $\lambda_1+ \lambda_2 = a^2+b^2+c^2+d^2$, $\lambda_1\lambda_2 = (ad-bc)^2$. We can solve that its eigen-values $\lambda_1, \lambda_2$ of $A^\intercal A$ are
\begin{equation}
\begin{split}
\{\lambda_1, \lambda_2\} = & \frac{a^2+b^2+c^2+d^2 \pm \sqrt{(a^2+b^2+c^2+d^2)^2 - 4(ad-bc)^2}}{2}\\
= & \frac{a^2+b^2+c^2+d^2 \pm \sqrt{(a^2+c^2 - b^2-d^2)^2 + 4(ab-cd)^2}}{2}
\end{split}
\end{equation}
Thus, $\lambda_2 = \frac{a^2+b^2+c^2+d^2 - \sqrt{(a^2+c^2 - b^2-d^2)^2 + 4(ab-cd)^2}}{2} \leq \frac{a^2+b^2+c^2+d^2 - (a^2+c^2-b^2-d^2)}{2} = b^2 + d^2$. Also,
$$\lambda_2 \geq \frac{a^2+b^2+c^2+d^2 -(a^2+c^2-b^2-d^2) - 2|ab-cd|}{2} \geq b^2 + d^2 - |ab-cd|. $$
Thus,
\begin{equation*}
\begin{split} 
& a^2+c^2 - \lambda_2 \leq a^2 + c^2 - b^2-d^2 + |ab-cd|\\
\leq & 2(a^2 - b^2-d^2) -(a^2 - b^2-c^2-d^2) + |ab+cd|\\
\leq & 2(a^2 -d^2 - b^2\wedge c^2) + |ab+cd|
\end{split}
\end{equation*}
By the definition of singular vectors, we have $(A^{\intercal}A - \lambda_2I_2)\begin{bmatrix}
v_{12}\\
v_{22}
\end{bmatrix} = 0$. Thus, 
$$(a^2 +c^2 -\lambda_2 )v_{12} + (ab+cd) v_{22}=0 $$
Given $v_{12}^2 + v_{22}^2 = 1$, we have
\begin{equation*}
\begin{split}
|v_{12}| = & \frac{ab+cd}{\sqrt{(a^2+c^2-\lambda_2)^2 + (ab+cd)^2}}\\ \geq & \frac{ab+cd}{\sqrt{(2(a^2-d^2-b^2\wedge c^2) + |ab+cd|)^2 + (ab+cd)^2}}\\
\geq & \frac{ab+cd}{\sqrt{4(a^2-d^2-b^2\wedge c^2)^2 + 4(a^2 - d^2-b^2\wedge c^2)|ab+cd| + 2(ab+cd)^2}}\\
\geq & \frac{ab+cd}{\sqrt{10\max\left\{(a^2 - d^2 - b^2\wedge c^2)^2, (ab+cd)^2\right\}}}\\
\geq & \frac{1}{\sqrt{10}}\left(\frac{ab+cd}{a^2 - d^2 - b^2\wedge c^2}\wedge 1\right).
\end{split}
\end{equation*}
This finished the proof of the second part of Lemma \ref{lm:2-by-2}.
\end{enumerate}
\end{proof}

\ \par

\begin{proof}[Proof of Lemma \ref{lm:denoising}] Suppose the SVD of $X$ is $X = U\Sigma V^{\intercal}$, where $U\in \mathbb{O}_{p_1, r}, V\in \mathbb{O}_{p_2, r}, \Sigma = \diag(\sigma_1, \cdots, \sigma_r)\in \mathbb{R}^{r\times r}$. We can extend $V\in \mathbb{O}_{p_2, r}$ into the full $p_2\times p_2$ orthogonal matrix 
$$[V~ V_\perp] = V_0\in \mathbb{O}_{p_2}, \quad V_{\perp}\in \mathbb{O}_{p_2, p_2-r}.$$ 
For convenience, denote $YV \triangleq Y_1$. 
Since,
\begin{equation}\label{eq:EYTY}
EY^{\intercal}Y = X^{\intercal}X + p_1I_{p_2} = V\Sigma^2V^{\intercal} + p_1I_{p_2}, \quad EV^{\intercal}Y^{\intercal}YV = \Sigma^2 + p_1I_{p_2},
\end{equation} 
we introduce fixed normalization matrix $M\in \mathbb{R}^{p_2\times r}$ as
$$M = \begin{bmatrix}
(\sigma_1^2(X) + p_1)^{-{1\over 2}} & & \\
& \ddots & \\
&  & (\sigma_r^2(X)+p_1)^{-{1\over 2}}
\end{bmatrix}_{r\times r}.
$$ 
By \eqref{eq:EYTY}, this design yields 
\begin{equation*}
M^{\intercal}EY_1^{\intercal}Y_1 M = I_r.
\end{equation*}
In other words, by right multiplying $M$ to $Y_1$, we can normalize its second moment. Now we are ready to show \eqref{ineq:sigma(YV)}, \eqref{ineq:sigma(Y)} and \eqref{ineq:P_YV(YV_perp)}.
\begin{enumerate}
	
	\item We target on \eqref{ineq:sigma(YV)} in this step. Note that the maximum diagonal entry of $M$ is $(\sigma_r^2(X) + p_1)^{-{1\over 2}}$, thus 
	\begin{equation*}
	\begin{split}
	\sigma_r^2(Y_1) \geq & (\sigma_r^2(X) + p_1)\sigma_r^2(Y_1M) \geq (\sigma_r^2(X) + p_1)\sigma_r(M^{\intercal}Y_1^{\intercal}Y_1M)\\
	\geq & (\sigma_r^2(X)+p_1) - (\sigma_r^2(X) +p_1)\left\| M^{\intercal}Y_1^{\intercal}Y_1M - I_r \right\|,
	\end{split}
	\end{equation*}
	\eqref{ineq:sigma(YV)} could be implied by 
	\begin{equation}\label{ineq:MVYYVM_probability}
	\P\left(\left\|M^{\intercal}Y_1^{\intercal}Y_1M - I_r\right\| \leq x \right) \geq 1 - C\exp\left\{Cr - c(\sigma_r^2(X)+p_1)x\wedge x^2\right\}.
	\end{equation}
	The main idea to proceed is to use $\varepsilon$-net to split the spectral norm deviation control to single random variable deviation control. Then use Hanson-Wright inequality to control the single random variable. 
	To be specific, for any unit vector $u\in \mathbb{R}^{r}$, by expansion $Y_1 = YV = XV+ZV$, we have
	\begin{equation}\label{eq:YVT_expansion}
	\begin{split}
	& u^{\intercal}M^{\intercal}Y_1^{\intercal}Y_1Mu - u^{\intercal}I_ru\\
	& = u^{\intercal}M^{\intercal}Y_1^{\intercal}Y_1Mu - Eu^{\intercal}M^{\intercal}Y_1^{\intercal}Y_1Mu \\
	= & \left(u^{\intercal}M^{\intercal}V^{\intercal} X^{\intercal}XVMu - Eu^{\intercal}M^{\intercal}V^{\intercal}X^{\intercal}XVMu\right)\\
	& + \left(2u^{\intercal}M^{\intercal}V^{\intercal}X^{\intercal}ZV Mu - E2u^{\intercal}M^{\intercal}V^{\intercal}X^{\intercal}ZVMu\right)\\
	& + \left(u^{\intercal}M^{\intercal}V^{\intercal}Z^{\intercal}ZVMu - Eu^{\intercal}M^{\intercal}V^{\intercal}Z^{\intercal}ZVMu\right)\\
	= & 2(XVMu)^{\intercal} Z V(Mu) + (VMu)^{\intercal}(Z^{\intercal}Z - p_1I_{p_2})(VMu).
	\end{split}
	\end{equation}
	We shall emphasize that the only random variable in the equation above is $Z\in \mathbb{R}^{p_1\times p_2}$. Our plan is to bound the two terms in \eqref{eq:YVT_expansion} separately as follows.
	\begin{itemize}
		\item For fixed unit vector $u\in \mathbb{R}^r$, we vectorize $Z\in \mathbb{R}^{p_1\times p_2}$ into $\vec{\z}\in \mathbb{R}^{p_1p_2}$ as follows,
		$$\vec{\z} = (z_{11}, z_{12}, \cdots, z_{1p_2}, z_{21}, \cdots z_{2p_2}, \cdots, z_{p_11},\cdots, z_{p_1p_2})^{\intercal}.$$
		We also repeat the $(VMu)(VMu)^{\intercal}$ block for $p_1$ times and introduce
		$$\vec{\D}= \begin{bmatrix}
		(VMu)(VMu)^{\intercal} & &\\
		& \ddots & \\
		& & (Mu)(Mu)^{\intercal}
		\end{bmatrix}\in \mathbb{R}^{(p_1p_2)\times (p_1p_2)}. $$
		It is obvious that $(VMu)^{\intercal}(Z^{\intercal}Z-I_{p_1}) (VMu) = \vec{\z}^{\intercal}\vec{\D}\vec{\z} - E\vec{\z}^{\intercal}\vec{\D}\vec{\z}$. Besides,
		$$\|\vec{\D}\| = \|(VMu)(VMu)^{\intercal}\| = \|Mu\|^2_2 \leq \|M\|^2\|u\|_2^2 = (\sigma_r^2(X) + p_1)^{-1}, $$
		$$\|\vec{\D}\|_F^2 = p_1\|(VMu)(VMu)^{\intercal}\|_F^2 = p_1\|Mu\|_2^4 \leq p_1(\sigma_r^2(X) + p_1)^{-2}. $$
		By Hanson-Wright Inequality (Theorem 1 in \cite{rudelson2013hanson}), 
		\begin{equation*}
		\begin{split}
		& \P\left\{\left|(VMu)^{\intercal}(Z^{\intercal}Z - p_1I_{p_2})(VMn)\right| > x\right\}\\
		 = &  \P\left\{\left|\vec{\z}^{\intercal}\vec{\D}\vec{\z} - E\vec{\z}^{\intercal}\vec{\D}\vec{\z}\right| > x\right\}\\
		\leq & 2\exp\left(-c\min\left(\frac{x^2(\sigma^2_r(X)+p_1)^2}{p_1}, x(\sigma^2_r(X) + p_1)\right)\right),
		\end{split}
		\end{equation*}
		where $c$ only depends on $\tau$.
		\item Next, we bound 
		\begin{equation*}
		\begin{split}
		& (XVMu)^{\intercal} Z (VMu) = \tr(Z (VMu) (XVMu)^{\intercal})\\
		= & \vec{\z}^{\intercal}{\rm vec}(VMu(XVMu)^{\intercal}),
		\end{split}
		\end{equation*}
		where $\vec{\z}$, ${\rm vec}(VMu(XVMu)^{\intercal})$ are the vectorized $Z$, $(VMu(XVMu)^{\intercal})$. Since 
		$$XVM = U\begin{bmatrix}
		\sigma_1(X)(\sigma_1^2(X) + p_1)^{-{1\over 2}} & & \\
		& \ddots & \\
		& & \sigma_r(X)(\sigma_r^2(X) + p_1)^{-{1\over 2}}
		\end{bmatrix}, $$
		we know $\|XVM\|\leq 1$, and
		\begin{equation*}
		\begin{split}
		& \|{\rm vec}(VMu(XVMu)^{\intercal})\|_2^2 = \|(MVu)(XVMu)^{\intercal}\|_F^2\\
		= &  \|Mu\|_2^2\cdot \|XVMu\|_2^2 \leq \|M\|^2 \leq (\sigma_r^2(X) + p)^{-1}.
		\end{split}
		\end{equation*}
		By the basic property of i.i.d. sub-Gaussian random variables, we have
		$$\P\left(\left|(XVMu)^{\intercal}Z (VMu)\right| > x \right) \leq C\exp\left(-c x^2(\sigma_r^2(X) + p_1) \right), $$
		where $C, c$ only depends on $\tau$.
	\end{itemize}
	The two bullet points above and \eqref{eq:YVT_expansion} implies for any fixed unit vector $u\in \mathbb{R}^r$,
	\begin{equation}
	\begin{split}
	& \P\left(\left|u^{\intercal}M^{\intercal}Y_1^{\intercal}Y_1Mu - u^{\intercal}I_ru\right| > x \right)\leq C\exp\left(-c\left(\sigma_r^2(X) + p_1\right)x^2\wedge x\right),
	\end{split} 
	\end{equation}
	for all $x>0$. Here the $C, c$ above only depends on $\tau$. Next, the $\varepsilon$-net argument (Lemma \ref{lm:epsilon_net}) leads to
	\begin{equation}
	\P\left(\left\|M^{\intercal}Y_1^{\intercal}Y_1M - I_r\right\| > x\right) \leq C\exp\left(Cr-c\left(\sigma_r^2(X)+p_1\right)x^2\wedge x\right).
	\end{equation}
	In other words, \eqref{ineq:MVYYVM_probability} holds, which implies \eqref{ineq:sigma(YV)}.
	
	\item In order to prove \eqref{ineq:sigma(Y)}, we use the following fact about best rank-$r$ approximation of $Y$,
	$$\sigma_r(Y) = \max_{\rank(B)\leq r} \|Y - B\| \geq \|Y - Y \cdot [V ~ 0]\| = \sigma_{\max}(YV_{\perp}).$$
	to switch our focus from $\sigma_r(Y)$ to $\sigma_{\max}(YV_{\perp})$. Next, 
	\begin{equation*}
	\begin{split}
	& \sigma_{\max}^2(YV_{\perp}) = \sigma_{\max}(V_{\perp}^{\intercal}Y^{\intercal}YV_{\perp})\leq p_1 + \left\|V_{\perp}^{\intercal} Y^{\intercal} Y V_{\perp} - EV_{\perp}^{\intercal} Y^{\intercal} Y V_{\perp}\right\|.
	\end{split}
	\end{equation*}
	Note that $EV_{\perp}^{\intercal} Y^{\intercal}YV_{\perp} = p_1 I_{p_2-r}$, based on essentially the same procedure as the proof for \eqref{ineq:sigma(YV)}, one can show that
	\begin{equation*}
	\begin{split}
	\P\left(\left\|p_1^{-1}V_{\perp}Y^{\intercal}YV_{\perp} - Ep_1^{-1}V_{\perp}Y^{\intercal}YV_{\perp}\right\| \geq x\right) \leq C\exp\left\{Cp_2 - cp_1\left(x^2\wedge x\right)\right\}.
	\end{split}
	\end{equation*}
	Then we obtain \eqref{ineq:sigma(Y)} by combining the two inequalities above.
	
	\item Finally, we consider $\|\bbP_{YV}YV_\perp\|$. Since
	\begin{equation}\label{ineq:lm_ineq3_part1}
	\begin{split}
	\|\bbP_{YV} YV_{\perp}\| = & \|\bbP_{YVM}YV_{\perp}\|\\
	= & \|(YVM)((YVM)^{\intercal}(YVM))^{-1}(YVM)^{\intercal}YV_{\perp}\|\\
	\overset{\eqref{eq:singular_projection_inverse}}{\leq} & \sigma_{\min}^{-1}(YVM) \|M^{\intercal}V^{\intercal}Y^{\intercal}Y V_{\perp}\|
	\end{split}
	\end{equation}
	We analyze $\sigma_{\min}(YVM)$ and $\|M^{\intercal}V^{\intercal}Y^{\intercal} V_{\perp}\|$ separately below.
	
	Since
	$$\sigma_{\min}^2(YVM) = \sigma_{\min}(M^{\intercal}V^{\intercal}Y^{\intercal}YVM) \geq 1 - \|M^{\intercal}V^{\intercal}Y^{\intercal}YVM - I_r \|,$$
	by \eqref{ineq:MVYYVM_probability}, we know there exists $C, c$ only depending on $\tau$ such that
	$$\P\left(\sigma_{\min}^2(YVM) \geq 1 - x \right) \geq 1 - \exp\left\{Cr - c(\sigma_r^2(X)+p_1)x\wedge x^2\right\}. $$
	Set $x= 1/2$, we could choose $C_{\rm gap}$ large enough, but only depends on $\tau$, such that whenever $\sigma_r^2(X) \geq C_{\rm gap}p_2 \geq C_{\rm gap}r$, $Cr - c(\sigma_r^2(X)+p_1)x\wedge x^2 \leq - \frac{c}{8}(\sigma_r^2(X)+p_1)$.
	Under such setting,
	\begin{equation}\label{ineq:lm_ineq3_part2}
	\P\left(\sigma_{\min}^2(YVM) \geq \frac{1}{2}\right) \geq 1 - C\exp\left\{-c(\sigma_r^2(X) + p_1)\right\}. 
	\end{equation}
	
	For $\|M^{\intercal}V^{\intercal}Y^{\intercal}YV_{\perp}\|$, since $XV_{\perp} = 0$, we have the following decomposition,
	\begin{equation*}
	\begin{split}
	& M^{\intercal}V^{\intercal}Y^{\intercal}Y V_{\perp} = M^{\intercal} V^{\intercal}(X + Z)^{\intercal} (X+Z)V_{\perp}\\
	= & M^{\intercal}V^{\intercal}X^{\intercal}Z V_{\perp} + M^{\intercal}V^{\intercal}Z^{\intercal}Z V_{\perp}. 
	\end{split}
	\end{equation*}
	Follow the similar idea of the proof for \eqref{ineq:sigma(YV)}, we can show for any unit vectors $u\in \mathbb{R}^{r}, v\in \mathbb{R}^{p_2-r}$, 
	\begin{equation*}
	\begin{split}
		& \P\left(\left|u^{\intercal}M^{\intercal}V^{\intercal}X^{\intercal}ZV_{\perp} v\right| \geq x \right) \leq C\exp\left(-cx^2/\|(V_{\perp}v)(u^{\intercal}M^{\intercal}V^{\intercal}X^{\intercal})\|_F^2\right)\\
		\leq & C\exp(-cx^2). 
	\end{split}
	\end{equation*}
	\begin{equation*}
	\begin{split}
	& \P\left(\left|u^{\intercal}M^{\intercal}V^{\intercal} Z^{\intercal}Z V_{\perp} v\right| \geq x\right) = \P\left(\left|u^{\intercal}M^{\intercal}V^{\intercal} (Z^{\intercal}Z - p_1I_{p_2}) V_{\perp} v\right| \geq x\right)\\
	\leq & C\exp(-c\min(x^2, \sqrt{\sigma_r^2 + p_1}x)). 
	\end{split}
	\end{equation*}
	By the $\epsilon$-net argument again (Lemma \ref{lm:denoising}), we have
	\begin{equation}\label{ineq:lm_ineq3_part3}
	\P\left(\left\|M^{\intercal}V^{\intercal}Y^{\intercal}YV_{\perp}\right\| \geq x\right) \leq C\exp\left\{Cp_2 - c\min(x^2, \sqrt{\sigma_r^2+p_1}x)\right\}. 
	\end{equation}
	Combining \eqref{ineq:lm_ineq3_part1}, \eqref{ineq:lm_ineq3_part2} and \eqref{ineq:lm_ineq3_part3}, we obtain \eqref{ineq:P_YV(YV_perp)}.
\end{enumerate}
\end{proof}

\ \par 

\begin{proof}[Proof of Lemma \ref{lm:epsilon_net}] 
	First, based on Lemma 2.5 in \cite{vershynin2011spectral}, there exists $\varepsilon$-nets $W_L$ in $\mathbb{B}^{p_1}$, $W_R$ in $\mathbb{B}^{p_2}$, namely for any $u\in \mathbb{B}^{p_1}$, $v\in \mathbb{B}^{p_2}$, there exists $u_0\in W_L$, $v_0\in W_R$ such that $\|u_0 - u\|_2 \leq \varepsilon$, $\|v_0 - v\|_2\leq \varepsilon$ and $|W_L| \leq (1+2/\varepsilon)^{p_1}, |W_R| \leq (1+2/\varepsilon)^{p_2}$. Especially we choose $\varepsilon = 1/3$. Under the event that
$$Q = \left\{\left| u^{\intercal}Kv\right| \geq t, \forall u\in W_L, v\in W_R\right\}, $$
denote $(u^\ast, v^\ast) = \argmax_{u\in \mathbb{B}^{p_1}, v\in \mathbb{B}^{p_2}} \left|u^{\intercal} K v\right|$, $\alpha = \max_{u\in \mathbb{B}^{p_1}, v\in \mathbb{B}^{p_2}} \left|u^{\intercal} K v\right|$, then $\alpha = \|K\|$. According to the definition of $\frac{1}{3}$-net, there exists $u_0^\ast \in W_L, v_0^\ast \in W_R$ such that $\|u^\ast - u_0^\ast\|_2 \leq 1/3, \|v^\ast - v_0^\ast\|_2 \leq 1/3$. Then, 
\begin{equation*}
\begin{split}
\alpha = & \left|(u^\ast)^{\intercal}K v^\ast\right| \leq \left|(u^\ast_0)^{\intercal}Kv^\ast_0\right| + \left|(u^\ast - u_0)^{\intercal} K v_0^\ast\right| + \left|(u^\ast)^{\intercal}K(v^\ast - v_0)\right| \\
\leq & t + \|u^\ast - u_0\|_2\cdot \|K\| \cdot \|v_0^\ast\|_2 + \|u^\ast\|_2\cdot \|K\| \cdot \|v^\ast - v_0\|_2\leq t + \frac{2}{3} \alpha 
\end{split}
\end{equation*}
Thus, $\alpha \leq 3t$ when $Q$ happens. Finally, since
$$\P(Q) \leq \sum_{u\in W_L}\sum_{v\in W_R} \P\left(\left|u^{\intercal}Kv\right| \geq t\right) \leq 7^{p_1}\cdot 7^{p_2} \cdot \max_{u\in \mathbb{B}^n} \P\left(\left|u^{\intercal}Ku\right| \geq t\right), $$
which finished the proof of this lemma.
\end{proof}

\ \par

\begin{proof}[Proof of Lemma \ref{lm:lower_bound_transformation}] For each representative $P_{\phi} \in co(P_{\phi})$, we suppose $m^{(\phi)}$ is a measure on $\mathcal{P}_{\phi}$ such that
\begin{equation}\label{eq:P_phi}
P_{\phi} = \int_{\mathcal{P}_{\phi}} P_{\theta} dm^{(\phi)}. 
\end{equation}
Thus,
\begin{equation*}
\begin{split}
& \sup_{\theta\in \Theta} \E_{P_{\theta}} \left[d^2(T(\theta), \hat{\phi})\right] = \sup_{\phi \in \Phi} \sup_{\theta \in \mathcal{P}_\phi}\\
= & \E_{P_\theta}\left[d^2(\phi, \hat{\phi})\right] \geq \sup_{\phi\in \Phi} \int_{\mathcal{P}_\phi} \E_{P_\theta} \left[d^2(\phi, \hat\phi)\right]dm^{(\phi)}\\
= & \sup_{\phi\in \Phi} \int_{\mathcal{P}_\phi} \int_{\mathbb{R}^p}\left[d^2(\phi, \hat{\phi})\right]dP_{\theta} dm^{(\phi)} = \sup_{\phi\in \Phi} \int_{\mathbb{R}^p}\left[d^2(\phi, \hat{\phi})\right]\int_{\mathcal{P}_\phi}dP_{\theta}  dm^{(\phi)} \\
\overset{\eqref{eq:P_phi}}{=} & \sup_{\phi\in \Phi} \int_{\mathbb{R}^p}\left[d^2(\phi, \hat{\phi})\right]dP_{\phi} = \sup_{\phi \in \Phi} \E_{P_{\phi}}[d^2(\phi, \hat{\phi})].
\end{split}
\end{equation*}
\end{proof}

\ \par

\begin{proof}[Proof of Lemma \ref{lm:D_KL_upper}]
The direct calculation for $D(\bar{P}_{V, t}|| \bar{P}_{V', t})$ is relatively difficult, thus we detour by introducing the similar density to $\bar{P}_{V, t}$ as follows,
\begin{equation}\label{eq:tilde P_Vt}
\begin{split}
\tilde{P}_{V, t}(Y) = \int_{\mathbb{R}^{p_2\times r}} & \frac{1}{(2\pi)^{p_1p_2/2}}\exp(-\|Y - 2t WV^{\intercal}\|_F^2/2 )\\
& \cdot (\frac{p_1}{2\pi})^{p_1r/2}\exp(-p_1\|W\|_F^2/2) dW.
\end{split}
\end{equation}
We can see $\tilde{P}_{V,t}$ is another mixture of Gaussian distributions, thus it is indeed a density which sums up to 1. Since $V\in \mathbb{O}_{n, r}$, 
\begin{equation}\label{eq:VTV}
V^{\intercal}V = I_r,
\end{equation}
Denote $Y_{i\cdot}$ as the $i$-th row of $Y$. Note that $\tilde{P}_{V, t}$ can be simplified as
\begin{equation}\label{eq:tilde P_Vt expansion}
\begin{split}
\tilde{P}_{V, t} (Y)
\overset{\eqref{eq:VTV}}{=} & \frac{p_1^{p_1r/2}}{(2\pi)^{p_1(p_2+r)/2}} \int_{\mathbb{R}^{n\times r}} \exp\Big(-\frac{1}{2}\tr\Big\{YY^{\intercal} - 2tYVW^{\intercal} - 2tWV^{\intercal}Y^{\intercal}\\
& \quad\quad\quad\quad\quad + 4t^2 WW^{\intercal} + p_1WW^{\intercal} \Big\}\Big)dW\\
= & \frac{p_1^{p_1r/2}}{(2\pi)^{p_1(p_2+r)/2}} \int_{\mathbb{R}^{n\times r}} \exp\Big(-\frac{1}{2}\tr\Big\{Y(I_{p_2} - \frac{4t^2}{4t^2+p_1}VV^{\intercal})Y^{\intercal}\\
& \quad + (4t^2 + p_1)(W - \frac{2t}{4t^2+p_1}YV)(W - \frac{2t}{4t^2+p_1}YV)^{\intercal} \Big\}\Big)dW.\\
= & \frac{\left(p_1/(4t^2+p_1)\right)^{p_1r/2}}{(2\pi)^{p_1p_2/2}} \exp\left(-\frac{1}{2}\sum_{i=1}^{p_1} Y_{i\cdot} (I-\frac{4t^2}{4t^2+p_1}VV^{\intercal}) Y_{i\cdot}^{\intercal}\right).
\end{split}
\end{equation}
From the calculation above we can see $\tilde{P}_{V, t}$ is actually joint normal, i.e. when $Y\sim \tilde{P}_{V, t}$, 
$$Y_{i\cdot} \overset{iid}{\sim} N\left(0, \left(I_{p_2} - \frac{4t^2}{4t^2+p_1} VV^{\intercal}\right)^{-1}\right) = N\left(0, I_{p_2} + \frac{4t^2}{p_1}VV^{\intercal}\right), \quad i=1\cdots, p_1. $$
It is widely known that the KL-divergence between two $p$-dimensional multivariate Gaussians is
\begin{equation}\label{eq:KL_Gaussian}
\begin{split}
& D(N(\mu_0, \Sigma_0)|| N(\mu_1, \Sigma_1))\\
= & \frac{1}{2}\left(\tr\left(\Sigma_0^{-1}\Sigma_1\right) + (\mu_1-\mu_0)^{\intercal}\Sigma_1^{-1}(\mu_1 - \mu_0) - p + \log\left(\frac{\det \Sigma_1}{\det \Sigma_0}\right)\right). 
\end{split}
\end{equation}
We can calculate that for any two $V, V' \in \mathbb{O}_{p_1, r}$,
\begin{equation*}
\begin{split}
& D(\tilde{P}_{V, t}||\tilde{P}_{V', t})\\
 = & \frac{p_1}{2}\left\{\tr\left(\left(I_{p_2} - \frac{4t^2}{4t^2+p_1}VV^{\intercal}\right)\left(I_{p_2} + \frac{4t^2}{p_1}V'(V')^{\intercal}\right)\right) - p_2\right\}\\
= & \frac{p_1}{2}\Big( - \frac{4t^2}{4t^2+p_1}\tr(VV^{\intercal}) + \frac{4t^2}{p_1}\tr(V'(V')^{\intercal})\\
&  - \frac{16t^4}{p_1(4t^2+p_1)}\tr(VV^{\intercal}(V')(V')^{\intercal})\Big)\\
\overset{\eqref{eq:VTV}}{=} & \frac{16t^4}{2(4t^2+p_1)} \left(r - \|V^{\intercal}V' \|_F^2\right)\\
\overset{\text{Lemma \ref{lm:sin_Theta_distance}}}{=} & \frac{16t^4}{2(4t^2+p_1)} \|\sin\Theta(V, V')\|_F^2
\end{split}
\end{equation*}

Next, we show that $\bar P_{V, t}(Y)$ and $\tilde{P}_{V,t}(Y)$ are very close in terms of calculating KL-divergence. To be specific, we show when $C_r$ is large enough but with a uniform choice, there exists a uniform constant $c$ such that
\begin{equation}\label{ineq:barP/tildeP}
1 - 2\exp(-cp_1) \leq \frac{\bar{P}_{V, t}(Y)}{\tilde{P}_{V, t}(Y)} \leq 1 + 2\exp(-cp_1),\quad \forall Y\in \mathbb{R}^{p_1 \times p_2}.
\end{equation}
According to \eqref{eq:bar P_Vt} and \eqref{eq:tilde P_Vt expansion}, we know for any fixed $Y$,
\begin{equation}\label{eq:barP/tildeP}
\begin{split}
& \frac{\bar{P}_{V, t}(Y)}{\tilde{P}_{V, t}(Y)}\\
 = & C_{V, t} \int_{\sigma_{\min}(W)\geq \frac{1}{2}} \exp\left(- \tr(YY^{\intercal} - 2tYVW^{\intercal} -2tWV^{\intercal} Y^{\intercal} + (4t^2+p_1)WW^{\intercal})/2\right)\\
& \cdot \exp\left(\tr(Y(I-\frac{4t^2}{4t^2 + p_1}VV^{\intercal})Y^{\intercal})/2\right)\cdot \left(\frac{4t^2+p_1}{2\pi}\right)^{p_1r/2} dW\\
= & C_{V, t}\int_{\sigma_{\min}(W) \geq \frac{1}{2}} \left(\frac{4t^2+p_1}{2\pi}\right)^{p_1r/2} \exp\left(-(4t^2+p_1) \left\|W - \frac{2t}{4t^2+p_1} YV\right\|_F^2/2\right) dW\\
= & C_{V, t} \P\left(\sigma_{\min}(\tilde{W}) \geq \frac{1}{2}\Bigg| \tilde{W}\in \mathbb{R}^{p_1\times r}, \tilde{W} \sim N\left(\frac{2t}{4t^2+p_1} YV, \frac{I_{(p_1r)}}{4t^2+p_1}\right)\right).
\end{split}
\end{equation}
For fixed $Y\in \mathbb{R}^{p_1\times p_2}$, $YV\in \mathbb{R}^{p_1\times r}$, we can find $Q\in \mathbb{O}_{p_1, p_1-r}$ which is orthogonal to $YV$, i.e. $Q^{\intercal}YV = 0$. Then $Q^{\intercal}\tilde{W} \in \mathbb{R}^{(p_1-r)\times r}$ and $Q^{\intercal}\tilde{W}$ are i.i.d. normal distributed with mean 0 and variance $1 / (4t^2+p_1)$. By standard result in random matrix (e.g. Corollary 5.35 in \cite{vershynin2012introduction}), we have
\begin{equation*}
\begin{split}
& \sigma_{\min}(\tilde{W}) = \sigma_r(\tilde{W}) \geq \sigma_{r}(Q^{\intercal}\tilde{W}) = \frac{1}{\sqrt{4t^2+p_1}}\sigma_r\left(\sqrt{4t^2+p_1}Q^{\intercal}\tilde{W}\right) \\
\geq & \frac{1}{\sqrt{4t^2+p_1}} \left(\sqrt{p_1-r}-\sqrt{r} - x\right)
\end{split}
\end{equation*}
with probability at least $1 - 2\exp(-x^2/2)$. Since $t^2 \leq p_1/4$, $p_1 \geq 16r$, we further set $x = 0.078\sqrt{p_1}$, the inequality above further yields
$$\sigma_{\min}(\tilde{W}) \geq \frac{1}{\sqrt{4t^2+p_1}}(\sqrt{p_1-r}-\sqrt{r} - x) \geq \frac{\sqrt{15p_1/16} - \sqrt{p_1/16} - 0.078\sqrt{p_1}}{\sqrt{2p_1}} \geq \frac{1}{2} $$
with probability at least $1 - \exp(-cp_1)$. 

%
Thus, for fixed $Y$, $p_1\geq 16r$, 
\begin{equation}\label{ineq:tilde_W}
\P\left(\sigma_{\min}(\tilde{W}) \geq \frac{1}{2}\Bigg| \tilde{W}\in \mathbb{R}^{p_1\times r}, \tilde{W} \sim N\left(\frac{2t}{4t^2+p_1} YV, \frac{I_{(p_1r)}}{4t^2+p_1}\right)\right) \geq 1  -\exp(-cp_1).
\end{equation}

Recall the definition of $C_{V, t}$, we have
$$C_{V, t}^{-1} = \P\left(\sigma_{\min}(W) \geq 1/2\Big| W\in \mathbb{R}^{p_1\times r}, W\overset{iid}{\sim}N(0, 1/p_1)\right).$$ 
Also recall the assumption that $r\leq 16 p_1$, Corollary 5.35 in \cite{vershynin2012introduction} yields
$$\P(\sigma_{\min}(W) < 1/2 ) \leq \exp(-cp_1). $$
Thus, 
\begin{equation}\label{ineq:C_Vt}
1< C_{V, t} < 1+2\exp(-cp_1)
\end{equation}
Combining \eqref{eq:barP/tildeP}, \eqref{ineq:C_Vt} and \eqref{ineq:tilde_W} we have proved \eqref{ineq:barP/tildeP}. 

Finally, we can bound the KL divergence for $\bar{P}_{V, t}$ and $\bar{P}_{V', t}$ based on the previous steps.
\begin{equation*}
\begin{split}
& D(\bar{P}_{V, t}||\bar{P}_{V', t})\\
= & \int_{Y\in \mathbb{R}^{p_1\times p_2}} \bar{P}_{V, t}(Y) \log\left(\frac{\bar P_{V, t}(Y)}{\bar P_{V', t}(Y)}\right) dY\\
= & \int_{Y\in \mathbb{R}^{p_1\times p_2}} \bar{P}_{V, t}(Y) \left[\log\left(\frac{\bar P_{V, t}(Y)}{\tilde P_{V, t}(Y)}\right) + \log\left(\frac{\tilde P_{V, t}(Y)}{\tilde P_{V', t}(Y)}\right) + \log\left(\frac{\tilde P_{V', t}(Y)}{\bar P_{V', t}(Y)}\right)\right] dY\\
\leq & \int_{Y}\bar P_{V, t}(Y) \left(\log\left(\frac{\tilde{P}_{V, t}(Y)}{\tilde{P}_{V', t}(Y)}\right) + 2\log(1 + \exp(-cp_1))\right)dY\\
\leq & C\exp(-cp_1) + \int_{Y} \tilde{P}_{V, t}(Y) \log\left(\frac{\tilde{P}_{V,t}(Y)}{\tilde{P}_{V', t}(Y)}\right)dY\\
& + \int_Y \left|(\tilde{P}_{V, t}(Y) - \bar {P}_{V, t}(Y))\log\left(\frac{\tilde{P}_{V,t}(Y)}{\tilde{P}_{V', t}(Y)}\right)\right|dY\\
\leq & \frac{16t^4}{2(4t^2 + p_1)}\|\sin\Theta(V, V')\|_F^2 + C\exp(-cp_1)\\
& + \exp(-cp_1) \int_{Y} \left|\tilde{P}_{V, t}(Y)\log\left(\frac{\tilde{P}_{V,t}(Y)}{\tilde{P}_{V', t}(Y)}\right)\right|dY
\end{split}
\end{equation*}
For the last term in the formula above, we can calculate accordingly as
\begin{equation*}
\begin{split}
& \int_{Y} \left|\tilde{P}_{V, t}(Y)\log\left(\frac{\tilde{P}_{V,t}(Y)}{\tilde{P}_{V', t}(Y)}\right)\right|dY\\
= & \int_{Y} \tilde{P}_{V, t}(Y) \cdot\left|\frac{4t^2}{4t^2+p_1}\sum_{i=1}^p Y_{i\cdot} \left(VV^{\intercal} - V'(V')^{\intercal}\right)Y_{i\cdot}^{\intercal}/2\right|dY\\
\leq & E\left( \frac{4t^2}{4t^2+p_1}\sum_{i=1}^{p_1} \frac{1}{2}Y_{i\cdot}\left(VV^{\intercal} + V'(V')^{\intercal}\right)Y_{i\cdot}^{\intercal} ~\Bigg| ~ Y_{i\cdot} \sim N\left(0, \left(I_{p_2} + \frac{4t^2}{p_1}VV^{\intercal}\right)\right)\right)\\
= & p_1\cdot \frac{4t^2}{2(4t^2 + p_1)} \tr\left(\left(VV^{\intercal} + V'(V')^{\intercal}\right)\cdot \left(I_{p_2} + \frac{4t^2}{p_1}VV^{\intercal}\right)\right) \\
\leq & p_1\cdot \frac{4t^2}{2(4t^2 + p_1)} \tr\left(2V^{\intercal} \left(I_{p_2} + \frac{4t^2}{p_1} VV^{\intercal}\right)V\right)\\
= & 4t^2r \leq p_1^2.
\end{split}
\end{equation*}
Since $\exp(-cp_1)p_1^2$ is upper bounded by some constant for all $p_1 \geq 1$. To sum up there exists uniform constant $C_{KL}$ such that for all $V, V' \in \mathbb{O}_{p_2, r}$,
\begin{equation*}
D(\bar P_{V, t}|| \bar P_{V', t}) \leq \frac{16t^4}{2(4t^2+p_1)}\|\sin\Theta(V, V')\|_F^2 + C_{KL},
\end{equation*}
which has finished the proof of this lemma.
\end{proof}

\ \par

\begin{proof}[Proof of Lemma \ref{lm:CCA}] 
	First of all, since left and right multiplication for $G_{[1:r, :]}$ does not change the essence of the problem, without loss of generality we can assume that 
$$U_G = \begin{bmatrix}
1 & & \\
& \ddots & \\
& & 1 \\
0 & \cdots & 0
\end{bmatrix} \in  \mathbb{O}_{d, r}.$$ 
In such case, all non-zero entries of $G$ are zero except the top left $r\times r$ block. 

Based on the random matrix theory (Lemma \ref{lm:denoising}), we know 
\begin{equation}\label{ineq:lm_L_1:r}
\P\left(\sigma_{\min}^2(L_{[1:r, :]}) \geq t^2 + p - C\left((dp)^{1\over 2} + d\right)\right) \geq 1 - C\exp(-cd),
\end{equation}
\begin{equation}\label{ineq:lm_L_r+1:n}
\P\left(\left\|L_{[(r+1):n, :]}\right\|^2 \leq n + C\left((pn)^{1\over 2} + p\right) \right) \geq 1 - C\exp(-cp),
\end{equation}
\begin{equation}\label{ineq:lm_L_r+1:p}
\P\left(\left\|L_{[(r+1): d, :]}\right\|^2 \leq Cd \right) \geq 1 - C\exp(-cd).
\end{equation}

\begin{enumerate}
	\item First we consider $\sigma_{\min}(\hat U_{[1:r, :]})$. By Lemma \ref{lm:U_svd} and \eqref{ineq:lm_L_1:r}, \eqref{ineq:lm_L_r+1:n}, we know with probability at least $ 1- C\exp(-cd\wedge p)$.
	\begin{equation*}
	\begin{split}
	\sigma_{\min}^2 (\hat U_{[1:r, :]}) \geq & \frac{\sigma_{\min}^2(L_{[1:r, :]})}{\sigma_{\min}^2(L_{[1:r, :]}) + \sigma_{\max}^2(L_{[(r+1):p, :]})}\\
	\geq & \frac{t^2 + p - C\left((dp)^{1\over 2} + d\right)}{t^2 + p - C\left((dp)^{1\over 2} + d\right) + n + C\left((pn)^{1\over 2} + p\right)}.\\
	\end{split}
	\end{equation*}
	We target on showing under the statement given in the lemma,
	\begin{equation}\label{eq:lm_t^2_geq_p+t^2}
	\frac{t^2 + p - C\left((dp)^{1\over 2} + d\right)}{t^2 + p - C\left((dp)^{1\over 2} + d\right) + n + C\left((pn)^{1\over 2} + p\right)} \geq \frac{p+t^2/2}{t^2+n}. 
	\end{equation}
	The inequality above is implied by
	\begin{equation*}
	\begin{split}
	& \left(t^2+p - C((dp)^{1\over 2} + d)\right)(t^2+n) \geq \left(t^2+n+C\left((pn)^{1\over 2} + p\right)\right)\left(p + t^2/2\right)\\
	\Leftarrow & t^2\left(\frac{n+t^2}{2} - C\left((pn)^{1\over 2} +  p + (dp)^{1\over 2} + d\right)\right)\\
	 &  - C\left(p^{3/2}n^{1\over 2} + p^2 + (dp)^{1\over 2}n + dn\right) \geq 0.
	\end{split}
	\end{equation*}
	In fact, whenever $t^2 > C_{\rm gap}\left((dp)^{1\over 2} + d + p^{3/2}n^{-{1\over 2}}\right)$, $n > C_0 p$, we have
	\begin{equation*}
	\begin{split}
	& t^2\left(\frac{n+t^2}{2} - C\left((pn)^{1\over 2} +  p + (dp)^{1\over 2} + d\right)\right)\\
	& - C\left(p^{3/2}n^{1\over 2} + p^2 + (dp)^{1\over 2}n + dn\right)\\ \geq & t^2\left((C_{\rm gap}/2 - C)\left((dp)^{1\over 2} + d\right) + n\left(\frac{1}{2} - \frac{C}{C_0} - \frac{C}{\sqrt{C_0}}\right) \right) - \frac{C}{C_{\rm gap}}\left(n t^2\right)
	\end{split}
	\end{equation*}
	From the inequality above we can see, there exists large but uniform choices of constants $C_0, C_{\rm gap}>0$ such that the term above is non-negative, additionally implies \eqref{eq:lm_t^2_geq_p+t^2}. In another word, there exists $C_0, C_{\rm gap}>0$, such that whenever \eqref{ineq:lm_CCA_assumption} holds, then
	\begin{equation}\label{ineq:lm_CCA_factor1}
	\P\left( \sigma_{\min}^2\left(\hat{U}_{[1:r, :]}\right) \geq \frac{p+t^2/2}{n + t^2}\right) \geq 1 - C\exp(-cd\wedge p).
	\end{equation}
	
	\item Next we consider $\sigma_{\max}(\hat U_{[(r+1):p, :]})$. Suppose we randomly generate $\tilde{R}\in \mathbb{O}_{n-r}$ as a unitary matrix of $(n-r)$ dimension, which is independent of $L$. Also, $R\in \mathbb{O}_{n-r, p-r}$ as the first $p-r$ columns of $\tilde{R}$. Clearly, 
	$$\begin{bmatrix}
	I_r & \\
	& \tilde{R}^{\intercal}
	\end{bmatrix}\cdot L \quad \text{ and }\quad L \quad \text{have the same distribution.} $$
	This implies
	$$\begin{bmatrix}
	I_r & \\
	& \tilde{R}^{\intercal}
	\end{bmatrix}\cdot \hat{U}, \quad \text{ and }\quad \hat{U} \quad \text{have the same distribution.} $$	
	When we focus on the $[(r+1):d]$-th rows, we get	
	$$R^{\intercal} \hat{U}_{[(r+1):n, :]}\quad \text{and}\quad \hat{U}_{[(r+1):d, :]} \quad \text{have the same distribution.} $$
	Thus, we can turn to consider $R^{\intercal} \hat{U}_{[(r+1):n, :]}$ rather than $\hat{U}_{[(r+1):p, :]}$. Conditioning on first $r$ rows of $\hat{U}$, i.e. $\hat{U}_{[1:r, :]}$, the rest part of $\hat{U}$, i.e. $\hat{U}_{[(r+1):n, :]}$ is a random matrix with spectral norm no more than 1. Applying Lemma \ref{lm:random_projection}, we get for any given $\hat{U}_{[1:r, :]}$ we have the following conditioning probability when $n \geq C_0 d$ for some large $C_0$,
	\begin{equation*}
	\begin{split}
	\P\left(\left\|R^{\intercal} \hat{U}_{[(r+1):n, :]}\right\|^2 \geq \frac{p+C\sqrt{(d-r)p} + C(d-r)}{n-r}\Bigg| \hat{U}_{[1:r, :]}\right) \leq C\exp(-cp).
	\end{split}
	\end{equation*}
	Therefore,
	\begin{equation}
	\P\left(\left\|\hat{U}_{[(r+1):d, :]}\right\|^2 \geq \frac{p+C\sqrt{(d-r)p} + C(d-r)}{n-r}\right) \leq C\exp(-cp).
	\end{equation}
	Next, under essentially the same argument as the proof in Part 1, one can show that there exists $C_{\rm gap}, C_0>0$ such that whenever \eqref{ineq:lm_CCA_assumption} holds, we have
	\begin{equation*}
	\frac{p + C\sqrt{(d-r)p} + C(d-r)}{n-r} \leq \frac{p + t^2/4}{n+t^2},
	\end{equation*}
	additionally we also have
	\begin{equation}\label{ineq:lm_CCA_factor2}
	\P\left(\left\|\hat{U}_{[(r+1):d, :]}\right\|^2 \geq \frac{p+t^2/4}{n+t^2}\right) \leq C\exp(-cp).	
	\end{equation}
	
	\item In this step we consider $\|\hat U_{[(r+1):p, :]}P_{\hat U_{[1:r, :]}}\|$. The idea to proceed is similar to the part for $\|\hat{U}_{[(r+1):p, :]}\|$. Conditioning on $\hat{U}_{[1:r, :]}$, 
	$$\hat{U}_{[(r+1): p, :]}P_{\hat{U}_{[1:r, :]}}\quad \text{and} \quad R^{\intercal}\hat{U}_{[(r+1): n, :]}P_{\hat{U}_{[1:r, :]}} \quad \text{have the same distribution}.$$
	Also, $\|\hat{U}_{[r+1:n, :]}P_{\hat{U}_{[1:r, :]}} \|\leq 1$, $\rank(\hat{U}_{[r+1:n, :]}P_{\hat{U}_{[1:r, :]}}) \leq r$. By Lemma \ref{lm:random_projection}, there exists uniform constants $C, c>0$ such that
	\begin{equation}\label{ineq:lm_CCA_factor3}
	\begin{split}
	& \P\left(\left\|\hat{U}_{[(r+1): p, :]}P_{\hat{U}_{[1:r, :]}} \right\| > C(d/n)^{1\over 2}\right)\\
	= & \P\left(\left\|R^{\intercal}\hat{U}_{[(r+1): n, :]}P_{\hat{U}_{[1:r, :]}} \right\| > C(d/n)^{1\over 2}\right) \leq C\exp(-cd). 
	\end{split}
	\end{equation}
	\item Combine \eqref{ineq:lm_CCA_factor1}, \eqref{ineq:lm_CCA_factor2}, \eqref{ineq:lm_CCA_factor3}, we know there exists $C_{\rm gap}, C_0>0$ such that whenever \eqref{ineq:lm_CCA_assumption} holds, then with probability at least $1 - C\exp(p\wedge d)$,
	$$\sigma_{\min}^2(\hat U_{[1:r, :]}) >  \left\|\hat{U}_{[(r+1):d, :]}\right\|^2 \geq \sigma_{r+1}^2(\hat{U}_{[1:d, :]}). $$
	\begin{equation}
	\begin{split}
	& \frac{\sigma_{\min}(\hat{U}_{[1:r, :]})\cdot \left\|\hat{U}_{[(r+1): p, :]}P_{\hat{U}_{[1:r, :]}} \right\|}{\sigma_{\min}^2(\hat{U}_{[1:r, :]}) - \left\|\hat{U}_{[(r+1):d, :]}\right\|^2} \leq \frac{\sqrt{\frac{p+t^2/4}{n+t^2}} \cdot C(d/n)^{1\over 2}}{\frac{p+t^2/2}{n+t^2} - \frac{p+t^2/4}{n+t^2}}\\
	\leq & \frac{C\sqrt{d(p + t^2)(1 + t^2/n)}}{t^2}.
	\end{split}
	\end{equation} 
	By Proposition \ref{lm:perturbation}, we have finished the proof of Lemma \ref{lm:U_svd}.
\end{enumerate}
\end{proof}

\begin{proof}[Proof of Lemma \ref{lm:U_svd}.] Based on the setting, $U = LV\Sigma^{-1}$. Thus,
\begin{equation*}
\begin{split}
& \sigma_{\max}^2(U_{[\Omega, :]}) = \max_{v\in \mathbb{R}^p} \frac{\left\|U_{[\Omega, :]}v\right\|_2^2}{\|Uv\|_2^2} \leq \max_{v\in \mathbb{R}^p} \frac{\left\|L_{[\Omega, :]}V\Sigma^{-1}v\right\|_2^2}{\left\|L_{[\Omega, :]}V\Sigma^{-1}v\right\|_2^2 + \left\|L_{[\Omega^c, :]}V\Sigma^{-1}v\right\|_2^2}\\
\leq & \max_{v\in \mathbb{R}^p} \frac{\sigma_{\max}^2(L_{[\Omega, :]})\left\|V\Sigma^{-1}v\right\|_2^2}{\sigma_{\max}^2(L_{[\Omega, :]})\left\|V\Sigma^{-1}v\right\|_2^2 + \sigma_{\min}^2(L_{[\Omega^c, :]})\left\|V\Sigma^{-1}v\right\|_2^2}\\ 
= & \frac{\sigma_{\max}^2(L_{[\Omega, :]})}{\sigma_{\max}^2(L_{[\Omega, :]}) + \sigma_{\min}^2(L_{[\Omega^c, :]})}.
\end{split}
\end{equation*}
The other inequality in the lemma can be proved in the same way.
\end{proof}

\ \par

\begin{proof}[Proof of Lemma \ref{lm:random_projection}] Suppose $\alpha = \|X\|$. Since left and right multiply orthogonal matrix to $X$ does not essentially change the problem, without loss of generality we can assume that $X\in \mathbb{R}^{n\times m}$ is diagonal, such that $X = \diag(\sigma_1(X), \sigma_2(X),\cdots, \sigma_p(X), 0, \cdots)$. Clearly $\sigma_i(X) \leq \alpha$. Now
$$\|X^{\intercal}R\| = \left\|\begin{bmatrix}
\sigma_1(X) R_{11} & \cdots & \sigma_1(X) R_{1d}\\
\vdots & & \vdots \\
\sigma_p(X) R_{p1} & \cdots & \sigma_p(X) R_{pd}\\
\end{bmatrix}\right\| \leq \alpha \cdot \left\|R_{[1:p, :]}\right\|. $$
In order to finish the proof, we only need to bound the spectral norm for $\|R_{[1:p, :]}\|$. For any unit vector $v\in \mathbb{R}^{d}$, $Rv$ is randomly distributed on $\mathbb{O}_{n, 1}$ with Haar measure. Thus
\begin{equation}
\|R_{[1:p, :]}v\|_2^2 \text{ has the same distribution as } \frac{x_1^2+\cdots + x_p^2}{x_1^2+\cdots + x_n^2}, \quad x_1,\cdots, x_n\overset{iid}{\sim}N(0, 1).
\end{equation}
By the tail bound for $\chi^2$-distribution, 
$$\P\left(p - \sqrt{2pt} \leq \sum_{k=1}^p x_k^2 \leq p + \sqrt{2pt} + 2t\right) \leq 2\exp(-t),$$
$$ \P\left(n - \sqrt{2nt} \leq \sum_{k=1}^n x_k^2 \leq n + \sqrt{2nt} + 2t\right) \leq 2\exp(-t),$$
which means
\begin{equation*}
\begin{split}
2\exp(-t) \geq & \P\left(v^{\intercal}\left(R_{[1:p, :]}^{\intercal}R_{[1:p, :]} - \frac{p}{n}I_r\right)v \geq \frac{p + \sqrt{2pt} + 2t}{n - \sqrt{2nt}} - \frac{p}{n}\right),\\
\end{split}
\end{equation*}
\begin{equation*}
\begin{split}
2\exp(-t) \geq & \P\left(v^{\intercal}\left(R_{[1:p, :]}^{\intercal}R_{[1:p, :]} -\frac{p}{n} I_r\right)v \leq \frac{p - \sqrt{2pt}}{n + \sqrt{2nt} + 2t} - \frac{p}{n}\right).
\end{split}
\end{equation*}
We set $t = Cd$ for large enough $C>0$ and apply $\varepsilon$-net method (Lemma \ref{lm:epsilon_net}), the following result hold for true.
\begin{equation*}
\begin{split}
 C\exp(-cd) \geq \P\Big( & \left\|R^{\intercal}_{1:d. :}R_{1:d, :} - \frac{d}{n}I_r\right\|\\
 \geq & 3\max\left\{\frac{p + C\sqrt{pd} + Cd}{n - C\sqrt{nd}} - \frac{p}{n},  \frac{p}{n} - \frac{p - C\sqrt{pd}}{n + C\sqrt{nd} + Cd}\right\} \Big).
\end{split}
\end{equation*}
Note that
\begin{equation*}
\begin{split}
& \max\left\{\frac{p + C\sqrt{pd} + Cd}{n - C\sqrt{nd}} - \frac{p}{n},  \frac{p}{n} - \frac{p - C\sqrt{pd}}{n + C\sqrt{nd} + Cd}\right\}\\
\leq & \max\left\{\frac{C\left(\sqrt{pd} + d + p\sqrt{d/n}\right)}{n - C\sqrt{nd}},  \frac{C\left(p\sqrt{d/n} + pd/n + \sqrt{pd}\right)}{n + C\sqrt{nd} + Cd}\right\}.\\
\leq & \max\left\{\frac{C\left(2\sqrt{pd} + d\right)}{n - C\sqrt{nd}},  \frac{C\left(3\sqrt{pd}\right)}{n + C\sqrt{nd} + Cd}\right\}
\end{split}
\end{equation*}
Thus there exists $C_0 >0 $ such that when $n > C_0r$, $n - C\sqrt{nr} > n/2$, and additionally, 
$$\max\left\{\frac{d + C\sqrt{dr} + Cr}{n - C\sqrt{nr}} - \frac{d}{n},  \frac{d}{n} - \frac{d - C\sqrt{dr}}{n + C\sqrt{nr} + Cr}\right\} \leq \frac{C \left((pd)^{1\over 2} + d\right)}{n}. $$
To sum up, we have finished the proof of Lemma \ref{lm:random_projection}. 
\end{proof}

 \end{document}